\documentclass[a4paper,10pt]{article}
\usepackage[frenchb]{babel}
\usepackage[T1]{fontenc}
\usepackage[utf8]{inputenc}
\usepackage{amsmath}
\usepackage{array}
\usepackage{amsthm}
\usepackage{amssymb}
\input xy
\xyoption{all}

\usepackage{hyperref}

\newcommand{\A}{{\mathcal{A}}}
\newcommand{\B}{{\mathcal{B}}}
\newcommand{\C}{{\mathcal{C}}}
\newcommand{\D}{{\mathcal{D}}}
\newcommand{\F}{{\mathcal{F}}}
\newcommand{\G}{{\mathcal{G}}}
\newcommand{\fct}{\mathbf{Fct}}
\newcommand{\st}{\mathbf{St}}

\newcommand{\E}{{\mathcal{E}}}

\newcommand{\Pol}{{\mathcal{P}}ol}
\newcommand{\s}{{\mathcal{S}}}
\newcommand{\I}{{\mathcal{I}}}
\newcommand{\J}{{\mathcal{J}}}

\newcommand{\M}{{\mathcal{M}}}
\newcommand{\N}{{\mathcal{N}}}
\newcommand{\FF}{{\mathbb{F}_2}}
\newcommand{\Z}{{\mathbb{Z}}}

\newcommand{\col}{{\rm colim}\,}
\newcommand{\T}{{\mathcal{T}}}

\newcommand{\mn}{\M on_{\rm nul}}
\newcommand{\mi}{\M on_{\rm ini}}

\title{Foncteurs faiblement polynomiaux}
\author{Aur\'elien DJAMENT\thanks{CNRS, Laboratoire de mathématiques Jean Leray (UMR 6629), aurelien.djament@univ-nantes.fr}\; et Christine VESPA\thanks{Institut de Recherche Mathématique Avancée, université de Strasbourg, vespa@math.unistra.fr}}
\newtheorem*{thma}{Theorem}

\newtheorem{thi}{Th\'eor\`eme}

\newtheorem{thm}{Th\'eor\`eme}[section]
\newtheorem{pr}[thm]{Proposition}
\newtheorem{cor}[thm]{Corollaire}
\newtheorem{lm}[thm]{Lemme}

\newtheorem{conj}[thm]{Conjecture}

\theoremstyle{definition}
\newtheorem{defi}[thm]{D\'efinition}
\newtheorem{nota}[thm]{Notation}

\newtheorem{conv}[thm]{Convention}

\theoremstyle{remark}
\newtheorem{rem}[thm]{Remarque}
\newtheorem{ex}[thm]{Exemple}

\begin{document}

\maketitle

\begin{abstract}
Nous introduisons et étudions deux notions de foncteur polynomial depuis une petite catégorie monoïdale symétrique dont l'unité est objet initial vers une catégorie abélienne. Dans le cas où la catégorie source est une catégorie d'espaces hermitiens, nous classifions les foncteurs polynomiaux de degré $n$ modulo les foncteurs de degré $n-1$. Une motivation importante provient de l'homologie des groupes de congruence ou des sous-groupes $IA$ des automorphismes des groupes libres, dont l'étude fonctorielle s'insère naturellement dans ce formalisme.
\end{abstract}

\begin{small}
\begin{center}
 \textbf{Extended abstract}
\end{center}

The usual definition of polynomial functors due to Eilenberg and Mac Lane can be extended to functors on a monoidal category whose unit is  a null object. Several natural functors having polynomial properties are defined only on monoidal categories whose unit is only an initial object. For example the cohomology of configuration spaces defines  a functor on the category of finite sets with injections.

In this paper we introduce and study two notions of polynomial functors from a small symmetric monoidal category whose unit is an initial object to an abelian category. The naive generalization of the classical notion of polynomial functors gives rise to the notion of strong polynomial functors introduced in Definition \ref{polfor}. We give several examples of strong polynomial functors in section~\ref{ex-S(ab)} and study this notion. For example, in Proposition \ref{CEF-1.2} we translate results of \cite{CEF} in terms of strong polynomial functors. However strong polynomial functors have some bad properties. For example they are not stable by subobjects. 
This motivates to define another notion of polynomiality in this setting. For this we consider a quotient of the category of functors, introduced in Definitions~\ref{cat-stable1} and~\ref{cat-stable2}. In this category we introduce in Definition \ref{pol-st(M,A)} a notion of polynomial functors which mimics the classical notion, in the sense that when the source category has the unit as null object we recover the definition of polynomial functors introduced by Eilenberg and Mac Lane. We prove in Proposition \ref{propol-gal}, that these categories of polynomial functors $\Pol_n(\C,\A)$  are thick and stable under limits and colimits. 

The heart of this paper consists in classifying polynomial functors of degree $n$ modulo functors of degree $n-1$. For functors from the category of finite sets with injections the classification is given in Proposition \ref{eq-theta}.
Then we consider the functors on a category of hermitian objects on an additive category with duality. This notion generalises the categories of hermitian spaces by Example  \ref{ex-hermitien}. In Theorem \ref{thm-pal} we classify polynomial functors of degree $n$ modulo functors of degree $n-1$ on a category of hermitian objects. More precisely we obtain the following result :

\begin{thma}
Let $\C$ be a small additive category equipped with a duality functor $\C^{op}\to\C$ and $\mathbf{H}(\C)$ be the category of {\em hermitian objects} associated to this situation. We have an equivalence of categories
$$\Pol_n(\mathbf{H}(\C),\A)/\Pol_{n-1}(\mathbf{H}(\C),\A) \simeq \Pol_n(\C,\A)/\Pol_{n-1}(\C,\A) $$
for all $n\in\mathbb{N}$ and all Grothendieck category $\A$.
\end{thma}

Important motivations of this paper come from the homology of congruence groups and of subgroups $IA$ of the automorphisms of free groups, whose functorial study is part in a natural way of this formalism. These two examples are studied in section \ref{ex-S(ab)}.

This paper is organized as follows. In Section 1 we give the definitions of  strong polynomial functors and weak polynomial functors and give them properties. In Section 2 we show that the notion of strong polynomial functors can be described in terms of cross effects and prove that it extends the usual notion of polynomial functors introduced by Eilenberg and Mac Lane. Section 3 concerns the construction of a universal symmetric monoidal category where the unit is a null object from a symmetric monoidal category where the unit is an initial object. In Section 4 we study the polynomial functors on the category of finite sets with injections. In Section 5 we give examples of polynomial functors from the category of finitely generated free groups where the morphisms are monomorphisms with a given splitting. In the last section we prove our main result concerning polynomial functors on categories of hermitians spaces.
\end{small}

\medskip

\noindent
{\em Mots clefs} : foncteurs polynomiaux ; espaces hermitiens ; catégories abéliennes ; catégories quotients.

\smallskip

\noindent
{\em Classification MSC 2010} : 18E15, 18A25, 18A40, 20J06.

\section*{Introduction}

La notion de {\em foncteur polynomial} entre des catégories de modules remonte au travail fondamental \cite{EML} d'Eilenberg et Mac Lane, dans les années 1950, qui l'introduit à des fins de topologie algébrique, pour l'étude de l'homologie des espaces qui portent leurs noms. Depuis, de nombreuses recherches autour des foncteurs polynomiaux se sont développées. Plusieurs travaux concernent l'étude des foncteurs polynomiaux d'un point de vue intrinsèque \cite{Pira-rec, PJ, V-pol, HV}. D'autres travaux sont motivés par l'utilisation de ces foncteurs en topologie algébrique, qui a connu un renouveau au début des années 1990 grâce à l'article \cite{HLS} de Henn, Lannes et Schwartz reliant la catégorie des modules instables sur l'algèbre de Steenrod à ces foncteurs. Les foncteurs polynomiaux sont également très étudiés pour leurs liens avec la théorie des représentations \cite{K2}, la $K$-théorie algébrique et la (co)homologie des groupes  \cite[appendix]{FFSS} \cite{ Bet, Bet-sym, Sco, DV, Dja-JKT}. 

Même si les foncteurs polynomiaux entre différentes catégories de modules sont restés le cas le plus étudié, plusieurs de ces travaux utilisent des foncteurs polynomiaux dans d'autres contextes \cite{P-hodge, HV, Baues-Pira}.
Ces généralisations reposent toujours sur la notion d'{\em effets croisés} inaugurée par Eilenberg et Mac Lane. Le cadre le plus général dans lequel les effets croisés ont été considérés jusqu'à présent est celui d'une catégorie source monoïdale $(\M,\oplus,0)$ dont l'unité $0$ est \textbf{objet nul} --- cf. \cite[section~2]{HPV} et le §\,\ref{smn} du présent article pour le cas symétrique. 
Dans ce cadre, les effets croisés permettent de scinder naturellement l'évaluation d'un foncteur sur une somme, au sens de la structure monoïdale $\oplus$. Ce scindage repose sur l'observation que, $0$ étant à la fois unité de $\oplus$ et objet nul de $\M$, on dispose d'un endomorphisme idempotent
$$x\oplus y\to x=x\oplus 0\to x\oplus y$$
naturel en les objets $x$ et $y$ de $\M$ (voir la définition \ref{epsilon} et la proposition \ref{pr-eps}).
Ces propriétés de scindage assurent un bon comportement des effets croisés et des sous-catégories $\Pol_n(\M,\A)$ de foncteurs polynomiaux de degré au plus $n$, où $n$ est un entier fixé et  $\A$ est une cat\'egorie abélienne. Ces catégories sont notamment épaisses et stables par limites et colimites. En particulier, on peut former les catégories abéliennes quotients $\Pol_n(\M,\A)/\Pol_{n-1}(\M,\A)$. La description de ces catégories constitue une partie fondamentale  de la théorie des foncteurs polynomiaux de $\M$ dans $\A$, qui est classique lorsque $\M$ est additive et que $\oplus$ est la somme directe catégorique. De fait, la compréhension fine de $\Pol_n(\M,\A)$ s'avère délicate dès que $n$ est supérieur ou égal à $2$, même lorsque $\M$ et $\A$ sont des catégories assez simples \cite{BDFP, HV, HPV}, tandis que les catégories $\Pol_n(\M,\A)/\Pol_{n-1}(\M,\A)$ sont beaucoup plus accessibles \cite{Pira-rec, DV2}.

Si les exemples de foncteurs polynomiaux s'insérant dans le cadre qu'on vient de rappeler abondent dans plusieurs parties des mathématiques, on rencontre également de nombreux foncteurs qui n'y entrent pas mais possèdent des propriétés analogues aux foncteurs polynomiaux classiques \cite[§\,2]{CEFN}.
% --- par exemple, pour des foncteurs vers des espaces vectoriels de dimension finie, prendre des valeurs dont les dimensions sont données par une formule polynomiale (au moins << à partir d'un certain rang >>). 
Ainsi, dans un contexte ensembliste, la notion fondamentale de $\Gamma$-module, i.e. de foncteur de la catégorie $\Gamma$ des ensembles finis pointés vers les modules sur un anneau fixé, est parfois trop rigide. De nombreux foncteurs qui apparaissent naturellement, tels que la cohomologie des espaces de configurations, sont définis seulement sur la catégorie des ensembles finis avec injections, notée ici $\Theta$ et baptisée I dans \cite{Schwede} et FI dans  \cite{CEF}.  Les FI-modules sont les foncteurs de la catégorie $\Theta$ vers les groupes abéliens ou les modules sur un anneau fixé. Dans \cite{CEF}, Church, Ellenberg et Farb étudient ces FI-modules et en donnent de nombreux exemples. Ils montrent notamment que les FI-modules de type fini possèdent des propriétés polynomiales. Pour les foncteurs à valeurs dans les espaces vectoriels de dimension finie, cela correspond aux foncteurs dont les valeurs sont de dimensions polynomiales, à partir d'un certain rang. 

En fait, plusieurs des FI-modules considérés dans \cite{CEF} possèdent une fonctorialité plus forte. Ce sont non seulement des foncteurs sur $\Theta$ mais également des foncteurs sur la catégorie $\mathbf{S}(\mathbf{ab})$ des groupes abéliens libres de rang fini avec monomorphismes scindés. L'étude de certains de ces foncteurs, tels que l'homologie des groupes de congruence, constitue la principale motivation de l'introduction et de l'étude de la notion de foncteur polynomial dans un contexte très général, qui fait l'objet du présent travail.

Dans cet article, nous considérons des foncteurs d'une (petite) catégorie monoïdale symétrique $(\M,\oplus,0)$ telle que $0$ soit \textbf{objet initial} de $\M$ vers une catégorie abélienne $\A$ assez régulière. Ce cadre inclut en particulier les catégories $\Theta$ et $\mathbf{S}(\mathbf{ab})$ déjà évoquées, mais aussi les catégories d'espaces hermitiens.

La généralisation na\"ive de la définition de foncteur polynomial en termes d'effets croisés fournit la notion de foncteur \textit{fortement polynomial} introduite dans la définition \ref{polfor}. Cette notion nous permet notamment d'exprimer le résultat de \cite{CEF} évoqué précédemment dans le langage des foncteurs polynomiaux sous la forme suivante : les FI-modules de type fini sont des foncteurs fortement polynomiaux. Cependant, cette notion de polynomialité se comporte mal. Elle n'est, par exemple, pas stable par sous-objet.
La catégorie $\fct(\M,\A)$ des foncteurs de $\M$ vers $\A$ n'est donc pas la plus appropriée
pour obtenir une notion adaptée à cette situation d'objet polynomial. Par exemple, on souhaiterait que les sous-catégories $\Pol_n(\M,\A)$ soient épaisses. 

Soit $\A$ une catégorie d'espaces vectoriels de dimension finie. Les exemples de foncteurs de $\fct(\Theta,\A)$ qu'on rencontre couramment prennent des valeurs dont les dimensions sont données par une fonction du cardinal de l'ensemble qui n'est polynomiale qu'{\em à partir d'un certain rang}. On ne souhaite pas tenir compte des << phénomènes instables >> qui peuvent se produire sur les petits ensembles, qui ne s'insèrent pas bien dans la théorie. Pour y remédier, on travaille dans une catégorie $\st(\M,\A)$ {\em quotient} de $\fct(\M,\A)$ qui fait disparaître tous ces phénomènes instables, d'où la notation $\st$ pour {\em stable}. Par exemple, lorsque $\M=\Theta$, cette catégorie s'obtient en quotientant par la sous-catégorie épaisse des foncteurs $F$ tels que
$$\underset{n\in\mathbb{N}}{\col}F\big(\{1,\dots,n\}\big)=0$$
(voir la proposition \ref{theta-stab-nul}, ainsi que la proposition~\ref{dsn2} pour le cas où $\M$ est arbitraire).
Dans $\st(\M,\A)$, on peut donner une définition des objets polynomiaux qui imite une variante de la définition classique, lorsque $0$ est objet nul de $\M$. Cette définition utilise les {\em foncteurs différences} plutôt que les effets croisés. 
Lorsque $\M$ a un objet nul, cette définition co\" incide avec la définition usuelle de foncteur polynomial, par la proposition \ref{eq-def-nul}. On considère la sous-catégorie pleine de $\st(\M,\A)$ des objets polynomiaux de degré au plus $n$, que l'on note $\Pol_n(\M,\A)$ comme dans le cas où $\M$ a un objet nul. On montre, à la proposition \ref{propol-gal}, que ces catégories sont épaisses et stables par limites et colimites. Un foncteur de $\fct(\M,\A)$ est dit \textit{faiblement polynomial} lorsque son image dans $\st(\M,\A)$ est polynomiale.

L'étude des catégories quotients $\Pol_n(\M,\A)/\Pol_{n-1}(\M,\A)$ est au c\oe ur de cet article. Lorsque $\M=\Theta$, cette étude est assez simple et aboutit à la proposition \ref{eq-theta}. Nous nous intéressons ensuite au cas où $\M$ est une catégorie d'objets hermitiens $\mathbf{H}(\C)$ sur une petite catégorie additive à dualité $\C$. Cette notion, définie à la section \ref{section-hermitiens}, est une généralisation des catégories d'espaces hermitiens d'après l'exemple \ref{ex-hermitien}. Remarquons que les catégories d'objets hermitiens ont un objet initial qui n'est pas nul. Au théorème~\ref{thm-pal},  nous obtenons le résultat fondamental suivant :

\begin{thi}\label{ti}
Soient $\C$ une petite catégorie additive munie d'un foncteur de dualité $\C^{op}\to\C$ et  $\mathbf{H}(\C)$  la catégorie des {\em objets hermitiens} associés à la situation. On a une équivalence de catégories
$$\Pol_n(\mathbf{H}(\C),\A)/\Pol_{n-1}(\mathbf{H}(\C),\A) \simeq \Pol_n(\C,\A)/\Pol_{n-1}(\C,\A) $$
pour tout $n\in\mathbb{N}$ et toute catégorie de Grothendieck $\A$.
\end{thi}

Les catégories $\Pol_n(\C,\A)/\Pol_{n-1}(\C,\A)$ sont bien comprises à partir des foncteurs additifs $\C\to\A$ et des représentations des groupes symétriques. Par exemple, lorsque $\C$  est la catégorie des modules projectifs de type fini sur un anneau $A$, la catégorie $\Pol_n(\C,\A)/\Pol_{n-1}(\C,\A)$ est équivalente à la catégorie des modules sur le produit en couronne de $A$ par le groupe symétrique $\Sigma_n$ \cite{Pira-rec}.

C'est surtout le cas de la catégorie source $\mathbf{S}(\mathbf{ab})$, qui est un cas particulier du cadre hermitien général du théorème~\ref{ti}, que nous avons en vue pour la plupart des applications.

\smallskip

Dans le théorème~\ref{applsco}, nous appliquons le théorème~\ref{ti} à l'étude de certains $\mathbb{Q}$-espaces vectoriels d'homologie stable de groupes linéaires à coefficients polynomiaux. Nous montrons que cette homologie s'exprime comme produit tensoriel de son terme de degré nul et de l'homologie des mêmes groupes à coefficients dans $\mathbb{Q}$. Ce résultat se déduit rapidement d'un profond théorème de Scorichenko \cite{Sco} et du théorème précédent.

\smallskip

La démonstration du théorème \ref{ti} se décompose en deux étapes indépendantes. La première consiste à utiliser, pour toute catégorie monoïdale symétrique $(\M,\oplus, 0)$ telle que $0$ soit objet initial de $\M$, une catégorie monoïdale symétrique $(\widetilde{\M},\oplus,0)$ dont l'unité $0$ soit objet nul, munie d'un foncteur monoïdal $\M\to\widetilde{\M}$ et universelle pour cette propriété, selon une construction inspirée de Quillen (cf. remarque~\ref{quillen}). On montre au théorème~\ref{tilde-pol} que ce foncteur induit une équivalence de catégories
$$\Pol_n(\widetilde{\M},\A)/\Pol_{n-1}(\widetilde{\M},\A)\to\Pol_n(\M,\A)/\Pol_{n-1}(\M,\A)$$
pour tout entier $n$, par un raisonnement assez direct reposant sur la considération d'adjonctions appropriées.

La construction $\M\mapsto\widetilde{\M}$ redonne des catégories bien connues : la catégorie $\widetilde{\Theta}$ est équivalente à la catégorie FI\# de \cite{CEF} et est équivalente à une sous-catégorie de la catégorie $\Gamma$ sur laquelle les foncteurs ont été très étudiés, notamment par Pirashvili \cite{PDK, P-hodge}. 
Lorsque $\M$ est la catégorie des espaces quadratiques sur le corps à deux éléments $\FF$, $\widetilde{\M}$ est équivalente à la catégorie notée $\T_q$ dans l'article \cite{V-pol}, où le second auteur étudie des propriétés fines de la catégorie $\F_{quad}$ des foncteurs de $\T_q$ vers les $\FF$-espaces vectoriels avant de les appliquer à la description des objets polynomiaux de cette catégorie. 
 
 La deuxième étape de la démonstration du théorème~\ref{ti} adapte et généralise à notre cadre une partie des arguments de \cite{V-pol}. On montre ainsi que le foncteur d'oubli $\widetilde{\mathbf{H}(\C)}\to\C$ induit une équivalence de catégories
$$\Pol_n(\C,\A)\to\Pol_n(\widetilde{\mathbf{H}(\C)},\A)$$
pour tout $n$. Pour $\F_{quad}$,  cette équivalence correspond exactement au résultat principal de la dernière section de \cite{V-pol}. On déduit ce résultat du critère abstrait, obtenu à la proposition \ref{polfq-abstr}, qui donne des équivalences entre catégories de foncteurs polynomiaux. A la proposition ~\ref{pcm}, on applique également ce critère à une situation, analogue mais plus simple, traitant de monomorphismes d'une catégorie additive.

\medskip

L'un des intérêts du théorème~\ref{ti} vient de ce que de nombreux foncteurs utiles définis sur $\mathbf{S}(\mathbf{ab})$ sont difficiles à étudier, une description complète étant souvent hors d'atteinte. Nous en donnons dans la suite deux familles d'exemples importants.
\begin{itemize}
\item \textbf{Les groupes d'automorphismes des groupes libres induisant l'identité sur l'abélianisation.}
Les quotients successifs de leur suite centrale descendante ou leur homologie font l'objet de travaux féconds, mais encore très parcellaires \cite[§\,6]{K-Magnus} \cite{And, SIA, Bar, Pet}. On voit facilement que ces objets induisent des foncteurs $\mathbf{S}(\mathbf{ab})\to\mathbf{Ab}$. \`A la proposition \ref{IA-poly} nous montrons que le foncteur $IA_{ab}$ est fortement polynomial de degré fort $3$ et à la proposition \ref{scd-poly} nous montrons que les quotients successifs de la suite centrale descendante sont fortement polynomiaux de degré fort au plus $3n$. Pour l'homologie, le caractère polynomial semble très probable et fait l'objet d'un travail en cours. La détermination de leurs degrés exacts paraît encore plus délicate. La description de leurs images dans les sous-quotients appropriés de la filtration polynomiale de la catégorie $\fct(\mathbf{S}(\mathbf{ab}),\mathbf{Ab})$ pourrait constituer un objectif ultérieur, moins inaccessible que la compréhension complète de ces foncteurs. Enfin, à la proposition \ref{Johnson-poly} nous montrons que les quotients de la filtration de Johnson-Andreadakis sont des foncteurs faiblement polynomiaux de degré $n+2$.

\item \textbf{L'homologie des groupes de congruence.}
Pour $I$ un anneau non unitaire, on note $I_+$ l'anneau unitaire obtenu en adjoignant formellement une unité à $I$. En considérant le morphisme induit par l'épimorphisme scindé d'anneaux $I_+\twoheadrightarrow I_+/I\simeq\mathbb{Z}$, on définit le groupe linéaire par $GL_n(I):={\rm Ker}\,(GL_n(I_+)\twoheadrightarrow GL_n(\mathbb{Z}))$. Il s'agit donc d'un groupe de congruence. L'homologie du groupe $GL_n(I)$  définit, lorsque $n$ varie, un foncteur de $\Theta$ vers les groupes abéliens gradués, mais elle possède également une action de $GL_n(\mathbb{Z})$. 
Cette action s'avère cruciale dans l'étude de cette homologie \cite{Pu}. De fait, l'homologie des groupes de congruence sur $I$ définit non seulement un foncteur sur $\Theta$ mais également un foncteur sur la catégorie $\mathbf{S}(\mathbf{ab})$ noté $H_*(\Gamma_I)$. Pour ces foncteurs d'homologie (en degré fixé) des groupes de congruence, les propriétés polynomiales sont difficiles à vérifier. Certains cas particuliers dans cette direction sont établis par Suslin dans \cite{SK} pour le plus petit degré homologique <<~non excisif~>> et par Putman dans \cite{Pu} sous des conditions sur l'anneau sans unité $I$ et sur la caractéristique des coefficients. De fait, l'étude approfondie de l'homologie des groupes de congruence, dans un cadre général incluant d'autres types de groupes, à coefficients non seulement constants mais aussi dans des représentations données par des foncteurs polynomiaux appropriés  \cite[chapitre~6]{Dja-cong}, constitue la principale motivation de l'introduction et de l'étude de la notion de foncteur polynomial dans le contexte très général considéré dans ce travail. Nous conjecturons que pour tout anneau sans unité $I$ et tout entier $n\in\mathbb{N}$, le foncteur $H_n(\Gamma_I)$ est faiblement polynomial de degré au plus $2n$. D'ailleurs, \cite[Theorem D]{CEFN} peut s'exprimer dans le langage des foncteurs polynomiaux sous la forme : 
  pour $I$ un idéal propre d'un anneau d'entiers de corps de nombres et pour tout $n\in\mathbb{N}$, le foncteur $H_n(\Gamma_I) : \mathbf{S}(\mathbf{ab})\to\mathbf{Ab}$ est fortement polynomial.

\end{itemize}
\medskip
 Cet article est organisé de la manière suivante. Dans la première section nous définissons les foncteurs fortement et faiblement polynomiaux et en donnons quelques propriétés. Le but de la deuxième section est de montrer que la notion de foncteur fortement polynomial, définie en termes de foncteurs différences, s'exprime facilement à l'aide d'effet croisés et étend donc naturellement la définition classique. La section 3 concerne la construction d'une catégorie mono\"idale symétrique universelle dont l'unité est objet nul à partir d'une catégorie mono\"idale symétrique dont l'unité est objet initial. Dans la quatrième section nous étudions les foncteurs sur $\Theta$ et dans la cinquième section nous donnons des exemples de foncteurs polynomiaux sur $\mathbf{S}(\mathbf{ab})$. La dernière section est consacrée à la démonstration de notre résultat principal concernant les foncteurs polynomiaux sur les catégories d'espaces hermitiens.

\paragraph*{Quelques notations utilisées dans tout l'article}
Pour toute catégorie $\C$, on notera ${\rm Ob}\,\C$ la classe des objets de $\C$, et ${\rm Hom}_\C(x,y)$ ou $\C(x,y)$ l'ensemble des morphismes de $\C$ de source $x$ et de but $y$. Si $\C$ est une petite catégorie et $\A$ une catégorie quelconque, on désignera par $\fct(\C,\A)$ la catégorie des foncteurs de $\C$ dans $\A$. On s'autorisera parfois, par abus, à appliquer cette construction (et d'autres relatives à des petites catégories) lorsque $\C$ est seulement essentiellement petite.

Si $\Phi : \D\to\C$ est un foncteur entre petites catégories et $\A$ une catégorie, on désignera par $\Phi^* : \fct(\C,\A)\to\fct(\D,\A)$ le foncteur de précomposition par $\Phi$.

On note $\mathbf{Ab}$ la catégorie des groupes abéliens et $\mathbf{ab}$ la sous-catégorie pleine des groupes abéliens libres de rang fini.

On désigne par $\Theta$ la catégorie des ensembles finis, les morphismes étant les fonctions injectives, ou plutôt le squelette constitué des $\mathbf{n}:=\{1,\dots,n\}$, pour $n\in\mathbb{N}$, munie de la structure monoïdale symétrique donnée par la réunion disjointe.

\paragraph*{Remerciements} 
Les auteurs remercient Jacques Darné, Arthur Soulié et des rapporteurs anonymes pour leurs lectures vigilantes de versions antérieures de ce travail, qui en ont permis plusieurs améliorations.

Le lien entre la construction de la section~\ref{section2} et les travaux de Quillen donné à la remarque~\ref{quillen} est issu de discussions avec Nathalie Wahl que l'on remercie ici.

\section{Foncteurs polynomiaux sur une catégorie mono\"idale symétrique ayant un objet initial}\label{s-gal}

On note $\M on$ la catégorie des petites catégories monoïdales symétriques. Plus précisément, les objets de $\M on$ sont les petites catégories monoïdales symétriques strictes. On notera généralement $(\M,\oplus,0)$ un objet de $\M on$ même si $\oplus$ n'est pas une somme catégorique et $0$ n'est pas un objet nul. On omettra souvent, dans la suite, d'écrire explicitement la structure monoïdale, notée par défaut $\oplus$. Rappelons qu'une  catégorie monoïdale symétrique est  stricte si $a\oplus 0=0\oplus a=a$ et $(a\oplus b)\oplus c=a\oplus (b\oplus c)$, fonctoriellement en $a$, $b$ et $c$. Cet objet sera donc noté $a\oplus b\oplus c$. Il est classique que toute catégorie monoïdale symétrique est équivalente à une telle catégorie stricte \cite[Theorem 1 p. 257]{ML}. Les morphismes $(\M,\oplus,0)\to (\N,\oplus,0)$ de $\M on$ sont les foncteurs monoïdaux symétriques stricts, c'est-à-dire les foncteurs $F : \M\to \N$ tels que $F(0)=0$ et $F(a\oplus b)=F(a)\oplus F(b)$ fonctoriellement en les objets $a$ et $b$ de $\M$. 

Si $I$ est un ensemble fini, on peut définir de façon usuelle la somme $\underset{i\in I}{\bigoplus}a_i$ d'une famille $(a_i)_{i\in I}$ d'objets de $\M\in {\rm Ob}\,\M on$, et ce fonctoriellement en ladite famille. Cela est lié à la remarque suivante.

\begin{rem}\label{sigma}
Soit $\Sigma$ la catégorie des ensembles finis avec {\em bijections} ou plutôt le squelette constitué des ensembles $\mathbf{n}$ pour $n\in\mathbb{N}$. La structure monoïdale de cette catégorie est donnée par la réunion disjointe. La catégorie $\Sigma$ vérifie la propriété universelle suivante : pour tout objet $\M$ de $\M on$, la fonction
$$\M on(\Sigma,\M)\to {\rm Ob}\,\M\qquad F\mapsto F(\mathbf{1})$$
est une bijection.
\end{rem}

Si $(\M,\oplus, 0)$ est un objet de $\M on$, $x$ un objet de $\M$ et $\C$ une catégorie, on notera $\tau_x : \mathbf{Fct}(\M,\C)\to\mathbf{Fct}(\M,\C)$ le foncteur de {\rm translation par $x$}, c'est-à-dire la précomposition par l'endofoncteur $x\oplus -$ de $\M$. La symétrie de la structure monoïdale se traduit par l'existence d'isomorphismes naturels $\tau_x\circ\tau_y\simeq\tau_y\circ\tau_x(\simeq\tau_{x\oplus y})$.

On note $\mn$ la sous-catégorie pleine de $\M on$ constituée des catégories monoïdales symétriques $(\M,\oplus, 0)$ dont l'unité $0$ est objet nul et $\mi$ celle constituée des catégories monoïdales symétriques $(\M,\oplus, 0)$ dont l'unité $0$ est objet initial.

Les exemples de catégories dans $\mi$ mais pas dans $\mn$ qui nous intéressent sont généralement des catégories de monomorphismes.

\begin{ex}\label{exini}
\begin{enumerate}
\item\label{nini}
Un exemple important d'objet de $\mi$ est la catégorie $\Theta$. Elle possède la propriété universelle suivante :
pour tout objet $\M$ de $\mi$, la fonction
$$\M on(\Theta,\M)\to {\rm Ob}\,\M\qquad F\mapsto F(\mathbf{1})$$
est une bijection (cf. remarque~\ref{sigma}). Pour une démonstration formelle de ce fait élémentaire, voir la remarque~\ref{rq-tilde} ci-après.
\item 
Les catégories d'objets hermitiens $\mathbf{H}(\A)$, pour lesquelles on renvoie le lecteur à la section \ref{section-hermitiens}, constituent des exemples fondamentaux d'objets de $\mi$.
\item \label{exm}
Soit $\A$ une petite catégorie additive. On note $\mathbf{M}(\A)$ la sous-catégorie de $\A$ ayant les mêmes objets et dont les morphismes sont les monomorphismes scindables.
\item\label{exs} 
Soit $\A$ une petite catégorie additive. On note $\mathbf{S}(\A)$ la catégorie ayant les mêmes objets que $\A$ et dont les morphismes sont donnés par :
$$\mathbf{S}(\A)(a,b)=\{(u,v) \in \A(b,a) \times \A(a,b) \mid u \circ v={\rm Id}_a \}.$$
Cette catégorie est un objet de $\mi$, la structure monoïdale symétrique étant la somme directe.
La catégorie $\mathbf{S}(\A)$ s'identifie à la catégorie d'objets hermitiens $\mathbf{H}(\A^{op} \times \A)$ (voir la définition~\ref{df-herm} ci-après).

D'autre part, on dispose d'un foncteur mono\"idal canonique de $\mathbf{S}(\A)$ dans $\A$, qui est l'identité sur les objets et dont l'image est $\mathbf{M}(\A)$.
\end{enumerate}
\end{ex}

\subsection{Foncteurs polynomiaux forts}

\begin{defi}\label{delta-kappa}
 Soient $\M$ un objet de $\mi$, $x$ un objet de $\M$, $\A$ une catégorie abélienne, et $F : \M\to\A$ un foncteur. On note $\kappa_x(F)$ (resp. $\delta_x(F)$) le noyau (resp. conoyau) du morphisme $i_x(F) : F=\tau_0(F)\to\tau_x(F)$ de $\fct(\M,\A)$ induit par l'unique morphisme $0\to x$ de $\M$.

L'endofoncteur $\delta_x$ de $\fct(\M,\A)$ s'appelle {\em foncteur différence} associé à~$x$.
\end{defi}
Nous avons donc une suite exacte :

\begin{equation}  \label{se1}
0 \rightarrow \kappa_x \rightarrow {\rm Id} \rightarrow \tau_x \rightarrow \delta_x \rightarrow 0.
\end{equation}

Comme le foncteur $\tau_x$ est exact, le lemme du serpent montre que toute suite exacte courte
$$0\to F\to G\to H\to 0$$
de $\fct(\M,\A)$ induit une suite exacte

\begin{equation}  \label{se2}
0\to\kappa_x(F)\to\kappa_x(G)\to\kappa_x(H)\to\delta_x(F)\to\delta_x(G)\to\delta_x(H)\to 0,
\end{equation}
observation que nous utiliserons couramment dans cet article.

Avant d'utiliser ces foncteurs pour définir les foncteurs polynomiaux, nous en donnons quelques propriétés générales qui nous serviront à plusieurs reprises.

\begin{pr}\label{pr-dkt}
 Soient $\M$ un objet de $\mi$, $x$ et $y$ des objets de $\M$ et $\A$ une catégorie abélienne.
\begin{enumerate}
 \item \label{p0p} Les endofoncteurs $\tau_x$ et $\tau_y$ de $\fct(\M,\A)$ commutent à isomorphisme naturel près. Ils commutent aux limites et aux colimites.
\item\label{p1p} Les endofoncteurs $\delta_x$ et $\delta_y$ de $\fct(\M,\A)$ commutent à isomorphisme naturel près. Ils commutent aux colimites.
\item\label{p2p} Les endofoncteurs $\kappa_x$ et $\kappa_y$ de $\fct(\M,\A)$ commutent à isomorphisme naturel près. Ils commutent aux limites.
\item\label{pid} L'endofoncteur $\kappa_x$ est idempotent : l'inclusion naturelle $(\kappa_x)^2\hookrightarrow\kappa_x$ est un isomorphisme.
\item\label{p4p} Les endofoncteurs $\tau_x$ et $\delta_y$ commutent à isomorphisme naturel près.
\item\label{p5p} Les endofoncteurs $\tau_x$ et $\kappa_y$ commutent à isomorphisme naturel près.
\item\label{p6p} Il existe une suite exacte naturelle
$$0 \to \kappa_y \to \kappa_{x\oplus y} \to \tau_y \kappa_x \to \delta_y\to\delta_{x\oplus y}\to\tau_y\delta_x\to 0.$$
\end{enumerate}

\end{pr}

\begin{proof}
 La première assertion est claire.

En s'appuyant sur celle-ci, on forme un diagramme commutatif aux lignes exactes
$$\xymatrix{0\ar[r] & \tau_y\kappa_x\ar[r]\ar@{-->}[d]^\simeq & \tau_y\ar[r]\ar@{=}[d]& \tau_y\tau_x\ar[r]\ar[d]^\simeq & \tau_y\delta_x\ar[r]\ar@{-->}[d]^\simeq & 0\\
0\ar[r] & \kappa_x\tau_y\ar[r] & \tau_y\ar[r]& \tau_x\tau_y\ar[r] & \delta_x\tau_y\ar[r] & 0
}$$
qui établit les assertions~{\em \ref{p4p}.} et~{\em \ref{p5p}.}

Le foncteur $\delta_x$, conoyau d'une transformation naturelle entre ${\rm Id}$ et $\tau_x$, qui commutent aux colimites, commute aux colimites. En utilisant~{\em \ref{p4p}.}, l'exactitude à droite de $\delta_y$ et un raisonnement analogue au précédent, on en déduit l'assertion~{\em \ref{p1p}.} ; la propriété~{\em \ref{p2p}.} est similaire.

L'assertion~{\em \ref{pid}.} est une conséquence formelle de ce que $\kappa_x$ est le noyau d'une transformation naturelle de l'identité vers un foncteur exact à gauche.

La propriété~{\em \ref{p6p}.} s'obtient en appliquant à $u=i_y$ et $v=\tau_y(i_x)$ le fait que si $u$ et $v$ sont deux morphismes composables d'une catégorie abélienne on a une suite exacte
$$0 \to \textrm{Ker}(u) \to \textrm{Ker}(v \circ u) \to \textrm{Ker}(v) \to \textrm{Coker}(u) \to \textrm{Coker}(v \circ u) \to \textrm{Coker}(v) \to 0$$
et en utilisant l'exactitude de $\tau_y$.
\end{proof}

\begin{defi}[Foncteurs fortement polynomiaux]\label{polfor}
 Sous les hypothèses précédentes, on définit par récurrence sur $n$ une suite de sous-catégories pleines $\Pol^{{\rm fort}}_n(\M,\A)$ de $\fct(\M,\A)$ comme suit.
\begin{enumerate}
 \item Pour $n<0$, $\Pol^{{\rm fort}}_n(\M,\A)=\{0\}$ ;
\item pour $n\geq 0$, $\Pol^{{\rm fort}}_n(\M,\A)$ est constitué des foncteurs $F$ tels que $\delta_x(F)$ appartienne à $\Pol^{{\rm fort}}_{n-1}(\M,\A)$ pour tout objet $x$ de $\M$.
\end{enumerate}
On dit qu'un foncteur $F$ est {\em fortement polynomial} s'il existe $n$ tel que $F$ appartienne à $\Pol^{{\rm fort}}_n(\M,\A)$. Le minimum dans $\mathbb{N}\cup\{\infty\}$ de ces $n$ s'appelle le {\em degré fort} de $F$.
\end{defi}

\begin{rem}
 Dans \cite[§\,4.4]{RWW}, Randal-Williams et Wahl étudient de manière unifiée la stabilité pour l'homologie des groupes discrets, y compris pour des systèmes de coefficients. Rappelons que {\em système de coefficients} est synonyme de foncteur. Le cadre de \cite{RWW}, un peu plus général que celui du présent article, est celui d'une catégorie source monoïdale {\em prétressée} (notion définie dans \cite{RWW}) dont l'unité est objet initial. Randal-Williams et Wahl introduisent une notion de {\em système de coefficients de degré $d$}. Cette notion est encore plus forte que celle de foncteur fortement polynomial de degré $d$. Elle est définie par récurrence sur $d$ en faisant intervenir non seulement les foncteurs différences $\delta_x$, mais aussi les foncteurs $\kappa_x$.
\end{rem}

La proposition suivante nous sera utile à la section \ref{ex-S(ab)}.
\begin{pr} \label{compo-mono-fort}
Soient $\M$ et $\M'$ des objets de $\mi$ et $\alpha: \M \to \M'$ un foncteur mono\"idal fort, alors la précomposition par $\alpha$ fournit un foncteur :
$$\Pol^{{\rm fort}}_n(\M',\A) \to \Pol^{{\rm fort}}_n(\M,\A).$$
\end{pr}
\begin{proof}
Pour tous objets $x$ et $y$ de $\M$, comme $\alpha$ est un foncteur mono\"idal fort on a :
$$\tau_x(F \circ \alpha)(y) \simeq \tau_{\alpha(x)}F(\alpha(y)).$$
On en déduit l'isomorphisme
$$\delta_x(F \circ \alpha) \simeq (\delta_{\alpha(x)}F)\circ \alpha$$
dont découle le résultat.
\end{proof}

\begin{pr}\label{pr-polfor}
 Les sous-catégories $\Pol^{{\rm fort}}_n(\M,\A)$ de $\fct(\M,\A)$ sont stables par quotients, extensions et colimites. Elles sont également stables sous l'action des foncteurs de translation $\tau_x$.

Si $E$ est un ensemble d'objets de $\M$ tel que tout objet de $\M$ est isomorphe à une somme (au sens de $\oplus$) finie d'éléments de $E$, alors un foncteur $F$ de $\fct(\M,\A)$ appartient à $\Pol^{{\rm fort}}_n(\M,\A)$ si $\delta_x(F)$ appartient à $\Pol^{{\rm fort}}_{n-1}(\M,\A)$ pour tout $x\in E$. 
\end{pr}

\begin{proof}
 On s'appuie sur la proposition~\ref{pr-dkt} : comme les foncteurs $\delta_x$ commutent aux colimites, on voit aussitôt par récurrence sur $n$ que $\Pol^{{\rm fort}}_n(\M,\A)$ est stable par quotients, extensions et colimites. La commutation des foncteurs $\delta_x$ aux foncteurs de translation implique que ces derniers préservent $\Pol^{{\rm fort}}_n(\M,\A)$.

La dernière assertion se déduit, par récurrence sur $n$, du point~{\em \ref{p6p}.} de la proposition~\ref{pr-dkt}, en utilisant la stabilité de $\Pol^{{\rm fort}}_n(\M,\A)$ par extensions, quotients et par les foncteurs de translation.
\end{proof}

Si $\M$ appartient à $\mn$, alors pour tout $x\in {\rm Ob}\,\M$, $i_x$ est, comme l'unique morphisme $0\to x$, un monomorphisme scindé, par conséquent, $\kappa_x$ est nul et $\delta_x$ est exact et commute aux limites. Toutes ces propriétés sont généralement fausses pour les catégories de $\mi$ sans objet nul. C'est ce qui explique que la définition~\ref{polfor} ne se comporte pas très bien : un sous-foncteur d'un foncteur fortement polynomial n'est pas forcément fortement polynomial, et même lorsqu'il l'est, son degré fort peut excéder celui du foncteur initial. L'exemple~\ref{ex-degfort} en donnera une illustration simple.

Les foncteurs fortement polynomiaux de degré nul sont caractérisés comme suit.

\begin{pr} \label{fort-0}
Soit $F$ dans $\fct(\M,\A)$. Les assertions suivantes sont équivalentes :
\begin{enumerate}
\item $F$ est fortement polynomial de degré fort $0$ ;
\item $F$ est quotient d'un foncteur constant ;
\item pour tout objet $x$ de $\M$, le morphisme $F(0) \to F(x)$ induit par $0\to x$ est un épimorphisme.
\end{enumerate}
\end{pr}

\begin{proof}
Supposons que $F$ est fortement polynomial. Soit $x$ un objet de $\M$.  La nullité de $\delta_x(F)(0)$ signifie que le morphisme $F(0)\to F(x)$ est un épimorphisme,
 ce qui montre que la première assertion implique la dernière.
 
 De la troisième assertion on déduit facilement que $F$ est quotient du foncteur constant en $F(0)$.
 
La deuxième assertion implique la première parce qu'un foncteur constant est fortement polynomial de degré $0$ et que $\Pol^{{\rm fort}}_0(\M,\A)$ est stable par quotients par la proposition \ref{pr-polfor}. 
\end{proof}

Les foncteurs fortement polynomiaux de degré fort $n$, pour $n$ strictement supérieur à $0$, sont beaucoup plus difficiles à caractériser, y compris pour $n=1$. 

\smallskip

On souhaite maintenant modifier la définition~\ref{polfor} de manière à ce qu'un sous-foncteur d'un foncteur polynomial de degré $d$ soit polynomial de degré au plus~$d$, par exemple. Cela se fait au prix du passage de la catégorie $\fct(\M,\A)$ à une catégorie quotient. Par commodité, on se placera d'emblée dans le cas où $\A$ est une <<~bonne~>> catégorie abélienne. On renvoie le lecteur à \cite{Gab} pour les généralités sur les catégories abéliennes, notamment les catégories de Grothendieck  qui sont les catégories abéliennes avec générateurs et limites inductives exactes  \cite[chap.~II, §\,6]{Gab} et les catégories abéliennes quotients \cite[chap.~III]{Gab}. Une autre motivation à l'introduction de cette catégorie quotient, qui en justifie l'appellation, consiste à oublier les phénomènes qui ne mettent en jeu que les <<~petites valeurs~>> d'un foncteur.

\subsection{La catégorie stable $\st(\M,\A)$}
\begin{defi}\label{cat-stable1}
 Soient $\M$ un objet de $\mi$, $\A$ une catégorie de Grothendieck et $F : \M\to\A$ un foncteur. On note $\kappa(F)$ le sous-foncteur $\underset{x\in {\rm Ob}\,\M}{\sum}\kappa_x(F)$ de $F$.

On dit que $F$ est {\em stablement nul} si $\kappa(F)=F$.

On note $\s n(\M,\A)$ la sous-catégorie pleine de $\fct(\M,\A)$ constituée des foncteurs stablement nuls.
\end{defi}

\begin{rem}\label{rqsne}
Un foncteur sur $\M$ est dit {\em presque nul} si ses valeurs sont nulles sauf sur un nombre fini de classes d'isomorphisme d'objets de $\M$.
 Sous des hypothèses légères sur $\M$, qui sont évidemment vérifiées dans tous les exemples abordés dans cet article, on peut montrer \cite[§\,3.3]{Dja-pol} que les foncteurs {\em presque nuls} sont stablement nuls et que les foncteurs stablement nuls de type fini sont presque nuls. De plus, toujours sous de légères hypothèses sur $\M$, un foncteur $F$ est presque nul si et seulement si l'un des décalages $\tau_x(F)$ est nul.
\end{rem}

Si la catégorie $\M$ appartient à $\mn$, $\kappa_x$ est toujours nul, de sorte que la catégorie $\s n(\M,\A)$ est réduite à $0$.

\begin{pr} \label{kappa-exact}
Le foncteur $\kappa$ est exact à gauche.
\end{pr}

\begin{proof}
Les flèches canoniques $x\to x\oplus y$ et $y\to x\oplus y$ fournissent des inclusions $\kappa_x(F)\subset\kappa_{x\oplus y}(F)$ et $\kappa_y(F)\subset\kappa_{x\oplus y}(F)$ qui montrent que la famille $(\kappa_x(F))$ de sous-objets de $F$ est {\em filtrante croissante}. Le foncteur $\kappa$ est donc exact à gauche en tant que colimite filtrante des foncteurs exacts à gauche $\kappa_x$.
\end{proof}

Si $E$ est un ensemble, on note $\mathbb{N}^{(E)}$ l'ensemble des fonctions $E\to\mathbb{N}$ dont toutes les valeurs, sauf un nombre fini, sont nulles ; on munit cet ensemble de la relation d'ordre induite par l'ordre produit, où $\mathbb{N}$ est muni de l'ordre naturel. Cet ensemble ordonné est {\em filtrant}. Si l'on suppose maintenant que $E$ est un ensemble d'objets d'une catégorie monoïdale $\M$ de $\mi$ et qu'on s'est donné un ordre total sur $E$, on dispose d'un foncteur $\zeta :  \mathbb{N}^{(E)}\to\M$ associant $\underset{e\in E}{\bigoplus}e^{\oplus\varphi(e)}$ à $\varphi\in\mathbb{N}^{(E)}$, la somme étant prise dans l'ordre dicté par celui de $E$. L'effet sur un morphisme $\varphi\leq\psi$ est  la somme sur les $e\in E$ des morphismes $e^{\oplus\varphi(e)}=e^{\oplus\varphi(e)}\oplus 0\to e^{\oplus\varphi(e)}\oplus e^{\oplus (\psi(e)-\varphi(e))}=e^{\oplus\psi(e)}$ utilisant le caractère initial de l'unité $0$. On munit $\mathbb{N}^{(E)}$ de la structure monoïdale induite par l'addition sur $\mathbb{N}$. Comme la structure monoïdale $\oplus$ est symétrique, $\zeta$ est monoïdal : $\zeta(\varphi+\psi)\simeq\zeta(\varphi)\oplus\zeta(\psi)$. De plus, la relation $\varphi\leq\varphi+\psi$ induit en appliquant $\zeta$ un morphisme isomorphe au morphisme canonique $\zeta(\varphi)\to\zeta(\varphi)\oplus\zeta(\psi)$.

\begin{pr}\label{dsn2} Soient $\M$ un objet de $\mi$, $\A$ une catégorie de Grothendieck et $E$ un ensemble d'objets de $\M$ tel que tout objet de $\M$ soit isomorphe à une somme finie, au sens de la structure monoïdale $\oplus$, d'objets de $E$. On munit $E$ d'un ordre total arbitraire. 
 
 Alors, avec les notations précédentes, un foncteur $F : \M\to\A$ est stablement nul si et seulement si $\underset{\mathbb{N}^{(E)}}{\col}\zeta^*F=0$.
\end{pr}

\begin{proof} L'hypothèse faite sur $E$ signifie que $\zeta$ est essentiellement surjectif.

Si $F : \M\to\A$ est tel que $\kappa_t(F)=F$ pour un objet $t$ de $\M$, alors le morphisme canonique $F(x)\to F(x \oplus t)$ est nul pour tout $x\in {\rm Ob}\,\M$. Soient $\psi$ un élément de $\mathbb{N}^{(E)}$ tel que $\zeta(\psi)\simeq t$ et $\varphi$ un élément quelconque de $\mathbb{N}^{(E)}$. Le morphisme de $\M$ induit par la relation $\varphi\leq\varphi+\psi$ de $\mathbb{N}^{(E)}$ fournit, en appliquant $\zeta$, une flèche isomorphe à la flèche canonique $\zeta(\varphi)\to\zeta(\varphi)\oplus t$. Cette relation induit donc $0$ en appliquant $\zeta^* F$. On en déduit $\underset{\mathbb{N}^{(E)}}{\col}\zeta^*F=0$. Comme tout foncteur stablement nul s'écrit comme une colimite de foncteurs $F$ tels que $\kappa_t(F)=F$ pour un certain $t$, grâce à la proposition~\ref{pr-dkt}.\ref{pid}, on en déduit que $\underset{\mathbb{N}^{(E)}}{\col}\zeta^*F=0$ pour $F$ dans $\s n(\M,\A)$.

 Supposons réciproquement que $F$ vérifie $\underset{\mathbb{N}^{(E)}}{\col}\zeta^*F=0$. Soient $x$ un objet de $\M$ et $\varphi$ un élément de $\mathbb{N}^{(E)}$ tel que $\zeta(\varphi)\simeq x$. 
 Le noyau de la flèche canonique 
 $$F(x)\simeq\zeta^* F(\varphi)\to\underset{\mathbb{N}^{(E)}}{\col}\zeta^*F$$
 est égal à $F(x)$ puisque le but de cette flèche est nul. Considérons un élément de $\mathbb{N}^{(E)}$ supérieur à $\varphi$, c'est-à-dire de la forme $\varphi+\psi$ pour un $\psi\in\mathbb{N}^{(E)}$.
 Les noyaux des flèches 
 $$F(x)\simeq \zeta^*  F(\varphi)\to \zeta^* F(\varphi+\psi)$$
  induites par le morphisme canonique $x\simeq\zeta(\varphi)\to\zeta(\varphi+\psi)\simeq x+\zeta(\psi)$, sont égaux à $\kappa_{\zeta(\psi)}(F)(x)$. 
  Comme l'ensemble ordonné $\mathbb{N}^{(E)}$ est {\em filtrant} et que les colimites filtrantes sont exactes dans $\A$, on en déduit que  la colimite sur les éléments de $\mathbb{N}^{(E)}$ supérieurs à $\varphi$ de $\kappa_{\zeta(\psi)}(F)(x)$ est égale à $F(x)$. On en tire $F(x)=\kappa(F)(x)$ comme souhaité.
\end{proof}

\begin{cor}\label{corlfi}
 Soient $\M$ un objet de $\mi$, $\A$ une catégorie de Grothendieck et $t$ un objet de $\M$. On suppose que tout objet de $\M$ est isomorphe à $t^{\oplus n}$ pour un $n\in\mathbb{N}$. Alors un foncteur $F : \M\to\A$ appartient à $\s n(\M,\A)$ si et seulement si $\underset{n\in\mathbb{N}}{\col}F(t^{\oplus n})=0$.
\end{cor}

Le caractère filtrant de l'ensemble ordonné $\mathbb{N}^{(E)}$ permet de déduire de la proposition~\ref{dsn2} le résultat suivant.

\begin{cor}
Soient $\M$ un objet de $\mi$ et  $\A$ une catégorie de Grothendieck. La catégorie $\s n(\M,\A)$ est une sous-catégorie épaisse stable par colimites de $\fct(\M,\A)$.
\end{cor}

\medskip

\textbf{Dans toute la suite de ce paragraphe, $\M$ désigne un objet de $\mi$ et $\A$ une catégorie de Grothendieck.}

\begin{defi}\label{cat-stable2}
 On note $\st(\M,\A)$ la catégorie quotient $\fct(\M,\A)/\s n(\M,\A)$, $\pi_\M : \fct(\M,\A)\to\st(\M,\A)$ le foncteur canonique et $s_\M : \st(\M,\A)\to\fct(\M,\A)$ le foncteur section, c'est-à-dire l'adjoint à droite du foncteur $\pi_\M$ \cite[chap. III]{Gab}.
\end{defi}

Ainsi, $\st(\M,\A)$ est une catégorie de Grothendieck et $\pi_\M$ est un foncteur exact, essentiellement surjectif et commutant à toutes les colimites.

Toutes ces constructions sont fonctorielles en $\M$ en le sens suivant : si $\Phi : \M\to\N$ est une flèche de $\mi$ qui est essentiellement surjective, ou plus généralement telle que pour tout objet $b$ de $\N$ existe un objet $a$ de $\M$ et un morphisme $b\to\Phi(a)$, on dispose d'un diagramme commutatif
$$\xymatrix{\s n(\N,\A)\ar[r]^{{\rm incl}}\ar[d] & \fct(\N,\A)\ar[r]^{\pi_\N}\ar[d]^{\Phi^*} & \st(\N,\A)\ar[d] \\
\s n(\M,\A)\ar[r]^{{\rm incl}} & \fct(\M,\A)\ar[r]^{\pi_\M} & \st(\M,\A)
}$$
 de catégories abéliennes, dont toutes les flèches sont exactes et commutent aux colimites.

\medskip

Le lemme suivant nous aidera à mener à bien certains raisonnements sur les morphismes de la catégorie $\st(\M,\A)$.

\begin{lm}\label{lm-stabnul}
 Soit $F : \M\to\A$ un foncteur.
\begin{enumerate}
 \item Le foncteur $\kappa(F)$ est le plus grand sous-objet de $F$ appartenant à $\s n(\M,\A)$.
\item On a ${\rm Hom}_{\fct(\M,\A)}(S,F)=0$ pour tout objet $S$ de $\s n(\M,\A)$ si et seulement si $i_x(F) : F\to\tau_x(F)$ est un monomorphisme pour tout objet $x$ de $\M$ (i.e. $\kappa(F)=0$).
\item\label{lsp3} Si $i_x(F) : F\to\tau_x(F)$ est un monomorphisme {\em scindé} pour tout objet $x$ de $\M$, alors ${\rm Ext}^*_{\fct(\M,\A)}(S,F)=0$ pour tout objet $S$ de $\s n(\M,\A)$. Par conséquent, le morphisme naturel
$${\rm Ext}^*_{\fct(\M,\A)}(G,F)\to {\rm Ext}^*_{\st(\M,\A)}(\pi_\M G,\pi_\M F)$$
est un isomorphisme pour tout objet $G$ de $\fct(\M,\A)$.
\item Si $F$ appartient à $\s n(\M,\A)$ et que $X : \M^{op}\to\mathbf{Ab}$ est un foncteur tel que la transformation naturelle $i_x : \tau_x(X)\to X$ soit un épimorphisme scindé pour tout $x\in {\rm Ob}\,\M$, alors ${\rm Tor}_*^\M(X,F)=0$.
\item En particulier, $H_*(\M;F)=0$ si $F$ appartient à $\s n(\M,\A)$.
\end{enumerate}
\end{lm}

\begin{proof}
 Le premier point est un fait formel classique qui découle de la propriété d'idempotence des foncteurs $\kappa_x$ donnée par la proposition~\ref{pr-dkt}.\ref{pid} : elle implique que les $\kappa_x(F)$, donc aussi leur somme $\kappa(F)$, appartiennent à $\s n(\M,\A)$. Réciproquement, si $G$ est un sous-foncteur de $F$ appartenant  à $\s n(\M,\A)$, alors $G=\kappa(G)\subset\kappa(F)$ puisque $\kappa$ est exact à gauche d'après la proposition \ref{kappa-exact}.

Le deuxième point en résulte : ${\rm Hom}_{\fct(\M,\A)}(S,F)=0$ pour tout objet $S$ de $\s n(\M,\A)$ si et seulement si $\kappa(F)=0$, ce qui signifie bien que $\kappa_x(F)=0$, i.e. $i_x(F)$ est injective, pour tout $x\in{\rm Ob}\,\M$.

Pour le troisième point, comme $S\in {\rm Ob}\,\s n(\M,\A)$ est la colimite de ses sous-objets $\kappa_x(S)$ et que $i_x(\kappa_x)=0$ il suffit de montrer que  ${\rm Ext}^*_{\fct(\M,\A)}(S,F)=0$ lorsqu'existe un $x\in{\rm Ob}\,\M$ tel que $i_x(S)=0$. En ce cas, on considère le diagramme commutatif, déduit de la naturalité de $i_x$ et de l'exactitude de $\tau_x$
$$\xymatrix{{\rm Ext}^*_{\fct(\M,\A)}(S,F)\ar[r]^{i_x(F)_*}\ar[d]_{(\tau_x)_*} & {\rm Ext}^*_{\fct(\M,\A)}(S,\tau_x(F)) \\
{\rm Ext}^*_{\fct(\M,\A)}(\tau_x(S),\tau_x(F))\ar[ru]_-{i_x(S)^*} &
}.$$
La flèche oblique est nulle par hypothèse sur $S$, tandis que la flèche horizontale est un monomorphisme (scindé) par hypothèse sur $F$, d'où la nullité souhaitée.

Le quatrième point se démontre de façon duale du troisième. Le cinquième s'en déduit en prenant pour $X$ le foncteur constant en $\mathbb{Z}$.
\end{proof}

On déduit aussitôt du dernier point de ce lemme le résultat suivant, que l'on utilisera dans la section~\ref{section2} et où l'indice $gr$ désigne les objets gradués.

\begin{pr}\label{pr-homst}
 Le foncteur homologique $H_*(\M;-) : \fct(\M,\A)\to\A_{gr}$ se factorise par $\pi_\M$, induisant un foncteur homologique $h_*^\M : \st(\M,\A)\to\A_{gr}$.
\end{pr}

On montre maintenant que $\tau_x$ et $\delta_x$ induisent des endofoncteurs de $\st(\M,\A)$ qui se comportent comme on l'attend.

\begin{pr}\label{pr-fde}
\begin{enumerate}
\item Pour tout objet $x$ de $\M$, les endofoncteurs $\tau_x$ et $\delta_x$ de $\fct(\M,\A)$ induisent des endofoncteurs de $\st(\M,\A)$ {\em exacts} et commutant aux colimites, encore notés $\tau_x$ et $\delta_x$. Ils vérifient les relations $\pi_\M \delta_x=\delta_x \pi_\M$ et $\pi_\M \tau_x=\tau_x \pi_\M$ et s'insèrent dans une suite exacte
$$0\to {\rm Id}\to\tau_x\to\delta_x\to 0$$
naturelle en $x$.
\item Pour tous objet $x$ et $y$ de $\M$, les endofoncteurs $\tau_x$, $\tau_y$, $\delta_x$, $\delta_y$ de $\st(\M,\A)$ commutent deux à deux à isomorphisme naturel près.
\end{enumerate}
\end{pr}

\begin{proof}
D'après la proposition~\ref{pr-dkt}, $\tau_x$ commute aux colimites et aux foncteurs $\kappa_y$. On en déduit la commutation de $\tau_x$ au foncteur $\kappa$ ce qui implique que $\tau_x$ induit un endofoncteur de $\st(\M,\A)$. Comme $\tau_x$ est exact et commute aux colimites, cet endofoncteur de $\st(\M,\A)$ a les mêmes propriétés.

Par ailleurs, la proposition~\ref{pr-dkt}.~{\em \ref{pid}.} montre que le foncteur $\kappa_x$ prend ses valeurs dans $\s n(\M,\A)$. Il s'en suit que la transformation naturelle $i_x : {\rm Id}\to\tau_x$ d'endofoncteurs de $\fct(\M, \A)$ induit un {\em monomorphisme} naturel d'endofoncteurs de $\st(\M,\A)$. Comme ${\rm Id}$ et $\tau_x$ sont exacts dans $\st(\M,\A)$, cela implique, grâce au lemme du serpent, que leur conoyau est un endofoncteur exact de $\st(\M,\A)$, qui est induit par l'endofoncteur $\delta_x$ de $\fct(\M,\A)$. Cet endofoncteur de $\st(\M,\A)$ commute aux colimites puisqu'il en est de même pour ${\rm Id}$ et $\tau_x$. Cela achève d'établir la première assertion. La deuxième se déduit directement de la proposition~\ref{pr-dkt}.
\end{proof}

Avant d'étudier le comportement des foncteurs $\tau_x$ et $\delta_x$ relativement au foncteur section, on rappelle \cite[chap. III]{Gab} qu'un foncteur $F$ de $\fct(\M,\A)$ est dit $\s n(\M,\A)$-fermé si 
$${\rm Ext}^i(H,F)=0$$
pour $i \leq 1$ et $H \in \s n(\M,\A)$, ce qui équivaut à dire que l'unité $\eta: F \to s_\M \pi_\M(F)$ de l'adjonction entre $\pi_\M$ et $s_\M$ est un isomorphisme \cite[corollaire de la page $371$]{Gab}.

Dans la proposition suivante et sa démonstration, tous les ensembles de morphismes sont pris sur la catégorie $\fct(\M,\A)$.
\begin{pr}\label{sect}\begin{enumerate}\item
 Soient $F$ un objet de $\fct(\M,\A)$ et $x$ un objet de $\M$.

\begin{enumerate}
 \item Si ${\rm Hom}(N,F)=0$ pour tout $N$ dans $\s n(\M,\A)$, alors $\tau_x(F)$ possède la même propriété --- autrement dit, l'unité $\tau_x(F) \to s_\M \pi_\M(\tau_x(F))$  de l'adjonction entre $\pi_\M$ et $s_\M$ est un monomorphisme.
\item Si $F$ est $\s n(\M,\A)$-fermé, alors ${\rm Hom}(N,\delta_x(F))=0$ pour tout $N$ dans $\s n(\M,\A)$.
\item Si $F$ est $\s n(\M,\A)$-fermé et que $0\to G\to F\to N\to 0$ est une suite exacte de $\fct(\M,\A)$, alors $\kappa_x(N)\simeq\kappa_x\delta_x(G)$. 
\end{enumerate}
\item\begin{enumerate}
\item Il existe un isomorphisme naturel de foncteurs $\tau_t s_\M\simeq s_\M\tau_t$ pour tout objet $t$ de $\M$.
\item Il existe un monomorphisme naturel de foncteurs $\delta_t s_\M\hookrightarrow s_\M\delta_t$ pour tout objet $t$ de $\M$.
\end{enumerate}
\end{enumerate}
\end{pr}

\begin{proof}
\begin{enumerate}
 \item 
\begin{enumerate}
 \item La condition ${\rm Hom}(N,F)=0$ pour tout $N$ dans $\s n(\M,\A)$ équivaut à $\kappa(F)=0$ par la première assertion du lemme~\ref{lm-stabnul}. Le premier point résulte donc de la commutation des foncteurs $\tau_x$ et $\kappa$ donnée par la proposition~\ref{pr-dkt}.
\item Comme $F$ est $\s n(\M,\A)$-fermé, $i_x : F\to\tau_x(F)$ est injectif, de sorte qu'on a une suite exacte
$${\rm Hom}(N,\tau_x(F))\to {\rm Hom}(N,\delta_x(F))\to {\rm Ext}^1(N,F)$$
déduite de la suite exacte courte $0\to F\xrightarrow{i_x}\tau_x(F)\to\delta_x(F)\to 0$. Si $N$ appartient à  $\s n(\M,\A)$, alors ${\rm Hom}(N,F)$ est nul, par conséquent ${\rm Hom}(N,\tau_x(F))$ est nul d'après le point précédent, et ${\rm Ext}^1(N,F)$ est nul par hypothèse sur $F$, d'où la nullité de ${\rm Hom}(N,\delta_x(F))$.
\item Par le lemme du serpent, la suite exacte de l'hypothèse induit une suite exacte
$$0\to\kappa_x(G)\to\kappa_x(F)\to\kappa_x(N)\to\delta_x(G)\to\delta_x(F)\to\delta_x(N)\to 0.$$
Puisque $F$ est  $\s n(\M,\A)$-fermé, $\kappa_x(F)$ est nul, et $\kappa_x\delta_x(F)$ l'est également grâce à ce qu'on vient de démontrer au point précédent et au lemme \ref{lm-stabnul}. Appliquant le foncteur exact à gauche $\kappa_x$ à la suite exacte $0\to\kappa_x(N)\to\delta_x(G)\to\delta_x(F)$ et utilisant le caractère idempotent de $\kappa_x$ (proposition~\ref{pr-dkt}.\ref{pid}), on obtient bien un isomorphisme $\kappa_x(N)\simeq(\kappa_x)^2(N)\xrightarrow{\simeq}\kappa_x\delta_x(G)$.
\end{enumerate}
\item \begin{enumerate}
       \item Il s'agit de montrer que, si $F$ est un objet $\s n(\M,\A)$-fermé de $\fct(\M,\A)$, alors $\tau_t(F)$ est $\s n(\M,\A)$-fermé. Soit $G$ le foncteur $s_\M\pi_\M(\tau_t F)$ : ce qu'on a démontré en 1.(a) implique que l'unité $\tau_t(F)\to G$ de l'adjonction entre $\pi_\M$ et $s_\M$ est un monomorphisme. Soit $N$ son conoyau : par le point 1.(c), pour tout objet
$x$ de $\M$, on a $\kappa_x(N)\simeq\kappa_x\delta_x\tau_t(F)$. On en déduit $\kappa_x(N)\simeq\tau_t\kappa_x\delta_x(F)$ par la proposition~\ref{pr-dkt}. Mais 1.(b) entraîne la nullité de $\kappa_x\delta_x(F)$, d'où $\kappa_x(N)=0$. Comme c'est vrai pour tout $x\in {\rm Ob}\,\M$ et que $N$ appartient à $\s n(\M,\A)$, on en déduit $N=0$, d'où notre assertion. 
\item C'est une conséquence directe du point 1.(b).
      \end{enumerate}

\end{enumerate}
\end{proof}

Cette proposition implique le résultat de compatibilité aux limites suivant des endofoncteurs $\tau_x$ et $\delta_x$ de $\st(\M,\A)$.

\begin{cor}\label{corlim}
 Soit $x$ un objet de $\M$.
\begin{enumerate}
 \item L'endofoncteur $\tau_x$ de $\st(\M,\A)$ commute aux limites.
\item Il existe un monomorphisme naturel $\delta_x(\underset{\T}{\lim}\,\Phi)\hookrightarrow\underset{\T}{\lim}\,(\delta_x\Phi)$ pour tout foncteur $\Phi$ d'une petite catégorie $\T$ vers $\st(\M,\A)$.
\end{enumerate}

\end{cor}

\begin{proof}
 L'assertion relative à $\tau_x$ résulte de ce que l'endofoncteur $\tau_x$ de $\fct(\M,\A)$ commute aux limites d'après la proposition \ref{pr-dkt}.\,\ref{p0p}. et de la proposition~\ref{sect}.2(a), les limites dans $\st(\M,\A)$ étant obtenues en appliquant le foncteur $s_\M$, en prenant la limite de $\fct(\M,\A)$ puis en appliquant le foncteur $\pi_\M$. 

L'assertion relative à $\delta_x$ s'en déduit en utilisant la suite exacte courte $0\to {\rm Id}\to\tau_x\to\delta_x\to 0$ d'endofoncteurs de $\st(\M,\A)$, qui permet de former un diagramme commutatif aux lignes exactes
$$\xymatrix{0\ar[r] & \underset{\T}{\lim}\,\Phi\ar[r]\ar@{=}[d] & \tau_x(\underset{\T}{\lim}\,\Phi)\ar[d]^\simeq\ar[r] & \delta_x(\underset{\T}{\lim}\,\Phi)\ar@{-->}[d]\ar[r] & 0\\
0\ar[r] & \underset{\T}{\lim}\,\Phi\ar[r] & \underset{\T}{\lim}\,(\tau_x\Phi)\ar[r] & \underset{\T}{\lim}\,(\delta_x\Phi) & 
} :$$
le carré commutatif de gauche fournit l'existence de la flèche verticale de droite (en pointillé), qui est un monomorphisme.
\end{proof}

\subsection{Objets polynomiaux dans $\st(\M,\A)$}

\begin{defi}[Objets polynomiaux de $\st(\M,\A)$] \label{pol-st(M,A)}
 On définit par récurrence une suite $(\Pol_n(\M,\A))_n$ de sous-catégories pleines de $\st(\M,\A)$ de la façon suivante :
\begin{enumerate}
 \item $\Pol_n(\M,\A)=\{0\}$ si $n<0$ ;
\item pour $n\geq 0$, $\Pol_n(\M,\A)$ est constituée des objets $X$ de $\st(\M,\A)$ tels que $\delta_x(X)$ appartienne à $\Pol_{n-1}(\M,\A)$ pour tout objet $x$ de $\M$.
\end{enumerate}

Un objet de $\st(\M,\A)$ est dit {\em polynomial} s'il appartient à une sous-catégorie $\Pol_n(\M,\A)$ ; son degré est le plus petit $n\in\mathbb{N}\cup\{-\infty\}$ pour lequel cela advient et on note $deg(F)$ le degré d'un objet $F$ de $\st(\M,\A)$.

Un foncteur de $\fct(\M,\A)$ est dit {\em faiblement polynomial} de degré faible au plus $n$ si son image par $\pi_\M$ dans $\st(\M,\A)$ appartient à $\Pol_n(\M,\A)$.
\end{defi}

De la commutation des foncteurs $\delta_x$ et $\pi_\M$, donnée par la proposition \ref{pr-fde}, on déduit facilement qu'un foncteur fortement polynomial de degré fort $n$ est faiblement polynomial de degré faible \textit{au plus} $n$. Par contre, certains foncteurs faiblement polynomiaux ne sont pas fortement polynomiaux, comme le montre l'exemple \ref{faible-pas-fort}. Ceci justifie la terminologie introduite.

Lorsque la catégorie $\M$ appartient à $\mn$, $\st(\M,\A)=\fct(\M,\A)$, de sorte qu'il n'y a qu'une seule notion de foncteur polynomial :

\begin{pr} \label{eq-def-nul}
Lorsque la catégorie $\M$ appartient à $\mn$, un foncteur $\M\to\A$ est faiblement polynomial de degré $n$ si et seulement s'il est fortement polynomial de degré $n$.
\end{pr}

Dans la définition \ref{pol-st(M,A)}, on peut se contenter d'appliquer les foncteurs $\delta_x$ pour certains objets $x$ de $\M$ :
\begin{pr}\label{pr-dpd}
 Si $E$ est un ensemble d'objets de $\M$ tel que tout objet de $\M$ est isomorphe à une somme (au sens de $\oplus$) finie d'éléments de $E$, alors un objet $X$ de $\st(\M,\A)$ appartient à $\Pol_n(\M,\A)$ si $\delta_x(X)$ appartient à $\Pol_{n-1}(\M,\A)$ pour tout $x\in E$.
\end{pr}

\begin{proof}
 La démonstration est la même que celle de la deuxième partie de la proposition~\ref{pr-polfor}.
\end{proof}

La proposition suivante montre que, dans la catégorie quotient $\st(\M,\A)$, les sous-catégories d'objets polynomiaux possèdent les mêmes propriétés de régularité que dans le cadre usuel que nous rappellerons brièvement au §\,\ref{smn}, contrairement aux foncteurs fortement polynomiaux.

\begin{pr}\label{propol-gal}
 Pour tout $n\in\mathbb{N}$, la sous-catégorie $\Pol_n(\M,\A)$ de $\st(\M,\A)$ est bilocalisante, c'est-à-dire épaisse et stable par limites et colimites. 
 
De plus, pour tout objet $x$ de $\M$ et tout objet $F$ de $\Pol_n(\M,\A)$, $\tau_x(F)$  est polynomial de degré au plus $n$.
\end{pr}

\begin{proof}
 Le caractère localisant (i.e. épais et stable par colimites) de $\Pol_n(\M,\A)$ provient de ce que les endofoncteurs $\delta_t$ sont exacts et commutent aux colimites  d'après la proposition~\ref{pr-fde}. La stabilité par limites se déduit de la deuxième assertion du corollaire~\ref{corlim}.

La dernière assertion découle des propriétés de commutation données dans la proposition~\ref{pr-fde}.
\end{proof}

La première partie de cette proposition peut se reformuler en disant qu'il existe un {\em diagramme de recollement}
$$\xymatrix{\Pol_{n-1}(\M,\A)\ar[r] & \Pol_n(\M,\A)\ar[r]\ar@/_/[l]\ar@/^/[l] & \Pol_n(\M,\A)/\Pol_{n-1}(\M,\A)\ar@/_/[l]\ar@/^/[l]
}.$$

Les sections~\ref{section2} et \ref{sherm} nous permettront de décrire, dans certains cas importants, la catégorie quotient qui apparaît à droite de ce diagramme.

\subsection{\'Etude des foncteurs polynomiaux de bas degré} \label{bas-degré}
La description des objets de degré $0$ est sans surprise :

\begin{pr}[Objets de degré $0$ de $\st(\M,\A)$]\label{fdeg0}
 Le foncteur
$$\A\xrightarrow{c}\fct(\M,\A)\xrightarrow{\pi_\M}\st(\M,\A),$$
où la flèche $c$ est l'inclusion des foncteurs constants, induit une équivalence de catégories
$$\A\simeq\Pol_0(\M,\A)$$
dont un quasi-inverse est induit par le foncteur section.
\end{pr}

On rappelle qu'un {\em quasi-inverse} d'un foncteur $T$ est un foncteur $U$ tel que les foncteurs composés $T\circ U$ et $U\circ T$ soient isomorphes aux identités.

\begin{proof}
 Comme $\M$ possède un objet initial $0$, le foncteur $c$ est adjoint à gauche au foncteur $\fct(\M,\A) \to \A$ d'évaluation en $0$. Cela montre en particulier que $c$ est pleinement fidèle. Maintenant, grâce au troisième point du lemme~\ref{lm-stabnul}, on voit que l'application
$${\rm Hom}_{\fct(\M,\A)}\big(F,c(A)\big)\to {\rm Hom}_{\st(\M,\A)}\big(\pi_\M(F),\pi_\M(c(A))\big)$$
induite par $\pi_\M$ est un isomorphisme pour tout objet $F$ de $\fct(\M,\A)$, puis que $\pi_\M\circ c$ est pleinement fidèle.

Il est clair que ce foncteur prend ses valeurs dans $\Pol_0(\M,\A)$ ; vérifions que son image essentielle est exactement $\Pol_0(\M,\A)$. En effet, si $X$ est un objet de $\Pol_0(\M,\A)$ et $x$ un objet de $\M$, $\delta_x(X)$ est nul, donc $\delta_x s_\M(X)$, qui est un sous-objet de $s_\M\delta_x(X)$ d'après la dernière assertion de la proposition~\ref{sect}, est également nul. Autrement dit, posant $F=s_\M(X)$, $i_x(F) : F\to\tau_x(F)$ est surjectif. Par le deuxième point du lemme~\ref{lm-stabnul}, le morphisme est également injectif, c'est donc un isomorphisme. En évaluant en $0$, on en déduit que $F(0)\to F(x)$ est un isomorphisme, ce qui signifie bien que $F$ est constant.
\end{proof}

Ainsi, l'image par le foncteur section d'un objet polynomial de degré nul est un foncteur fortement polynomial de degré fort nul. En revanche, l'assertion analogue pour les degrés supérieurs est fausse, dès le degré $1$. Ainsi, dans l'exemple~\ref{exdeg} ci-après, on exhibe un objet de degré $1$ de $\st(\Theta,\mathbf{Ab})$ dont l'image par $s_\Theta$ est de degré fort $2$. Le phénomène illustré par cet exemple est d'autant plus frappant que le fait que le degré fort de l'image par le foncteur section d'un objet de $\Pol_1(\M,\A)$ soit strictement supérieur à $1$ implique qu'aucun objet de $\Pol^{{\rm fort}}_1(\M,\A)$ n'a une image par $\pi_\M$ isomorphe à cet objet, comme le montre la proposition suivante.

\begin{pr}\label{rq-pol}
 Soient $X$ un objet de $\Pol_1(\M,\A)$ et $F$ le foncteur $s_\M(X)$. Les assertions suivantes sont équivalentes.
 \begin{enumerate}
  \item\label{pp1} Il existe un foncteur $G : \M\to\A$ fortement polynomial de degré fort au plus $1$ tel que $\pi_\M(G)\simeq X$ ;
  \item\label{pp2} $F$ est fortement polynomial de degré fort au plus $1$ ;
  \item\label{pp3} il existe dans $\st(\M,\A)$ une suite exacte $0\to\pi_M(C)\to X\to\pi_\M(A)\to 0$, où $C : \M\to\A$ est un foncteur constant et $A$ un foncteur monoïdal fort ;
  \item\label{pp4} il existe dans $\fct(\M,\A)$ une suite exacte $0\to C\to F\to A\to 0$, où $C : \M\to\A$ est un foncteur constant et $A$ un foncteur monoïdal fort.
 \end{enumerate}
\end{pr}

\begin{proof}
 {\em \ref{pp1}.}$\Rightarrow${\em \ref{pp2}.} Quitte à remplacer $G$ par son image par l'unité $G\to F$ de l'adjonction entre $s_\M$ et $\pi_\M$, on peut supposer que celle-ci est injective. En effet, cette image est de degré fort au plus $1$ puisque $\Pol^{{\rm fort}}_1(\M,\A)$ est stable par quotients. Notons $N$ le conoyau de l'unité : on a donc une suite exacte courte $0\to G\to F\to N\to 0$ avec $F$ $\s n(\M,\A)$-fermé et $N$ dans $\s n(\M,\A)$.

Soient $x$ et $t$ des objets de $\M$. Notons $H$ l'image du morphisme $\delta_x(G)\to\delta_x(F)$ induit par l'inclusion $G\to F$. La suite exacte courte $0\to H\to\delta_x(F)\to\delta_x(N)\to 0$ induit une suite exacte
$$\kappa_t\delta_x(F)\to\kappa_t\delta_x(N)\to\delta_t(H).$$
Mais $\delta_t(H)$, quotient de $\delta_t\delta_x(G)$, est nul puisque $G$ est par hypothèse de degré fort $1$. Quant à $\kappa_t\delta_x(F)$, il est nul par la proposition~\ref{sect}. Donc $\kappa_t\delta_x(N)$ est nul pour tout $t$ et tout $x$. Comme $\delta_x(N)$ est stablement nul comme $N$ (c'est un quotient de $\tau_x(N)$, et la commutation des foncteurs $\kappa_a$ et $\tau_x$ implique que $\s n(\M,\A)$ est stable par $\tau_x$), on en déduit $\delta_x(N)=0$, i.e. que $N$ est fortement polynomial de degré fort (au plus) $0$. Par conséquent, la suite exacte courte $0\to G\to F\to N\to 0$ montre que $F$ est fortement polynomial de degré fort au plus $1$ (utiliser la proposition~\ref{pr-polfor}).
 
 {\em \ref{pp2}.}$\Rightarrow${\em \ref{pp4}.} Pour tout objet $x$ de $\M$, $\delta_x(F)$ est un foncteur polynomial fort de degré au plus $0$, donc un quotient d'un foncteur constant (proposition~\ref{fort-0}). Par ailleurs, en utilisant le point 1.(b) de la proposition~\ref{sect}, on voit que le morphisme canonique  $\delta_x(F)\to s_\M\pi_\M(\delta_x(F))$ est un monomorphisme. La proposition~\ref{fdeg0} montre également que $s_\M\pi_\M(\delta_x(F))$ est un foncteur constant. Ainsi, $\delta_x(F)$ est l'image d'un morphisme entre foncteurs constants, il est donc lui-même constant, disons en $A(x)$. On dispose donc d'une suite exacte courte
 $$0\to F(0)\to F(x)\to\delta_x(F)(0)=A(x)\to 0$$
 naturelle en l'objet $x$ de $\M$. Il s'agit de montrer que $A$ est monoïdal fort.
 
 Si $y$ est un objet de $\M$, l'épimorphisme
 $$A(x\oplus y)=\delta_{x\oplus y}(F)(0)=F(x\oplus y)/F(0)\twoheadrightarrow F(x\oplus y)/F(y)=\delta_x(F)(y)=A(x)$$
 est une rétraction du morphisme canonique $A(x)\to A(x\oplus y)$. Combinée à la suite exacte courte
 $$0\to A(x)=F(x)/F(0)\to A(x\oplus y)=F(x\oplus y)/F(0)\to F(x\oplus y)/F(x)=A(y)\to 0,$$
 cette observation procure un isomorphisme naturel monoïdal $A(x)\oplus A(y)\simeq A(x\oplus y)$.

 L'implication {\em \ref{pp3}.}$\Rightarrow${\em \ref{pp4}.} vient de ce que le morphisme canonique
  $${\rm Ext}^1_{\fct(\M,\A)}(A,C)\to {\rm Ext}^1_{\st(\M,\A)}(\pi_\M(A),\pi_\M(C))$$
  est un isomorphisme, en vertu de la troisième assertion du lemme~\ref{lm-stabnul}.
  
  L'implication {\em \ref{pp4}.}$\Rightarrow${\em \ref{pp2}.} découle de la stabilité par extensions de $\Pol^{{\rm fort}}_1(\M,\A)$.
  
Les implications {\em \ref{pp4}.}$\Rightarrow${\em \ref{pp3}.} et {\em \ref{pp2}.}$\Rightarrow${\em \ref{pp1}.} sont immédiates.  
 
 \end{proof}

\begin{rem}
Il semble difficile d'obtenir des généralisations de la proposition~\ref{rq-pol} aux degrés polynomiaux $d$ supérieurs, ne serait-ce que $2$. La description des objets de $\st(\M,\A)$ possédant un représentant dans $\Pol^{{\rm fort}}_d(\M,\A)$ est certainement plus complexe que celle donnée par l'implication {\em \ref{pp1}.}$\Rightarrow${\em \ref{pp3}.} pour $d=1$. La proposition~\ref{pr-filtr} donnera toutefois une façon de construire tous les objets de $\Pol_d(\M,\A)$ à partir d'images par $\pi_\M$ de foncteurs de $\Pol_d^{{\rm fort}}(\M,\A)$ particuliers mais sans supposer qu'existe un représentant dans $\Pol_d^{{\rm fort}}(\M,\A)$, de sorte que le résultat diffère, même pour $d=1$, de celui de la proposition~\ref{rq-pol}. Nous ignorons si l'implication {\em \ref{pp1}.}$\Rightarrow${\em \ref{pp2}.} s'étend aux degrés forts supérieurs.
\end{rem}

\begin{rem} \label{lien-pol-polfort}
Plusieurs résultats nuancent le phénomène notable qu'un objet de $\Pol_1(\M,\A)$ peut n'avoir aucun représentant dans $\Pol_1^{{\rm fort}}(\M,\A)$. L'un est donné par le corollaire~\ref{rq-ptil} ci-après.

Un autre est fourni par l'un des résultats principaux de \cite[proposition~6.4]{Dja-pol}, qui montre que, sous une hypothèse de finitude faible sur $\M$, pour tout objet $X$ de $\Pol_n(\M,\A)$, il existe un objet $a$ de $\M$, dépendant uniquement de $\M$ et $n$, tel que $\tau_a(s_\M(X))$ soit {\em fortement} polynomiale de degré (fort) au plus $n$.

Par conséquent, on en déduit, en utilisant la remarque~\ref{rqsne}, que, sous de légères hypothèses sur $\M$,  tout objet polynomial de degré $n$ de $\st(\M,\A)$ :
\begin{itemize}
 \item possède, {\em à un décalage près}, un représentant de degré fort $n$ ;
 \item possède un représentant de degré fort $n$ {\em à un nombre fini de valeurs près} ;
 \item a pour image par le foncteur section un foncteur {\em fortement} polynomial, d'un degré en général strictement supérieur à $n$, mais qu'on peut borner par un entier ne dépendant que de $n$ et $\M$.
\end{itemize}

Toutes ces propriétés sont vérifiées dès que $\M$ vérifie les propriétés $(P_0)$ et $(P_1)$ de \cite[Définition 3.2]{Dja-pol}, ce qui est le cas pour les catégories sources usuelles comme $\Theta$, $\mathbf{S}(\mathbf{ab})$ ou $\mathbf{M}(\mathbf{ab})$.
\end{rem}

\begin{rem} Anticipant sur la section~\ref{section2}, sous les hypothèses de la proposition~\ref{rq-pol}, les foncteurs $C$ et $A$ se factorisent par le foncteur canonique $\M\to\widetilde{\M}$. Il n'est néanmoins généralement pas possible de trouver un représentant de $X$ possédant une telle factorisation. En effet, comme $\widetilde{\M}$ possède un objet nul, l'existence d'une telle factorisation implique un scindement $X\simeq\pi_\M(C)\oplus\pi_\M(A)$, ce qui n'est pas toujours le cas. Un exemple est donné comme suit. Soient $k$ un corps et $\M$ la catégorie dont les objets sont les couples $(V,v)$ où $V$ est un $k$-espace vectoriel de dimension finie et $v$ un élément non nul de $V$. Les morphismes $(U,u)\to (V,v)$ sont les applications $k$-linéaires $f : U\to V$ telles que $f(u)=v$. Une variante consistant à ne considérer que les monomorphismes linéaires qui vérifient cette propriété fonctionne pareillement. C'est un objet de $\mi$  pour la structure monoïdale donnée par $(U,u)\oplus (V,v)=(U\underset{k}{\oplus}V,\bar{u}=\bar{v})$ où la somme amalgamée est relative aux morphismes $k\to U$ et $k\to V$ donnés par $u$ et $v$ respectivement et où  $\bar{u}$ et $\bar{v}$ désignent les images respectives de $u$ et $v$ dans cette somme. Notons $F$ le foncteur d'oubli $(V,v)\mapsto V$ de $\M$ vers les $k$-espaces vectoriels et $A$ le quotient de $F$ donné par $(V,v)\mapsto V/<v>$. Alors le foncteur $A$ est fortement monoïdal, et l'on dispose d'une suite exacte courte $0\to k\to F\to A\to 0$ : on est dans la situation de la proposition~\ref{rq-pol}. On vérifie toutefois aisément que cette suite exacte courte ne se scinde pas, même après application du foncteur $\pi_\M$, ce qui ne change rien aux groupes d'extensions entre $k$ et $A$ en raison du lemme~\ref{lm-stabnul}.\,\ref{lsp3}. 
\end{rem}

\section{Effets croisés}\label{sef}

La définition usuelle des foncteurs polynomiaux, telle que donnée par Eilenberg et Mac Lane, par exemple, repose sur la notion d'{\em effets croisés}. Elle fonctionne très bien pour des foncteurs définis sur un objet de $\mn$, et s'étend sans difficulté aux foncteurs polynomiaux {\em forts} sur un objet de $\mi$, comme on va le voir, faisant le lien entre la section précédente et les notions classiques de polynomialité. En revanche, dès lors qu'on traite des catégories quotients $\st(\M,\A)$ et de foncteurs faiblement polynomiaux, le maniement des foncteurs différences $\delta_x$ s'avère bien plus efficace.

\subsection{Définition, lien avec les foncteurs polynomiaux}\label{sdef}

Soit $\M$ un objet de $\mi$.

Pour toute famille finie $(a_i)_{i\in E}$ d'objets de $\M$ et tout sous-ensemble $I$ de $E$, on dispose d'un morphisme canonique
\begin{equation}\label{er}
r^E_I(\mathbf{a}) : \bigoplus_{i\in I}a_i\Big(\simeq\bigoplus_{i\in E} a'_i\Big)\to\bigoplus_{i\in E} a_i
\end{equation}
où $a'_i=a_i$ si $i\in I$, $0$ sinon et où le morphisme est induit par les $a'_i\to a_i$ égaux à l'identité si $i\in I$ et à l'unique morphisme sinon.

\begin{defi}[Effets croisés]\label{ecr-gal}
Soient $\M$ un objet de $\mi$, $\A$ une catégorie additive possédant des colimites finies et $F : \M\to\A$ un foncteur. Pour toute famille finie $\mathbf{a}=(a_i)_{i\in E}$ d'objets de $\M$, on note
$$cr_E(F)(\mathbf{a})={\rm Coker}\,\Big(\bigoplus_{i\in E}F\Big(\bigoplus_{j\in E\setminus\{i\}}a_j\Big)\to F\Big(\bigoplus_{i\in E}a_i\Big)\Big)$$ 
le morphisme dont les composantes sont les $F(r^E_{E\setminus\{i\}}(\mathbf{a}))$.

Lorsque $E=\mathbf{d}$ pour un $d\in\mathbb{N}$, on note $cr_d(F)(a_1,\dots,a_d)$ pour $cr_E(F)(\mathbf{a})$. Le foncteur $cr_d(F) : \M^d\to\A$ ainsi obtenu s'appelle {\em $d$-ème effet croisé de $F$}.
\end{defi}

\begin{pr}\label{pr-ecr}
  Soient $\M$ un objet de $\mi$, $\A$ une catégorie additive avec colimites finies, $F : \M\to\A$ un foncteur, $i$ et $j$ des entiers naturels et $x_1,\dots,x_i,a_1,\dots,a_j,t$ des objets de $\M$. Alors :
\begin{enumerate}
 \item $cr_i(F)(x_1,\dots,x_i)$ est nul si l'un des objets $x_r$ est nul ;
 \item\label{pcn} pour tout $d\in\mathbb{N}$, le foncteur $cr_d : \fct(\M,\A)\to\fct(\M^d,\A)$ commute aux colimites ;
\item pour toute permutation $\sigma\in\Sigma_i$, l'isomorphisme entre $\bigoplus_{r=1}^i x_r$ et $\bigoplus_{r=1}^i x_{\sigma(r)}$ donné par la structure monoïdale symétrique induit un isomorphisme entre $cr_i(F)(x_1,\dots,x_i)$ et $cr_i(F)(x_{\sigma(1)},\dots,x_{\sigma(i)})$ ;
\item\label{pfn} il existe un isomorphisme naturel
$$cr_{i+j}(F)(x_1,\dots,x_i,a_1,\dots,a_j)\simeq cr_i\big(cr_{j+1}(F)(-,a_1,\dots,a_j)\big)(x_1,\dots,x_i)\,;$$
\item\label{pecrd} il existe un isomorphisme naturel
$$cr_{i+1}(F)(t,x_1,\dots,x_i)\simeq cr_i(\delta_t(F))(x_1,\dots,x_i).$$
\end{enumerate}
\end{pr}

\begin{proof}
 Les trois premières assertions sont immédiates.
 
 Pour l'assertion~{\em \ref{pecrd}.}, on note que $cr_i(\tau_t(F))(x_1,\dots,x_i)$ est le conoyau de l'application canonique
 $$\bigoplus_{k=1}^i F\Big(t\oplus\underset{r\neq k}{\bigoplus} x_r\Big)\to F\Big(t\oplus\underset{1\leq r\leq i}{\bigoplus} x_r\Big),$$
 de sorte que $cr_i(\delta_t(F))(x_1,\dots,x_i)$ est le conoyau de l'application canonique
 $$F\Big(\underset{1\leq r\leq i}{\bigoplus} x_r\Big)\oplus\bigoplus_{k=1}^i F\Big(t\oplus\underset{r\neq k}{\bigoplus} x_r\Big)\to F\Big(t\oplus\underset{1\leq r\leq i}{\bigoplus} x_r\Big),$$
 c'est-à-dire $cr_{i+1}(F)(t,x_1,\dots,x_i)$.
 
 L'assertion~{\em \ref{pfn}.} s'établit de façon analogue, en écrivant
 $$cr_i\big(cr_{j+1}(F)(-,a_1,\dots,a_j)\big)(x_1,\dots,x_i)$$
comme quotient explicite de $cr_i(\tau_{a_1\oplus\dots\oplus a_j}(F))(x_1,\dots,x_i)$.
\end{proof}

\begin{pr}\label{p-ecr}
 Soient $\M$ un objet de $\mi$, $\A$ une catégorie abélienne, $F$ un foncteur de $\fct(\M,\A)$ et $d\in\mathbb{N}$. Alors $F$ appartient à $\Pol_d^{{\rm fort}}(\M,\A)$ si et seulement si le foncteur $cr_{d+1}(F) : \M^{d+1}\to\A$ est nul.
\end{pr}

\begin{proof}
 On procède par récurrence sur $d$. Pour $d=0$, cela provient de la proposition~\ref{fort-0}, puisque $cr_1(F)(x)$ est le conoyau du morphisme canonique $F(0)\to F(x)$. Le pas de la récurrence découle de l'assertion~{\em \ref{pecrd}.} de la proposition~\ref{pr-ecr}.
\end{proof}

\begin{rem}
 {\em Si $d$ est strictement supérieur à $1$}, on peut vérifier que les foncteurs $cr_d : \fct(\M,\A)\to\fct(\M^d,\A)$ induisent des foncteurs $\st(\M,\A)\to\st(\M^d,\A)$, de façon analogue à ce qu'on a fait pour les foncteurs différences --- cf. proposition~\ref{pr-fde}. Puis on peut montrer qu'un objet de $\fct(\M,\A)$ appartient à $\Pol_d(\M,\A)$, pour $d>1$, si et seulement si son image dans $\st(\M^{d+1},\A)$ par l'effet croisé est nulle. Ce résultat est plus difficile à démontrer que la proposition~\ref{p-ecr}, ne serait-ce qu'à cause du problème d'initialisation. D'une manière générale, les effets croisés semblent moins aisés à manipuler dans les catégories du type $\st(\M,\A)$ que les foncteurs différences, c'est pourquoi nous avons privilégié l'utilisation de ces derniers, malgré le caractère plus <<~symétrique~>> des définitions en termes d'effets croisés.
\end{rem}

\begin{cor}\label{poles}
 Soient $\M$ un objet de $\mi$, $\A$ une catégorie abélienne, $d$ un entier, $E$ un ensemble fini de cardinal supérieur à $d$, $F$ un foncteur de $\Pol_d^{{\rm fort}}(\M,\A)$ et $(x_i)_{i\in E}$ une famille d'objets de $\M$. Alors le morphisme canonique
 $$\bigoplus_{I\in\mathcal{P}_d(E)}F\Big(\bigoplus_{i\in I}x_i\Big)\to F\Big(\bigoplus_{i\in E}x_i\Big)$$
 est un épimorphisme, où $\mathcal{P}_d(E)$ désigne l'ensemble des parties à $d$ éléments de $E$.
\end{cor}

\begin{proof}
 L'assertion est vide lorsque le cardinal de $E$ est $d$ et équivaut à la nullité de $cr_{d+1}(F)$ lorsque $E$ a $d+1$ éléments, elle est donc alors donnée par la proposition~\ref{p-ecr}. Le cas général s'obtient par récurrence, en notant qu'on a $cr_j(F)=0$ non seulement pour $j=d+1$, mais aussi pour tout entier $j>d$, grâce à l'assertion~{\em \ref{pfn}.} de la proposition~\ref{pr-ecr}.
\end{proof}

\subsection{Cas d'une source dans $\mn$}\label{smn}

L'étude des foncteurs polynomiaux depuis une catégorie dans $\mn$ vers une catégorie abélienne constitue un sujet classique. Il a été inauguré, entre catégories de modules, dès le début des années 1950 par le travail d'Eilenberg et Mac Lane \cite[chap.~II]{EML}. On trouvera par exemple dans \cite[section~2]{HPV} une définition et des propriétés générales des foncteurs polynomiaux d'une catégorie monoïdale, même pas nécessairement symétrique, dont l'unité est objet nul vers une catégorie abélienne. La situation d'une catégorie source avec objet nul et possédant des coproduits finis est aussi abordée dans \cite[§\,1]{HV}.

Signalons également qu'une manière de généraliser le cadre usuel pour les foncteurs polynomiaux consiste, tout en conservant une catégorie source dans $\mn$, à considérer une catégorie but non abélienne, comme la catégorie des groupes. Cette généralisation, orthogonale à celle dont traite  le présent article, a notamment été étudiée et utilisée par Pirashvili \cite{Pira82a, Pira82b, Pira99} et Baues-Pirashvili \cite{Baues-Pira}.

\begin{defi}\label{epsilon}
 Soient $\M$ un objet de $\mn$, $E$ un ensemble fini, $I$ une partie de $E$ et $\mathbf{a}=(a_i)_{i\in E}$ une famille d'objets de $\M$. On note $\epsilon_I(\mathbf{a})$ l'endomorphisme $\underset{i\in E}{\bigoplus}f_i$ de l'objet $\underset{i\in E}{\bigoplus}a_i$ de $\M$, où $f_i$ est l'endomorphisme identique de $a_i$ lorsque $i\in I$ et l'endomorphisme nul de $a_i$ lorsque $i\notin I$.
\end{defi}

Lorsqu'aucune ambiguïté ne peut en résulter, on notera $\epsilon_I$ pour $\epsilon_I(\mathbf{a})$.

\begin{pr}\label{pr-eps}
  Soient $\M$ un objet de $\mn$, $E$ un ensemble fini et $\mathbf{a}=(a_i)_{i\in E}$ une famille d'objets de $\M$. Les endomorphismes $\epsilon_I(\mathbf{a})$ ($I\subset E$) de $x=\underset{i\in E}{\bigoplus}a_i$ forment une famille commutative d'idempotents. De plus :
\begin{enumerate}
 \item $\epsilon_E=1_x$ et $\epsilon_\emptyset=0_x$ ;
\item $\epsilon_I.\epsilon_J=\epsilon_{I\cap J}$ pour tous sous-ensembles $I$ et $J$ de $E$.
\end{enumerate}
\end{pr}

Cette proposition est immédiate ; l'énoncé qui suit s'en déduit en utilisant un argument classique d'inversion de Möbius  \cite[theorem~3.9.2 ]{Stan}  \cite[théorème~C.1]{DV}.

\begin{cor}\label{cor-eps}
 Sous les mêmes hypothèses, dans l'anneau $\mathbb{Z}[\M(x,x)]$ du monoïde $\M(x,x)$, posons, pour $I\subset E$,
$$e_I=e_I(\mathbf{a})=\sum_{J\subset I}(-1)^{|I|-|J|}[\epsilon_J(\mathbf{a})]$$
où $|J|$ désigne le cardinal de l'ensemble fini $J$. Alors les $e_I$ forment, lorsque $I$ parcourt les parties de $E$, une famille complète d'idempotents orthogonaux, c'est-à-dire : $e_I.e_J=0$ pour $I\neq J$ et la somme de tous les $e_I$ est égale à $1$.
\end{cor}

\begin{pr}\label{df2-ecr}
Soient $\M$ un objet de $\mn$, $\A$ une catégorie abélienne, $E$ un ensemble fini, $\mathbf{a}=(a_i)_{i\in E}$ une famille d'objets de $\M$ et $F : \M\to\A$ un foncteur.

On dispose d'un isomorphisme naturel
$$cr_E(F)(\mathbf{a})\simeq {\rm Im}\,F(e_E(\mathbf{a})),$$
où $e_E(\mathbf{a})$ est l'idempotent introduit dans le corollaire précédent.
\end{pr}

\begin{proof}
 Pour toute partie $I$ de $E$, l'image de l'idempotent $F(\epsilon_I(\mathbf{a}))$ égale l'image du morphisme canonique de $F\Big(\underset{i\in I}{\bigoplus}a_i\Big)$ dans $F\Big(\underset{i\in E}{\bigoplus}a_i\Big)$. Comme le noyau de l'idempotent $F(e_E(\mathbf{a}))$ s'identifie à la somme des images des idempotents $F(\epsilon_I(\mathbf{a}))$ lorsque $I$ parcourt les parties {\em strictes} de $E$ \cite[§\,3.9]{Stan}, on en déduit la proposition.
\end{proof}

Par conséquent, notre définition des foncteurs polynomiaux coïncide, pour une source dans $\mn$, avec la notion usuelle, présentée dans le cas général dans \cite{HPV}.

La proposition précédente montre également le caractère auto-dual des effets croisés d'une catégorie $\M$ de $\mn$ vers une catégorie abélienne $\A$, c'est-à-dire le fait que l'isomorphisme de catégories canonique $\fct(\M,\A)^{op}\simeq\fct(\M^{op},\A^{op})$ préserve les effets croisés. On en tire en particulier le résultat suivant :
\begin{cor}
 Soient $\M$ un objet de $\mn$, $\A$ une catégorie abélienne et $d\in\mathbb{N}$. 
 \begin{enumerate}
  \item Il existe un isomorphisme
  $$cr_d(F)(a_1,\dots,a_d)\simeq {\rm Ker}\,\Big(F\Big(\bigoplus_{i=1}^d a_i\Big)\to\bigoplus_{i=1}^d F\Big(\bigoplus_{j\neq i}a_j\Big)\Big)$$
  naturel en le foncteur $F : \M\to\A$ et les objets $a_1,\dots,a_d$ de $\M$.
  \item Le foncteur $cr_d : \fct(\M,\A)\to\fct(\M^d,\A)$ est exact.
 \end{enumerate}
\end{cor}

Une conséquence immédiate du corollaire~\ref{cor-eps} et de la proposition~\ref{df2-ecr} est l'importante propriété suivante.
\begin{pr}\label{dec-ecr}
 Soient $\M$ un objet de $\mn$, $\A$ une catégorie abélienne, $E$ un ensemble fini, $\mathbf{a}=(a_i)_{i\in E}$ une famille d'objets de $\M$ et $F : \M\to\A$ un foncteur. On dispose dans $\A$ d'un isomorphisme
$$F\Big(\bigoplus_{i\in E}a_i\Big)\simeq\bigoplus_{I\subset E}cr_I(F)(\mathbf{a}|_I)$$
 naturel en $F$ et en $\mathbf{a}$.
\end{pr}

Les effets croisés constituent un outil commode pour aborder le problème de la polynomialité d'une composée de foncteurs polynomiaux.

\begin{pr} \label{deg-compo}
Soient $\M$ un objet de $\mi$, $\A$ et $\B$ des catégories abéliennes, $F : \M \to \A$ et $X : \A \to \B$ des foncteurs, $d$ et $n$ des entiers naturels. On suppose que $X$ est polynomial de degré au plus $n$ et que $F$ est {\em fortement} polynomial de degré au plus $d$. Alors $X \circ F$ est fortement polynomial de degré au plus $nd$ dès lors que l'une des deux propriétés suivantes est satisfaite :
\begin{itemize}
 \item le foncteur $X$ préserve les épimorphismes ;
 \item la catégorie $\M$ appartient à $\mn$.
\end{itemize}
\end{pr}

\begin{proof}
  Soit $(x_i)_{i\in E}$ une famille d'objets de $\M$, où $E$ est un ensemble à $nd+1$ éléments. Par le corollaire~\ref{poles} (dont on conserve la notation $\mathcal{P}_d$), le morphisme canonique
  \begin{equation}\label{eeecr}
\bigoplus_{I\in\mathcal{P}_d(E)}F\Big(\bigoplus_{i\in I}x_i\Big)\to F\Big(\bigoplus_{i\in E}x_i\Big)   
  \end{equation}
  est un épimorphisme. En appliquant le foncteur $X$, on obtient, sous nos hypothèses, que le morphisme canonique
 \begin{equation}\label{eint}
 X\Big(\bigoplus_{I\in\mathcal{P}_d(E)}F\Big(\bigoplus_{i\in I}x_i\Big)\Big)\to (X\circ F)\Big(\bigoplus_{i\in E}x_i\Big)
  \end{equation}
  est un épimorphisme. Lorsque $\M$ appartient à $\mn$, cela provient du fait que l'épimorphisme~(\ref{eeecr}) est {\em scindé}, en appliquant la décomposition de la proposition~\ref{dec-ecr}.
  
  Du fait que $X$ appartient à $\Pol_n(\A,\B)$, le corollaire~\ref{poles} montre également que le morphisme canonique
\begin{equation}\label{eint2}
\bigoplus_{J\in\mathcal{P}_n(\mathcal{P}_d(E))}X\Big(\bigoplus_{I\in J}F\Big(\bigoplus_{i\in I}x_i\Big)\Big)\to X\Big(\bigoplus_{I\in\mathcal{P}_d(E)}F\Big(\bigoplus_{i\in I}x_i\Big)\Big)
  \end{equation}
  est un épimorphisme. La composée de~(\ref{eint}) et~(\ref{eint2}) montre que le morphisme canonique
  $$\underset{J\in\mathcal{P}_n(\mathcal{P}_d(E))}{\bigoplus}X\Big(\bigoplus_{I\in J}F\Big(\bigoplus_{i\in I}x_i\Big)\Big)\to (X\circ F)\Big(\bigoplus_{i\in E}x_i\Big)$$
  est un épimorphisme. Comme celui-ci s'insère dans un diagramme commutatif
  $$\xymatrix{\underset{J\in\mathcal{P}_n(\mathcal{P}_d(E))}{\bigoplus}X\Big(\underset{I\in J}{\bigoplus}F\Big(\underset{i\in I}{\bigoplus}x_i\Big)\Big)\ar[r]\ar[d] & (X\circ F)\Big(\underset{i\in E}{\bigoplus}x_i\Big) \\
 \underset{J\in\mathcal{P}_n(\mathcal{P}_d(E))}{\bigoplus}X\Big(F\Big(\underset{i\in\cup J}{\bigoplus}x_i\Big)\Big) \ar@{^(->}[r] & \underset{K\in\mathcal{P}_{nd}(E)}{\bigoplus}(X\circ F)\Big(\underset{i\in K}{\bigoplus}x_i\Big)\ar[u]
  }$$
dont la flèche verticale de droite est celle dont le conoyau définit $cr_E(X\circ F)$,  on en déduit $cr_{nd+1}(X\circ F)=0$, ce qui donne, par la proposition~\ref{p-ecr}, la conclusion.
\end{proof}

\begin{rem}
 \begin{enumerate}
  \item La conclusion tombe grossièrement en défaut si aucune des deux hypothèses de la fin de l'énoncé n'est vérifiée. À titre d'exemple, pour $r$ un entier naturel fixé, considérons les foncteurs
  $X:={\rm Hom}\,(\mathbb{Z}/2,-) : \mathbf{Ab}\to\mathbf{Ab}$ et $F_r : \Theta\to\mathbf{Ab}$ donné par $F_r(E)=\mathbb{Z}$ si le cardinal de $E$ est strictement inférieur à $r$, $\mathbb{Z}/2$ sinon. L'effet sur les morphismes est le suivant : pour toute flèche $A\to B$ de $\Theta$, $F_r(A)\to F_r(B)$ est l'épimorphisme canonique. Alors $F_r$ est fortement polynomial de degré $0$, d'après la proposition~\ref{fort-0}, tandis que $X$, additif, est de degré $1$. Toutefois, la composée $X\circ F_r$ prend la valeur $0$ sur les ensembles de cardinal strictement inférieur à $r$ et la valeur $\mathbb{Z}/2$ sur les autres ensembles finis, d'où l'on déduit aisément (cf. exemple~\ref{ex-degfort} ci-après) que $X\circ F_r$ est polynomial de degré fort exactement $r$.
  \item Dans l'exemple précédent, on note néanmoins que $X\circ F_r$ a le degré faible attendu, $0$. Il serait intéressant de savoir si l'on peut obtenir des énoncés analogues à la proposition~\ref{deg-compo} pour des foncteurs {\em faiblement} polynomiaux.
 \end{enumerate}
\end{rem}

\section{L'adjoint à gauche de l'inclusion $\mn\hookrightarrow\mi$} \label{section2}

Le but de cette section est de donner une description explicite de l'adjoint à gauche de l'inclusion $\mn\hookrightarrow\mi$ et d'en étudier le comportement sur les foncteurs polynomiaux. Cet adjoint est donné par une légère variante d'une construction classique en $K$-théorie algébrique, due à Quillen (voir remarque \ref{quillen}).

\smallskip

Soit $(\M,\oplus,0)$ un objet de $\M on$. On définit une catégorie $\widetilde{\M}$ par :
\begin{enumerate}
 \item ${\rm Ob}\,\widetilde{\M}={\rm Ob}\,\M$ ;
\item $\widetilde{\M}(a,b)=\underset{\M}{\col}\tau_b\M(a,-)$ ;
\item la composition $\widetilde{\M}(b,c)\times\widetilde{\M}(a,b)\to\widetilde{\M}(a,c)$ s'obtient en prenant la colimite sur les objets $t$ et $u$ de $\M$ des fonctions
$$\M(b,c\oplus u)\times\M(a,b\oplus t)\xrightarrow{(-\oplus t)_*\times Id}\M(b\oplus t,c\oplus u\oplus t)\times\M(a,b\oplus t)\xrightarrow{\circ}\M(a,c\oplus u\oplus t)$$
puis en appliquant la fonction  
$$\underset{(t,u)\in\M\times\M}{\col}\M(a,c\oplus u\oplus t)\to\underset{s\in\M}{\col}\M(a,c\oplus s)$$
induite par le foncteur $\oplus : \M\times\M\to\M$.
\end{enumerate}

On vérifie aussitôt que $\widetilde{\M}$ est bien une catégorie et que $\oplus$ induit une structure monoïdale symétrique dessus (encore notée de la même manière) dont $0$ est l'unité. Noter ici qu'on a besoin de la symétrie : cette structure est obtenue sur les morphismes --- disons $\widetilde{\M}(a,b)\times\widetilde{\M}(c,d)\to\widetilde{\M}(a\oplus c,b\oplus d)$ --- à partir des  fonctions
$$\M(a,b\oplus t)\times\M(c,d\oplus u)\to\M(a\oplus c,b\oplus t\oplus d\oplus u)\simeq\M(a\oplus c,b\oplus d\oplus t\oplus u)$$
où la première flèche est induite par $\oplus$ et la dernière par l'isomorphisme structural $t\oplus d\simeq d\oplus t$, puis en prenant la colimite sur $t$ et $u$. 

On remarque que, pour tout objet $a$ de $\M$, l'ensemble $\widetilde{\M}(a,0)$ est la colimite du foncteur $\M(a,-)$. Par le lemme de Yoneda, on en déduit que  $\widetilde{\M}(a,0)$ est l'ensemble à un élément. Autrement dit, $0$ est toujours {\em objet final} de la catégorie $\widetilde{\M}$. Il est clair également que, si $0$ est objet initial de $\M$, alors il en est de même dans $\widetilde{\M}$. Ainsi, si $\M$ est objet de $\mi$, alors $\widetilde{\M}$ est objet de $\mn$.

On dispose d'un foncteur $\eta_\M : \M\to\widetilde{\M}$ égal à l'identité sur les objets et donné sur les morphismes par l'application canonique
$$\M(a,b)=\big(\tau_b\M(a,-)\big)(0)\to\underset{\M}{\col}\tau_b\M(a,-)=\widetilde{\M}(a,b).$$
Ce foncteur est monoïdal au sens strict.

\begin{rem}\label{quillen} \cite[§\,1.1]{RWW}
Notre catégorie $\widetilde{\M}$ est très analogue à la catégorie, introduite par Quillen, qui est notée $<\M,\M>$ dans \cite[page~3]{Gray} ou plutôt à $<\M^{op},\M^{op}>^{op}$. Cette catégorie a les mêmes objets que $\M$ et ses morphismes sont donnés par
$$<\M^{op},\M^{op}>^{op}(a,b)=\underset{t\in {\rm Iso}(\M)}{\col}\M(a,b\oplus t)$$
où ${\rm Iso}(\M)$ désigne la sous-catégorie des isomorphismes de $\M$. On dispose d'un foncteur monoïdal $<\M^{op},\M^{op}>^{op}\to\widetilde{\M}$ qui est l'identité sur les objets et est donné sur les morphismes par l'application canonique
$$\underset{t\in {\rm Iso}(\M)}{\col}\M(a,b\oplus t)\to\underset{t\in \M}{\col}\M(a,b\oplus t)=\widetilde{\M}(a,b).$$
\end{rem}

\begin{ex} \label{ex-theta-tilde}
La catégorie $\widetilde{\Theta}$ est équivalente à la catégorie ${\rm FI}\#$ de \cite{CEF} des ensembles finis avec injections partiellement définies. Cette catégorie apparaît également dans \cite{CDG}, où elle est notée $\Theta$. L'équivalence entre $\widetilde{\Theta}$ et ${\rm FI}\#$ est l'identité sur les objets et associe à un morphisme $f\in\widetilde{\Theta}(A,B)$, représenté par une fonction injective, partout définie, $u : A\to B\sqcup E$, l'injection partielle de $A$ vers $B$ définie sur $u^{-1}(B)$ coïncidant avec $u$.
\end{ex}

\begin{ex}\label{rq-tq}
 Soit $\E_q$ la catégorie des espaces quadratiques (non dégénérés) de dimension finie sur le corps à deux éléments, et dont les morphismes sont les applications linéaires qui préservent les formes quadratiques. Les articles \cite{V-pol, V-quad} étudient la catégorie des foncteurs depuis la catégorie, notée $\T_q$, ayant les mêmes objets que $\E_q$ et dont les morphismes de $V$ dans $W$ sont des classes d'équivalence de diagrammes $[V\rightarrow X\leftarrow W]$ appelés triplets, pour des objets $V, W$ de $\T_q$ \cite[Definition $2.8$]{V-pol}.
 La catégorie $\widetilde{\E_q}$ est équivalente à la catégorie $\T_q$. L'équivalence est l'identité sur les objets et associe à un morphisme de $\widetilde{\E_q}(V,W)$, représenté par un morphisme quadratique $f : V\to W\overset{\perp}{\oplus}H$ de $\E_q$, le morphisme $[V\xrightarrow{f}W\overset{\perp}{\oplus}H\hookleftarrow W]$ de $\T_q$, où la flèche de droite est l'inclusion canonique $W\hookrightarrow W\overset{\perp}{\oplus}H$. La remarque~1.2 de \cite{V-pol} permet de voir que c'est effectivement une équivalence.

On peut donner une description analogue de la catégorie $\widetilde{\E_q}$, où $\E_q$ est la catégorie des espaces quadratiques (ou symplectiques, hermitiens) non dégénérés sur un anneau quelconque, en termes de classes de triplets.

D'une manière générale, pour tout objet $\M$ de $\M on$, on peut voir les foncteurs depuis $\widetilde{\M}$ comme des {\em foncteurs de Mackey généralisés} sur $\M$, suivant le même point de vue que \cite{V-quad}. Rappelons qu'un foncteur de Mackey est la donnée d'un foncteur covariant et d'un foncteur contravariant qui co\"incident sur les objets et satisfont des propriétés de compatibilité appropriées sur les morphismes. En généralisant \cite{Lin}, on obtient que les catégories de foncteurs sur des catégories dont les morphismes sont des classes de triplets correspondent à la donnée d'un foncteur de Mackey généralisé.
\end{ex}

\begin{pr}\label{proltilde}
 Soit $\Phi : \M\to\N$ un morphisme de $\mi$. On suppose que $\N$ appartient à $\mn$. Alors il existe un et un seul morphisme $\Psi : \widetilde{\M}\to\N$ tel que $\Phi=\Psi\circ\eta_\M$.
\end{pr}

\begin{proof}
 Le foncteur $\Psi$ est défini comme suit : sur les objets, il coïncide avec $\Phi$ ; si $f : a\to b\oplus t$ est un morphisme de $\M$ représentant un élément $\xi$ de $\widetilde{\M}(a,b)$, $\Psi$ envoie $\xi$ sur
$$\Phi(a)\xrightarrow{\Phi(f)}\Phi(b\oplus t)=\Phi(b)\oplus\Phi(t)\to\Phi(b)$$
(où la deuxième flèche est le morphisme canonique déduit de ce que $0$ est objet final de $\N$) --- on vérifie aussitôt que cette flèche de $\N$ ne dépend pas du choix du représentant $f$ de $\xi$, puis que cette construction est compatible à la composition des morphismes de $\widetilde{\M}$ et aux unités. L'égalité $\Phi=\Psi\circ\eta_\M$ et le caractère monoïdal de $\Psi$ sont clairs. L'unicité de $\Psi$ découle de ce qu'un morphisme vérifiant les conditions de l'énoncé envoie nécessairement la flèche canonique $b\oplus t\to b$ de $\widetilde{\M}$ sur la flèche canonique $\Psi(b)\oplus\Psi(t)\to\Psi(b)$ de $\N$ et de ce que $\widetilde{\M}$ est engendrée par $\M$ et ces flèches canoniques. 
\end{proof}

Comme conséquence, pour tout morphisme $\Phi : \M\to\N$ de $\mi$, il existe un et un seul morphisme $\widetilde{\Phi} : \widetilde{\M}\to\widetilde{\N}$ de $\mn$ faisant commuter le diagramme
$$\xymatrix{\M\ar[r]^\Phi\ar[d]_{\eta_\M} & \N\ar[d]^{\eta_\N}\\
\widetilde{\M}\ar[r]^{\widetilde{\Phi}} & \widetilde{\N}
}.$$

Cela permet de faire de $\M\mapsto\widetilde{\M}$ un foncteur $\mi\to\mn$ ; la proposition~\ref{proltilde} peut se reformuler comme suit :
\begin{cor}\label{adj-tilde}
 Le foncteur $\M\mapsto\widetilde{\M}$ est adjoint à gauche à l'inclusion $\mn\hookrightarrow\mi$.
\end{cor}

\begin{rem}\label{rq-tilde}
 On peut procéder de façon entièrement analogue pour l'inclusion dans $\M on$ de la sous-catégorie pleine des catégories monoïdales symétriques dont l'unité est objet final. Nous n'abordons pas ce cas dans le présent article car, dans toutes les catégories qui nous intéressent, l'unité est objet initial.

On rappelle que $\Sigma$ désigne la catégorie des ensembles finis avec bijections. La catégorie $\Theta$ est équivalente à $\widetilde{\Sigma}^{op}$. En effet, en considérant deux entiers $m\leq n$ et en notant encore $G$ la catégorie à un objet associée au groupe $G$ on a:
$$\widetilde{\Sigma}(n,m)=\underset{\Sigma}{\col}\tau_m(\Sigma(n,-))=\underset{i \in \Sigma}{\col}\Sigma(n,i+m)=\underset{i \in \Sigma_{n-m}}{\col}\Sigma_n$$
$$\simeq\Sigma_n/\Sigma_{n-m} \simeq \Theta(m,n).$$
Cela permet de vérifier la propriété universelle de l'exemple~\ref{exini}.\ref{nini} à partir de celle, plus directe, de la remarque~\ref{sigma}.
\end{rem}

\begin{pr}\label{ext-kan}
 Soient $\M$ un objet de $\mi$ et $\A$ une catégorie cocomplète. La précomposition par $\eta_\M : \M\to\widetilde{\M}$  possède un adjoint à gauche $\alpha_\M : \fct(\M,\A)\to\fct(\widetilde{\M},\A)$ donné sur les objets par
 $$\big(\alpha_\M(F)\big)(x)=\underset{\M}{\col}\tau_x(F).$$
\end{pr}

\begin{proof} La théorie générale des extensions de Kan \cite[chap.~X, §\,3]{ML} montre que $\eta_\M^*$ possède un adjoint à gauche dont l'évaluation sur $F : \M\to\A$ vaut, évaluée sur $x\in {\rm Ob}\,\widetilde{\M}$,
$$\underset{\M_x}{\col}F\circ\iota_x,$$
où $\M_x$ est la catégorie dont les objets sont les couples $(t,\phi)$ formés d'un objet $t$ de $\M$ et d'un morphisme $\phi : \eta_\M(t)\to x$ de $\widetilde{\M}$, les morphismes $(t,\phi)\to (u,\psi)$ étant les morphismes $f : t\to u$ de $\widetilde{\M}$ tels que $\phi=\psi\circ\eta_\M(f)$) et $\iota_x : \M_x\to\M$ désigne le foncteur d'oubli $(t,\phi)\mapsto t$. Soit $\eta_\M(t)\oplus x\to x$ le morphisme canonique venant de ce que $0$ est objet final dans $\widetilde{\M}$. La conclusion résulte donc de l'observation que le foncteur
$$\rho_x : \M\to\M_x\qquad t\mapsto (t\oplus x,\eta_\M(t)\oplus x\to x)$$
 est {\em final} au sens de \cite[chap.~IX, §\,3]{ML}  (la terminologie usuelle est {\em cofinal}), de sorte que le morphisme canonique
$$\underset{\M}{\col}\tau_x(F)=\underset{\M}{\col}F\circ\iota_x\circ\rho_x\to\underset{\M_x}{\col}F\circ\iota_x$$
est un isomorphisme.
\end{proof}

Nous pouvons maintenant énoncer et démontrer le résultat principal de cette section.

\begin{thm}\label{tilde-pol}
 Soient $\M$ un objet de $\mi$ et $\A$ une catégorie de Grothendieck. Pour tout $n\in\mathbb{N}$, les foncteurs $\eta_\M^* : \fct(\widetilde{\M},\A)\to\fct(\M,\A)$ et $\alpha_\M : \fct(\M,\A)\to\fct(\widetilde{\M},\A)$ induisent des équivalences de catégories quasi-inverses l'une de l'autre entre $\Pol_n(\M,\A)/\Pol_{n-1}(\M,\A)$ et $\Pol_n(\widetilde{\M},\A)/\Pol_{n-1}(\widetilde{\M},\A)$.
\end{thm}

\begin{proof}
On dispose d'un foncteur $x\mapsto\tau_x$ de $\M$ vers la catégorie des endofoncteurs de $\st(\M,\A)$, par la proposition~\ref{pr-fde}. En le combinant au foncteur homologique $h^\M_*$ donné par la proposition~\ref{pr-homst}, on obtient un foncteur homologique $\bar{\alpha}_* : \st(\M,\A)\to\fct(\widetilde{\M},\A)_{gr}$ donné sur les objets par
$$\bar{\alpha}_*(F)(x)=h^\M_*(\tau_x(F))\;.$$
Le foncteur $\alpha_\M$ est égal à la composée de $\bar{\alpha}_0$ et du foncteur canonique $\pi_\M : \fct(\M,\A)\to\st(\M,\A)$.

Notons $\bar{\eta} : \fct(\widetilde{\M},\A)\to\st(\M,\A)$ la composée de $\eta^*_\M$ et $\pi_\M$. Alors $\bar{\alpha}_0$ est adjoint à gauche à $\bar{\eta}$. Cette adjonction provient de celle entre $\alpha_\M$ et $\eta_\M^*$ et de l'observation suivante. Pour tous objets $x$ et $t$ de $\M$, l'image dans $\widetilde{\M}$ du morphisme canonique $x\to t\oplus x$ est scindée {\em naturellement} puisque $0$ est objet nul de $\widetilde{\M}$. On en déduit que, pour tous objets $F$ de $\fct(\widetilde{\M},\A)$ et $t$ de $\M$, la flèche
$$i_t\big(\eta_\M^*(F)\big) : \eta_\M^*(F)\to\tau_t\big(\eta_\M^*(F)\big)$$
est un monomorphisme scindé. Par le lemme~\ref{lm-stabnul}, on conclut que le morphisme
$${\rm Hom}_{\fct(\M,\A)}\big(X,\eta_\M^*(F)\big)\to {\rm Hom}_{\st(\M,\A)}\big(\pi(X),\pi(\eta_\M^*(F))\big)$$
induit par $\pi_\M$ est, pour tous $X\in{\rm Ob}\,\fct(\M,\A)$ et $F\in {\rm Ob}\,\fct(\widetilde{\M},\A)$, un isomorphisme. Ceci établit notre adjonction à partir de celle de la proposition~\ref{ext-kan}.

On note à présent que le foncteur $\bar{\eta}$ commute, à isomorphisme naturel près, aux foncteurs $\tau_t$ et $\delta_t$, puisque c'est la composition de $\eta_\M^*$, précomposition par un foncteur monoïdal, et de $\pi_\M$, qui commute également à ces foncteurs. Il en est de même pour les foncteurs $\bar{\alpha}_i$ : la commutation aux foncteurs $\tau_t$ est claire ; pour $\delta_t$, considérer la suite exacte longue obtenue en appliquant $\bar{\alpha}_*$ à la suite exacte courte $0\to {\rm Id}\to\tau_t\to\delta_t\to0$ d'endofoncteurs de $\st(\M,\A)$.

Ces propriétés de commutation montrent que $\bar{\eta}$ et $\bar{\alpha}_*$ envoient un objet polynomial de degré au plus $n$ sur un objet polynomial de degré au plus $n$. Comme $\bar{\eta}$ est exact, il induit par conséquent des foncteurs exacts
$$E_n : \Pol_n(\widetilde{\M},\A)/\Pol_{n-1}(\widetilde{\M},\A)\to\Pol_n(\M,\A)/\Pol_{n-1}(\M,\A).$$

On va montrer par récurrence sur $n\in\mathbb{N}$ les assertions suivantes :
\begin{enumerate}
 \item pour $i>0$, les foncteurs $\bar{\alpha}_i$ envoient $\Pol_n(\M,\A)$ dans $\Pol_{n-1}(\widetilde{\M},\A)$, ce qui implique que le foncteur exact à droite $\bar{\alpha}_0$ induit un foncteur exact
$$A_n : \Pol_n(\M,\A)/\Pol_{n-1}(\M,\A)\to\Pol_n(\widetilde{\M},\A)/\Pol_{n-1}(\widetilde{\M},\A)\;;$$
\item les foncteurs $A_n$ et $E_n$ sont des équivalences de catégories quasi-inverses l'une de l'autre.
\end{enumerate}

Pour $n=0$, on utilise la proposition~\ref{fdeg0} : comme l'homologie réduite de la catégorie $\M$ à coefficients constants est nulle (puisque $\M$ possède un objet initial), on a $\bar{\alpha}_i(X)=0$ pour $i>0$ et $X$ polynomial de degré $0$. Il est également clair que $A_0$ et $E_0$ sont des équivalences mutuellement quasi-inverses, l'effet de $\alpha_\M$ et $\eta_\M^*$ sur les foncteurs constants étant transparent.

Supposons désormais $n>0$ et les assertions vérifiées pour $n-1$. Pour $i>0$, l'inclusion $\bar{\alpha}_i(\Pol_{n-1}(\M,\A))\subset\Pol_{n-2}(\widetilde{\M},\A)$ et la commutation de $\bar{\alpha}_i$ aux foncteurs différences $\delta_x$ observée plus haut impliquent que $\bar{\alpha}_i$ envoie $\Pol_n(\M,\A)$ dans $\Pol_{n-1}(\widetilde{\M},\A)$. On raisonne de même avec l'unité ${\rm Id}\to\bar{\eta}\bar{\alpha}_0$ et la coünité $\bar{\alpha}_0\bar{\eta}\to {\rm Id}$ de l'adjonction : elles commutent aux foncteurs différences, et leurs noyaux et conoyaux sont polynomiaux de degré au plus $n-2$ sur les objets polynomiaux de degré au plus $n-1$. C'est ce que signifie que $A_{n-1}$ et $E_{n-1}$ sont des équivalences mutuellement quasi-inverses. Par conséquent, comme les foncteurs différences de $\st(\M,\A)$ et $\fct(\widetilde{\M},\A)$ commutent également aux foncteurs différences, ces noyaux et conoyaux envoient les objets polynomiaux de degré au plus $n$ sur des objets polynomiaux de degré au plus $n-1$. Ceci termine la démonstration.
\end{proof}

Le corollaire suivant tempère le phénomène observé à la proposition~\ref{rq-pol}.

\begin{cor}\label{rq-ptil}
Tout objet polynomial de degré $d$ de $\st(\M,\A)$ est isomorphe {\em modulo $\Pol_{d-1}(\M,\A)$} à l'image par le foncteur $\pi_\M$ d'un foncteur fortement polynomial de degré fort~$d$.
\end{cor}

\begin{proof}
Le foncteur $\eta_\M$ étant mono\"idal fort, on déduit de la proposition \ref{compo-mono-fort} que l'image par $\eta_\M^*$ d'un foncteur polynomial de degré $d$ est un foncteur  {\em fortement} polynomial dont le degré fort est égal à $d$. L'assertion se déduit donc de l'essentielle surjectivité du foncteur induit par $\eta_\M^*$  dans le théorème~\ref{tilde-pol}.
\end{proof}

La variation suivante du théorème~\ref{tilde-pol} nous sera également utile.

\begin{pr}\label{pr-filtr}
 Soient $\M$ un objet de $\mi$, $\A$ une catégorie de Grothendieck, $n\in\mathbb{N}$. La catégorie $\Pol_n(\M,\A)$ est la plus petite sous-catégorie pleine $\C_n$ de $\st(\M,\A)$ contenant l'image de $\Pol_n(\widetilde{\M},\A)$ par $\pi_\M\eta_\M^*$ et vérifiant les deux conditions suivantes :
\begin{enumerate}
 \item pour toute suite exacte $0\to A\to B\to C\to 0$ de $\st(\M,\A)$, si $B$ et $C$ appartiennent à $\C_n$, alors $A$
 appartient à $\C_n$ ;
 \item pour toute suite exacte $0\to A\to B\to C\to 0$ de $\st(\M,\A)$, si $A$ et $C$ appartiennent à $\C_n$, alors $B$
 appartient à $\C_n$.
\end{enumerate}
\end{pr}

\begin{proof}
La sous-catégorie $\Pol_n(\M,\A)$ de $\st(\M,\A)$ est épaisse (proposition~\ref{propol-gal}), elle vérifie donc les conditions de stabilité de l'énoncé, et contient l'image de  $\Pol_n(\widetilde{\M},\A)$ par $\pi_\M\eta_\M^*$ puisque ce foncteur préserve le degré polynomial. Pour la réciproque, on raisonne par récurrence sur $n$ : le théorème~\ref{tilde-pol} montre que, pour tout objet $X$ de $\Pol_n(\M,\A)$, les noyau $N$ et conoyau $C$ de l'unité $u : X\to\pi_\M\eta_\M^*(Y)$, où $Y:=\bar{\alpha}(X)$ (cf. les notations de la démonstration de la proposition~\ref{propol-gal}), sont de degré au plus $n-1$, ils appartiennent donc, par l'hypothèse de récurrence, à la sous-catégorie $\C_{n-1}\subset\C_n$. Les suites exactes courtes
$$0\to D\to\pi_\M\eta_\M^*(X)\to C\to 0,$$
où $D$ est l'image de $u$, et
$$0\to N\to X\to D\to 0$$
permettent donc de conclure.
\end{proof}

\begin{rem}\label{rq-stab}
 On ne peut pas se contenter de la stabilité par noyaux dans la proposition précédente. En effet, la plus petite sous-catégorie pleine de $\st(\M,\A)$ stable par noyaux et contenant l'image de $\Pol_n(\widetilde{\M},\A)$ par $\pi_\M\eta_\M^*$ est incluse dans la sous-catégorie pleine de $\st(\M,\A)$ des sous-objets des $\bar{\eta}(F)=\pi_\M\eta_\M^*(F)$, pour $F$ dans $\Pol_n(\widetilde{\M},\A)$. Comme $\bar{\eta}$ est adjoint à droite à $\bar{\alpha}$ (cf. la démonstration du théorème~\ref{tilde-pol}), tout morphisme $X\to\bar{\eta}(F)$ se factorise par l'unité $X\to\bar{\eta}\bar{\alpha}(X)$ de l'adjonction, de sorte que celle-ci est un monomorphisme si $X$ est un sous-objet d'un $\bar{\eta}(F)$. Or l'exemple~\ref{ex-upm} ci-après exhibe un objet $X$ de $\Pol_2(\Theta,\A)$, pour $\A$ la catégorie des $\FF$-espaces vectoriels, tel que $X\to\bar{\eta}\bar{\alpha}(X)$ n'est pas un monomorphisme.
\end{rem}

\section{Foncteurs de $\Theta $ dans $\mathbf{Ab}$} \label{exple-theta}

Le but de cette section est d'étudier quelques exemples explicites d'objets de la catégorie $\fct(\Theta,\mathbf{Ab})$ qui illustrent les notions introduites dans cet article et de faire le lien avec les résultats de \cite{CEF, CEFN}.

Notons que, grâce à la dernière partie de la proposition~\ref{pr-polfor}, un foncteur $F$ dans cette catégorie est fortement polynomial de degré fort au plus $n$ si et seulement si $\delta^{n+1}_\mathbf{1}(F)=0$.

On commence par donner plusieurs exemples illustrant le comportement des foncteurs fortement et faiblement polynomiaux et les liens entre ces deux notions. Ces exemples montrent l'intérêt d'introduire la notion de foncteur faiblement polynomial. 

Les deux exemples suivants témoignent de ce que la classe des foncteurs fortement polynomiaux de degré au plus $n$ n'est généralement pas stable par sous-foncteur, dès le cas $n=0$.

\begin{ex}\label{ex-degfort}
 Le foncteur constant $\mathbb{Z}$ de $\fct(\Theta,\mathbf{Ab})$ est fortement polynomial de degré~$0$. Pour tout $n\in\mathbb{N}$, il possède un sous-foncteur $\mathbb{Z}_{\geq n}$ qui est nul sur les ensembles de cardinal strictement inférieur à $n$ et égal à $\mathbb{Z}$ sur les autres ensembles finis. Le foncteur $\mathbb{Z}_{\geq n}$ est polynomial fort de degré exactement $n$. Pour le voir, on effectue les calculs suivants dans $\fct(\Theta,\mathbf{Ab})$ :
$$\forall (n,i)\in\mathbb{N}^2\quad \delta_{\mathbf{1}}(\mathbb{Z}_{\geq n})=\mathbb{Z}_{n-1},\qquad\delta_{\mathbf{1}}(\mathbb{Z}_i)=\mathbb{Z}_{i-1}$$
où $\mathbb{Z}_i$ est le foncteur égal à $\mathbb{Z}$ sur les ensembles de cardinal $i$ et nul ailleurs.
\end{ex}
Un foncteur fortement polynomial peut avoir des sous-foncteurs qui ne sont pas fortement polynomiaux comme le montre l'exemple suivant.
\begin{ex}\label{ex-degfort2}
Le foncteur $\underset{i \in \mathbb{N}}{\bigoplus}\mathbb{Z}$ est fortement polynomial de degré $0$ et a pour sous-foncteur le foncteur $\underset{i \in \mathbb{N}}{\bigoplus}\mathbb{Z}_{\geq i}$ qui n'est pas fortement polynomial (car $\mathbb{Z}_{\geq i}$ est de degré fort $i$).
\end{ex}

L'exemple suivant illustre le fait que l'image par le foncteur section d'un objet polynomial de degré $n>0$ n'est pas forcément un foncteur fortement polynomial de degré fort $n$, contrairement à ce qui se passe pour les foncteurs polynomiaux de degré nul (voir la proposition \ref{fdeg0}).
\begin{ex}\label{exdeg}
 Notons $P$ le foncteur $E\mapsto\mathbb{Z}[E]$ de $\fct(\Theta,\mathbf{Ab})$, on dispose d'un morphisme $P\to\mathbb{Z}$ donné par l'augmentation $\mathbb{Z}[E]\to\mathbb{Z}$ ; son conoyau est $\mathbb{Z}_0$ (on conserve les notations de l'exemple~\ref{ex-degfort}), nous noterons $F$ son noyau. Un calcul facile montre que $\delta_\mathbf{1}(F)\simeq\mathbb{Z}_{\geq 1}$. En utilisant l'exemple~\ref{ex-degfort}, on obtient que $F$ est polynomial fort de degré fort $2$  et en utilisant, de plus, la proposition \ref{pr-fde} on obtient que $\delta_{\textbf{1}}\delta_{\textbf{1}}(\pi_\Theta(F)) \simeq \pi_\Theta(\mathbb{Z}_0)=0$. Par la proposition~\ref{pr-dpd}, on en déduit que $\pi_\Theta(F)$ est polynomial de degré $1$ autrement dit, $F$ est faiblement polynomial de degré faible $1$. Pour autant, $F$ est $\s n(\Theta,\mathbf{Ab})$-fermé, de sorte que $F\simeq s_\Theta\pi_\Theta(F)$. Cela se déduit directement de la suite exacte
$$0\to F\to P\to\mathbb{Z}$$
et de la nullité de ${\rm Ext}^i_{\fct(\Theta,\mathbf{Ab})}(S,P)$ pour $i\leq 1$ et de ${\rm Hom}_{\fct(\Theta,\mathbf{Ab})}(S,\mathbb{Z})$ pour $S$ stablement nul, qu'on tire du lemme~\ref{lm-stabnul}.{\em 3}.

En revanche, l'objet $\pi_{\Theta}(F)$ de $\st(\Theta,\mathbf{Ab})$ est isomorphe, modulo $\Pol_0(\Theta,\mathbf{Ab})$, à $\pi_{\Theta}(P)$, et $P$ est fortement polynomial de degré $1$, puisque $\delta_1(P)\simeq \Z$. Ceci illustre le corollaire~\ref{rq-ptil}.

Par ailleurs, le foncteur $\tau_{\textbf{1}}(F)$ est isomorphe à $P$, donc fortement polynomial de degré fort $1$, ce qui illustre un des phénomènes expliqués dans la remarque~\ref{lien-pol-polfort}.
\end{ex}

Dans l'exemple qui suit on exhibe un foncteur faiblement polynomial qui n'est pas fortement polynomial.

\begin{ex} \label{faible-pas-fort}
Le foncteur $\underset{n\in\mathbb{N}}{\bigoplus}\mathbb{Z}_n$ n'est pas fortement polynomial car $\mathbb{Z}_n$ est de degré fort $n$. En revanche, son image dans la catégorie quotient $\mathbf{St}(\Theta,\mathbf{Ab})$ est nulle, donc polynomiale de degré $-\infty$. 
\end{ex}

L'exemple suivant illustre la remarque~\ref{rq-stab}.

\begin{ex}\label{ex-upm}
Dans la catégorie des foncteurs de $\Theta$ vers la catégorie $\A$ des espaces vectoriels sur $\FF$, notons $P_2:=\FF[\Theta(\mathbf{2},-)]$ la linéarisation du foncteur ensembliste $\Theta(\mathbf{2},-)$ et $A:=\FF[\Theta(\mathbf{2},-)/\Sigma_2]$. Il existe une unique (à isomorphisme près) suite exacte courte non scindée
$$0\to\FF\to F\to A\to 0\;.$$
Le foncteur $F$ peut être construit comme la somme amalgamée du morphisme $\nu : A\hookrightarrow P_2$ donné par la norme, dont le conoyau est également isomorphe à $A$ et de l'augmentation $A\to\FF$, qui est surjective dans $\st(\Theta,\A)$. Comme le foncteur $\alpha : \fct(\Theta, \A) \to \fct(\tilde{\Theta}, \A)$ défini à la proposition \ref{ext-kan} préserve les colimites, on en déduit un diagramme cocartésien
$$\xymatrix{\alpha(A)\ar[r]\ar[d] & \alpha(P_2)\ar[d]\\
\alpha(\FF)\ar[r] & \alpha(F)
}$$
On vérifie qu'il est isomorphe à
$$\xymatrix{A\oplus\FF[-]\oplus\FF\ar[r]\ar[d] & P_2\oplus\FF[-]^{\oplus 2}\oplus\FF\ar[d] \\
\FF\ar[r] & A\oplus\FF[-]\oplus\FF
}$$
où la flèche horizontale supérieure est la somme directe de $\nu$, de l'inclusion diagonale $\FF[-]\to\FF[-]^{\oplus 2}$ et du morphisme nul $\FF\to\FF$ et la flèche verticale de gauche la projection évidente ; la flèche horizontale inférieure est {\em nulle}. Expliquons par exemple pourquoi $\alpha(A)\simeq A\oplus\FF[-]\oplus\FF$ (les autres calculs sont tout-à-fait analogues) : on dispose d'un isomorphisme
$$\tau_E(A)\simeq A\oplus(\FF[E]\otimes\FF[-])\oplus A(E)$$
naturel en l'objet $E$ de $\Theta$, d'où
$$\alpha(A)\simeq A\oplus\big((\underset{\Theta}{\col}\FF[-])\otimes\FF[-]\big)\oplus\underset{\Theta}{\col}A\;;$$
la colimite du foncteur $\FF[-]\simeq\FF[\Theta(\mathbf{1},-)]$ est isomorphe à $\FF$, et celle de $A$ également :
$$\underset{\Theta}{\col}A=\underset{\Theta}{\col}\FF[\Theta(\mathbf{2},-)]_{\Sigma_2}\simeq (\underset{\Theta}{\col}\FF[\Theta(\mathbf{2},-)])_{\Sigma_2}\simeq (\FF)_{\Sigma_2}\simeq\FF.$$

Le carré commutatif
$$\xymatrix{\FF\ar[r]\ar[d] & F\ar[d] \\
\eta^*\alpha(\FF)\ar[r]^0 & \eta^*\alpha(F)
}$$
dont la flèche horizontale du haut est l'inclusion et les flèches verticales sont les unités montre donc que $F\to\eta^*\alpha(F)$ n'est pas injective, même dans $\st(\Theta,\A)$ après application de $\pi_\Theta$. Pour autant, $F$ est fortement polynomial de degré $2$. En effet, $F$ est une extension de $A$, qui est un quotient de $P_2$, lui-même fortement polynomial de degré $2$, par le foncteur constant $\FF$.
\end{ex}

\begin{conv}
 Dans la suite de cette section, $\A$ désigne une catégorie de Grothendieck.
\end{conv}

Dans la suite, on caractérise les foncteurs stablement nuls de $\fct(\Theta,\A)$, les foncteurs fortement polynomiaux de cette catégorie et on étudie les catégories quotients $\Pol_n(\Theta,\A)/\Pol_{n-1}(\Theta,\A)$. Ces résultats, particulièrement simples, sont spécifiques à la catégorie $\Theta$ et ne se généralisent pas simplement aux foncteurs sur un objet $\M$ de $\mi$ quelconque.

Les foncteurs $\mathbb{Z}_i: \Theta\to\mathbf{Ab}$ introduits dans l'exemple \ref{ex-degfort} sont stablement nuls. Plus généralement, le corollaire~\ref{corlfi} s'exprime comme suit sur la catégorie $\Theta$ :

\begin{pr} \label{theta-stab-nul}
Un objet $F$ de $\fct(\Theta,\A)$ est stablement nul si et seulement si
$$\underset{n\in\mathbb{N}}{\col}F(\mathbf{n})=0.$$
\end{pr}

Si $\A$ est une catégorie abélienne, $A$ un objet de $\A$  et $E$ un ensemble, on note $A[E]$ la somme de copies de $A$ indexées par $E$. 
Un foncteur $F$ de $\fct(\Theta,\A)$, où $\A$ est une catégorie abélienne, est {\em engendré en cardinal au plus $n$} si tout sous-foncteur $G$ de $F$ tel que l'inclusion $G(\mathbf{i})\subset F(\mathbf{i})$ soit une égalité pour $i\leq n$ est égal à $F$. Il revient au même de demander que $F$ soit isomorphe à un quotient d'une somme directe de foncteurs du type $A[\Theta(\mathbf{i},-)]$ avec $A\in {\rm Ob}\,\A$ et $i\leq n$. 

La proposition suivante est établie dans \cite[proposition 4.4]{Dja-pol}.  

\begin{pr}\label{pfpt}
Un objet de $\fct(\Theta,\A)$ est fortement polynomial de degré fort au plus $n$ si et seulement s'il est engendré en cardinal au plus $n$.
\end{pr}

Le fait qu'un foncteur engendré en cardinal au plus $n$ soit polynomial est essentiellement spécifique à la catégorie source $\Theta$, il ne possède notamment aucun analogue dans $\mathbf{S}(\mathbf{ab})$ \cite[remarque~4.5]{Dja-pol}.

\medskip

Un cas particulier du théorème de Pirashvili à la Dold-Kan \cite{PDK} montre que $\fct(\widetilde{\Theta},\A)\simeq\fct(\Sigma,\A)\simeq\underset{n\in\mathbb{N}}{\prod}\fct(\Sigma_n,\A)$ pour toute catégorie abélienne $\A$, l'équivalence étant fournie à l'aide d'effets croisés. On rappelle que $\Sigma$ désigne la catégorie des ensembles finis avec bijections. Ce résultat apparaît également dans \cite[theorem~2.24]{CEF}. En effet, la catégorie $\widetilde{\Theta}$ est équivalente à la catégorie ${\rm FI}\#$ de \cite{CEF} des ensembles finis avec injections partiellement définies (voir l'exemple~\ref{ex-theta-tilde}). 

Par conséquent, on déduit du théorème~\ref{tilde-pol} le résultat suivant :

\begin{pr} \label{eq-theta}
Pour tout $n\in \mathbb{N}$, on a des équivalences de catégories :
$$\Pol_n(\Theta,\A)/\Pol_{n-1}(\Theta,\A)\simeq\Pol_n(\widetilde{\Theta},\A)/\Pol_{n-1}(\widetilde{\Theta},\A)\simeq\fct(\Sigma_n,\A).$$
\end{pr}
En revanche, contrairement à $\Pol_n(\widetilde{\Theta},\A)$, $\Pol_n(\Theta,\A)$ ne se décrit pas simplement à partir des seules représentations dans $\A$ des groupes symétriques $\Sigma_i$ pour $i\leq n$.

\medskip

Nous montrons maintenant que la notion de foncteur fortement polynomial sur $\Theta$ est reliée à des propriétés de finitude classiques. Ceci nous permet de réexprimer certains résultats de \cite{CEF, CEFN} en termes de foncteurs fortement polynomiaux.

On rappelle qu'un objet $A$ d'une catégorie abélienne est  dit \textit{de type fini} si toute famille filtrante croissante de sous-objets de $A$ de réunion $A$ stationne. Pour les $FI$-modules, cette notion est équivalente à la notion de \textit{finitely generated $FI$-modules} de \cite[Definition $1.2$]{CEF}.
\begin{pr} \label{CEF-1.2}
\begin{enumerate}
\item
Un foncteur de $\fct(\Theta,\A)$ est de type fini si et seulement s'il existe un entier $n$ tel qu'il soit engendré en cardinal au plus $n$ et qu'il prend des valeurs de type fini dans $\A$. On peut d'ailleurs se restreindre aux valeurs sur $\mathbf{i}$ pour $i\leq n$.
\item
Tout foncteur de type fini de $\fct(\Theta,\A)$ est fortement polynomial.
\item 
Tout foncteur de $\fct(\Theta,\A)$ fortement polynomial et prenant des valeurs de type fini est de type fini.
\end{enumerate}
\end{pr}

\begin{proof}
Soient $F$ un foncteur de $\fct(\Theta,\A)$ engendré en cardinal au plus $n$ et tel que $F(\mathbf{i})$ soit de type fini pour $i\leq n$ et $(G_r)_{r\in\mathbb{N}}$ une suite croissante de sous-foncteurs de $F$ de réunion $F$. Pour $r$ assez grand, on a $G_r(\mathbf{i})=F(\mathbf{i})$ pour $i\leq n$, puisque les $F(\mathbf{i})$ sont de type fini. On a donc $G_r=F$ pour $r$ assez grand, par hypothèse d'engendrement.
 
 Réciproquement, si $F$ est de type fini, $F$ est la réunion de la suite croissante $(G_r)$ de sous-foncteurs donnée comme suit : $G_r(E)$ est la somme des images des applications $F(\mathbf{i})\to F(E)$ induites par tous les morphismes $\mathbf{i}\to E$ de $\Theta$, donc $F=G_n$ pour un $n$, ce qui montre que $F$ est engendré en cardinal au plus $n$. Si $F(\mathbf{i})$ est réunion croissante d'une suite $(A_t)$ de sous-objets de $\A$, alors $F$ est réunion croissante de la suite $(R_t)$ de sous-foncteurs définie ainsi : $R_t(E)$ est le sous-objet de $F(E)$ intersection sur les morphismes $f : E\to\mathbf{i}$ de $\Theta$ de l'image inverse par $F(f)$ de $A_t$. On a donc $R_t=F$ pour un certain $t$, d'où $A_t= F(\mathbf{i})$, de sorte que $F(\mathbf{i})$ est de type fini dans $\A$.

Ce premier point étant démontré, le reste de la proposition se déduit de la proposition~\ref{pfpt}.
\end{proof}

On rappelle qu'un objet $A$ d'une catégorie abélienne est dit \textit{noethérien} si toute suite croissante de sous-objets de $A$ stationne.
Dans la proposition qui suit, on note $G_0(\A)$ le groupe de Grothendieck des objets noethériens de $\A$ qui est bien défini car $\A$ est une catégorie de Grothendieck. Si $F : \Theta\to\A$ est un foncteur prenant des valeurs noethériennes, on note ${\rm dv}_F : \mathbb{N}\to G_0(\A)$ la fonction associant à $n$ la classe de $F(\mathbf{n})$ dans $G_0(\A)$.  On rappelle qu'une fonction $f : \mathbb{N}\to A$, où $A$ est un groupe abélien, est dite polynomiale de degré au plus $n$ si sa $n$-ème déviation (au sens de \cite[ §\,8]{EML} où l'on peut remplacer à la source le groupe abélien par un monoïde abélien comme $\mathbb{N}$ sans changement) est nulle ; elle est dite polynomiale de degré au plus $n$ à partir d'un certain rang si sa $n$-ème déviation est nulle sur tout $n+1$-uplet d'entiers assez grands.

\begin{pr}\label{rq-fi1}
 Soit $F : \Theta\to\A$ un foncteur noethérien.
 \begin{enumerate}
  \item On a ${\rm dv}_F(n+1)={\rm dv}_F(n)+{\rm dv}_{\delta F}(n)$ pour tout entier $n$ assez grand ;
  \item la fonction ${\rm dv}_F$ est polynomiale à partir d'un certain rang.
 \end{enumerate}
\end{pr}

\begin{proof}
 Le premier point résulte de ce que le morphisme canonique $i_{\mathbf{1}}(F)(\mathbf{n}): F(\mathbf{n})\to F(\mathbf{n+1})$ est un monomorphisme pour $n$ assez grand, car son noyau est somme directe de foncteurs atomiques (i.e. prenant des valeurs nulles sauf sur les ensembles finis d'un cardinal donné). Le deuxième s'en déduit par récurrence sur le degré polynomial fort de $F$, en utilisant la proposition~\ref{pfpt}.
\end{proof}

Comme le groupe de Grothendieck d'une catégorie des espaces vectoriels de dimension finie est isomorphe, via la dimension, à $\mathbb{Z}$, on a ${\rm dv}_F(n)=\textrm{dim}(F(n))$ lorsque $F$ est un foncteur de $\Theta$ vers une telle catégorie. Ainsi, la proposition précédente explique notamment comment déduire le théorème B de \cite{CEFN} (et même un énoncé un peu plus général) du théorème A de ce même article, qui étend le travail \cite{CEF}. Précisément, le théorème A de \cite{CEFN} affirme que tout objet de type fini de $\fct(\Theta,\A)$ est noethérien si $\A$ est la catégorie des modules sur un anneau noethérien. Cela permet, dans ce cas, de remplacer, dans la proposition précédente, l'hypothèse de noethérianité de $F$ par celle que ce foncteur est de type fini. Le théorème~B de \cite{CEFN} est l'énoncé ainsi obtenu lorsque l'anneau noethérien de base est un corps.

\begin{rem}\label{rq-dimaber}
 Si le caractère polynomial des dimensions des valeurs (à partir d'un certain rang) reflète le caractère polynomial d'un foncteur {\em noethérien} $F$ de $\Theta$ vers une catégorie $\A$ d'espaces vectoriels sur un corps $k$, ce n'est plus du tout le cas sans aucune hypothèse sur $F$. En effet, pour $k_n:=k\otimes\mathbb{Z}_n$  le foncteur égal à $k$ sur $\mathbf{n}$ et nul ailleurs et $a$ une fonction $\mathbb{N}\to\mathbb{N}^*$, le foncteur $\underset{n\in\mathbb{N}}{\bigoplus} k_n^{a(n)}$ n'est pas noethérien mais est stablement nul, donc faiblement polynomial de degré faible $-\infty$. Quant à la fonction ${\rm dv}_F : \mathbb{N}\to\mathbb{Z}$, c'est la fonction $a$, c'est-à-dire {\em n'importe quelle} fonction à valeurs strictement positives !
\end{rem}

\section{Exemples de foncteurs de $\mathbf{S}(\mathbf{ab})$ dans $\mathbf{Ab}$} \label{ex-S(ab)}

On rappelle que $\mathbf{S}(\mathbf{ab})$ désigne la catégorie des groupes abéliens libres $\mathbb{Z}^n$, $n\in\mathbb{N}$, avec pour morphismes les monomorphismes scindés, le scindage étant donné dans la structure.

Nous donnons quelques exemples de foncteurs de $\mathbf{S}(\mathbf{ab})$ dans $\mathbf{Ab}$ dont l'image dans la catégorie $\st(\mathbf{S}(\mathbf{ab}),\mathbf{Ab})$ est polynomiale, au moins conjecturalement, et dont la compréhension fine constitue l'une des motivations de cet article. Le recours à la catégorie quotient est indispensable dans la mesure où la description complète des foncteurs semble hors de portée et fait de surcroît manifestement apparaître des phénomènes instables qu'on souhaite écarter dans un premier temps.

\subsection{Exemples liés aux groupes d'automorphismes des groupes libres et à leurs sous-groupes $IA$}

 Notons $\G$, comme dans \cite{DV2}, la catégorie (ou un squelette de celle-ci) dont les objets sont les groupes libres de type fini et les morphismes $G\to H$ sont les couples $(u,K)$ formés d'un monomorphisme de groupes $u : G\hookrightarrow H$ et d'un sous-groupe $K$ de $H$ tels que $H=K* u(G)$. On dispose alors d'un foncteur Aut de $\G$ vers la catégorie $\mathbf{Gr}$ des groupes associant à un groupe libre son groupe d'automorphismes et à un tel morphisme le morphisme
$${\rm Aut}\,G\to {\rm Aut}\,H\qquad \varphi\mapsto (K* u\varphi u^{-1} : H=K*u(G)\to H).$$
On dispose de même d'un foncteur $\G\to\mathbf{Gr}$ donné sur les objets par $G\mapsto {\rm Aut}\,(G_{ab})$, qu'on peut voir comme la composée
$$\G\xrightarrow{G\mapsto G_{ab}}\mathbf{S}(\mathbf{ab})\xrightarrow{{\rm Aut}}\mathbf{Gr}$$
où le Aut n'est pas le même que le précédent. Le morphisme canonique $G\twoheadrightarrow G_{ab}$ induit un épimorphisme de groupes naturel ${\rm Aut}\,(G)\to {\rm Aut}\,(G_{ab})$ dont le noyau est noté $IA(G)$. On définit ainsi un sous-foncteur $IA : \G\to\mathbf{Gr}$ de Aut. L'étude des groupes $IA(G)$ est d'une grande difficulté (voir \cite[§\,7]{V-icm} par exemple). Voici des foncteurs fondamentaux vers les groupes abéliens construits à partir de ceux-ci qu'on aimerait bien comprendre, au moins stablement :
\begin{enumerate}
 \item les groupes d'homologie $H_n(IA) : \G\to\mathbf{Ab}$ pour lesquels on connaît une réponse complète seulement pour $n\leq 1$;
 \item les foncteurs $\gamma_n(IA)/\gamma_{n+1}(IA) : \G\to\mathbf{Ab}$, où $\gamma_n : \mathbf{Gr}\to\mathbf{Gr}$ est le foncteur associant à un groupe le $n$-ème terme de sa suite centrale descendante (i.e. $\gamma_1(G)=G$ et $\gamma_{n+1}(G)=[\gamma_n(G),G]$) ;
\item  les foncteurs $\A_n/\A_{n+1} : \G\to\mathbf{Ab}$ où $(\A_n(G))_{n\in\mathbb{N}}$ désigne la filtration de ${\rm Aut}\,(G)$ définie par
$$\A_n(G)={\rm Ker}\,\big({\rm Aut}\,(G)\to {\rm Aut}\,(G/\gamma_{n+1}(G))\big).$$
On a, en particulier, $\A_0(G)={\rm Aut}\,(G)$ et $\A_1(G)=IA(G)$. Cette filtration est dite parfois de Johnson, mais est due à Andreadakis \cite{And}.
\end{enumerate}
En fait, tous les foncteurs précédents se factorisent (à isomorphisme près, et de façon unique) par le foncteur canonique $\G\to\mathbf{S}(\mathbf{ab})$ qui est l'abélianisation sur les objets. Cela provient formellement de ce que les automorphismes intérieurs d'un objet de $\G$ ont une action triviale sur tous ces foncteurs. Nous noterons encore, par abus, de la même façon les foncteurs $\mathbf{S}(\mathbf{ab})\to\mathbf{Ab}$ ainsi obtenus.

\paragraph*{Bas degré}

Le cas $n=1$, dans lequel tous ces foncteurs coïncident, est le seul qu'on sache décrire de façon complète. On renvoie à \cite{And} pour le fait que $\A_1/\A_{2}=IA_{ab}$.
\begin{pr} \label{IA-poly}
Le foncteur $IA_{ab} : \mathbf{S}(\mathbf{ab})\to\mathbf{Ab}$ est fortement polynomial de degré fort $3$. Son degré faible est également de degré $3$.
\end{pr}

\begin{proof}
De l'isomorphisme fonctoriel
$$IA_{ab}(V)\simeq {\rm Hom}_\mathbf{Ab}(V,\Lambda^2(V))\qquad (V\in {\rm Ob}\,\mathbf{S}(\mathbf{ab}))$$
\cite[§\,6]{K-Magnus}, on déduit que $IA_{ab} : \mathbf{S}(\mathbf{ab})\to\mathbf{Ab}$ est la composée des foncteurs $\Delta_{\Z}: \mathbf{S}(\mathbf{ab})\to\mathbf{ab}^{op} \times \mathbf{ab}$ \label{Delta} et ${\rm Hom}(Id, \Lambda^2): \mathbf{ab}^{op} \times \mathbf{ab} \to \mathbf{Ab}$ où $\Delta_{\Z}$ est le foncteur qui à un groupe abélien $G$ associe $(G,G)$ et à un morphisme $(u,v) \in \mathbf{ab}(G',G) \times \mathbf{ab}(G,G')$ de $\mathbf{S}(\mathbf{ab})(G,G')$ associe le morphisme $(u,v)$ de $(\mathbf{ab}^{op} \times \mathbf{ab})((G,G), (G',G'))$. Le premier foncteur est mono\"idal fort et le second est polynomial de degré $3$. Par conséquent, par la proposition \ref{compo-mono-fort}, $IA_{ab}$ est fortement polynomial de degré fort $3$, et son image dans $\st(\mathbf{S}(\mathbf{ab}),\mathbf{Ab})$ est également de degré $3$.
\end{proof}

À partir de $n\geq 2$, la description de $H_n(IA)$ devient très largement inaccessible (pour $n=2$, voir les résultats partiels de Pettet \cite{Pet}). Dans un travail futur, on montrera que ces foncteurs ont tous une image polynomiale dans $\st(\mathbf{S}(\mathbf{ab}),\mathbf{Ab})$, en s'inspirant des travaux de Putman \cite{Pu} et de Church-Ellenberg-Farb-Nagpal \cite{CEFN}. Toutefois, déterminer le degré exact de ces objets semble un problème par\-ti\-cu\-liè\-re\-ment délicat.

\paragraph*{Quelques résultats en degré supérieur}

La compréhension partielle des deux autres familles de foncteurs est facilitée par les foncteurs de Lie. Rappelons que le foncteur d'oubli de la catégorie des algèbres de Lie dans celle des groupes abéliens admet un adjoint à gauche $\mathcal{L}$ qui associe à un groupe abélien son algèbre de Lie libre. Ce foncteur admet une graduation provenant de celle sur l'algèbre tensorielle. Cela fournit, pour tout $n\in\mathbb{N}^*$, un foncteur $\mathcal{L}^n : \mathbf{Ab}\to\mathbf{Ab}$ qui est polynomial de degré $n$ en tant que quotient non nul de la $n$-ème puissance tensorielle.

On dispose d'un épimorphisme $\mathcal{L}^n(G_{ab})\to\gamma_{n}(G)/\gamma_{n+1}(G)$ naturel en le groupe $G$ ; cet épimorphisme est un isomorphisme si $G$ est libre \cite[§\,3]{Cur}.

\begin{pr} \label{scd-poly}
Le foncteur $\gamma_n(IA)/\gamma_{n+1}(IA) : \mathbf{S}(\mathbf{ab})\to\mathbf{Ab}$ est fortement polynomial de degré fort au plus $3n$. 
\end{pr}

\begin{proof}
Le foncteur $\mathcal{L}^n(IA_{ab})$ est isomorphe à la composition des foncteurs  $\Delta_{\Z}: \mathbf{S}(\mathbf{ab})\to\mathbf{ab}^{op} \times \mathbf{ab}$ et $\mathcal{L}^n \circ {\rm Hom}(Id, \Lambda^2)$. Le premier foncteur est mono\"idal fort et le second est fortement polynomial de degré $3n$ d'après la proposition~\ref{deg-compo}. Par la proposition \ref{compo-mono-fort}, on en déduit que le foncteur $\mathcal{L}^n(IA_{ab})$ est fortement polynomial de degré fort au plus $3n$. Les catégories $\Pol^{{\rm fort}}_d(\mathbf{S}(\mathbf{ab}),\mathbf{Ab})$ étant stables par quotients d'après la proposition~\ref{pr-polfor}, on déduit le résultat du rappel précédant l'énoncé.
\end{proof}

On notera que le caractère polynomial des foncteurs $(\gamma_n/\gamma_{n+1})(IA)$ recoupe largement le théorème~7.3.8 de \cite{CEF} par l'intermédiaire de la proposition~\ref{rq-fi1}. 

\begin{pr} \label{Johnson-poly}
Le foncteur $\A_n/\A_{n+1} : \mathbf{S}(\mathbf{ab})\to\mathbf{Ab}$ est faiblement polynomial de degré faible $n+2$. 
\end{pr}

\begin{proof}
Pour tout $n\in\mathbb{N}$, on dispose d'un morphisme naturel injectif classique, dit parfois de Johnson \cite{K-Magnus}
$$(\A_n/\A_{n+1})(V)\hookrightarrow {\rm Hom}_\mathbf{Ab}(V,\mathcal{L}^{n+1}(V))\qquad (V\in {\rm Ob}\,\mathbf{S}(\mathbf{ab}))$$
qui montre que
\begin{equation} \label{maj-deg}
deg(\pi_{\mathbf{S}(\mathbf{ab})}(\A_n/\A_{n+1})) \leq n+2.
\end{equation}

La tensorisation par $\mathbb{Q}$ de son conoyau est polynomiale faible de degré au plus $n$ d'après \cite[Theorem~1]{SIA}. Cet article montre même que le conoyau rationalisé de la composée de ce morphisme avec le morphisme canonique $(\gamma_n/\gamma_{n+1})(IA)\to\A_n/\A_{n+1}$ possède cette propriété. On peut voir aussi l'article \cite{Bar} de Bartholdi à ce sujet. Cela implique
\begin{equation} \label{min-deg}
deg(\pi_{\mathbf{S}(\mathbf{ab})}(\A_n/\A_{n+1})) \geq deg(\pi_{\mathbf{S}(\mathbf{ab})}(\A_n/\A_{n+1}\otimes\mathbb{Q})) \geq n+2,
\end{equation}
d'où la proposition. 
\end{proof}

Outre la détermination du degré, les notions introduites dans le présent article peuvent être appliquées pour comprendre de façon qualitative tous ces foncteurs  sans les déterminer entièrement, notamment en étudiant leurs images dans les catégories quotients $\Pol_d(\mathbf{S}(\mathbf{ab}),\mathbf{Ab})/\Pol_{d-1}(\mathbf{S}(\mathbf{ab}),\mathbf{Ab})$.
Par exemple, les foncteurs $\gamma_n(IA)$ et  $\A_n$ sont des sous-foncteurs de $IA$ tels que  $\gamma_n(IA) \subset \A_n$. 
Andreadakis a conjecturé dans \cite{And} que cette inclusion est une égalité. Sous cette forme, cette conjecture a été infirmée récemment  par des calculs de Bartholdi dans \cite{Bar} et son erratum \cite{Bar-erratum}. L'utilisation  des catégories quotients précédentes devraient permettre de mesurer à quel point la conjecture d'Andreadakis est fausse.

\subsection{Homologie des groupes de congruence}

Soient $I$ un anneau sans unité et $I_+=\mathbb{Z}\oplus I$ l'anneau unitaire obtenu en ajoutant formellement une unité à $I$. 

On considère la catégorie $\mathbf{S}(I_+)$ des $I_+$-modules libres de rang fini avec monomorphismes scindés, le scindage étant donné dans la structure. On dispose d'un foncteur d'automorphismes ${\rm Aut} : \mathbf{S}(I_+)\to\mathbf{Gr}$ qui associe à un $I_+$-module le groupe de ses automorphismes linéaires, ainsi que d'un foncteur ${\rm Aut}\circ (\mathbb{Z}\underset{I_+}{\otimes} -)  : \mathbf{S}(I_+)\to\mathbf{Gr}$. Notons $\Gamma_I$ le foncteur noyau de la transformation naturelle ${\rm Aut}\twoheadrightarrow {\rm Aut}\circ (\mathbb{Z}\underset{I_+}{\otimes}-)$ induite par l'épimorphisme scindé d'anneaux $I_+\twoheadrightarrow I_+/I\simeq\mathbb{Z}$. Ainsi, $\Gamma_I(I_+^n)$ n'est autre que le groupe de congruence $GL_n(I):={\rm Ker}\,\big(GL_n(I_+)\twoheadrightarrow GL_n(\mathbb{Z})\big)$. Cette situation est très analogue à celle discutée dans l'exemple précédent avec les groupes $IA$ ; \cite[chapitre~6]{Dja-cong} donne un cadre général pour l'étude homologique de ce genre de groupes.

Du fait que la conjugaison par les éléments de $GL_n(I)$ opère trivialement sur $H_*(GL_n(I))$, ce groupe est muni d'une action naturelle de $GL_n(\mathbb{Z})$, de sorte qu'on voit facilement que $H_*(\Gamma_I)$ se factorise de manière unique à isomorphisme près par le foncteur $\mathbf{S}(I_+)\to\mathbf{S}(\mathbb{Z})=\mathbf{S}(\mathbf{ab})$. L'étude des foncteurs $\mathbf{S}(\mathbf{ab})\to\mathbf{Ab}$ ainsi obtenus (qu'on notera encore $H_*(\Gamma_I)$ par abus) constitue un problème profond et difficile, relié à la question de l'excision en $K$-théorie algébrique. De fait, l'anneau sans unité $I$ est excisif en $K$-théorie algébrique si et seulement si l'image de $H_d(\Gamma_I)$ dans $\st(\mathbf{S}(\mathbf{ab}),\mathbf{Ab})$ est constante pour tout $d\in\mathbb{N}$. En effet, cette condition équivaut à dire que l'action de $GL_\infty(\mathbb{Z})$ sur $H_d(GL_\infty(I))$ est triviale, ce qui équivaut à l'excisivité en $K$-théorie algébrique pour $I$ d'après Suslin-Wodzicki \cite[§1]{SW}.

Le théorème~4.3 de l'article \cite{SK} de Suslin s'ex\-prime, avec notre vocabulaire, comme suit.

\begin{thm}[Suslin]\label{th-sus}
 Soit $n>0$ un entier tel que ${\rm Tor}^{I_+}_i(\mathbb{Z},\mathbb{Z})=0$ pour $0<i<n$. Alors on dispose, pour $M\in {\rm Ob}\,\mathbf{S}(\mathbf{ab})$, d'un morphisme naturel
$$H_n(\Gamma_I(M))\to {\rm End}_\mathbb{Z}(M)\otimes {\rm Tor}^{I_+}_n(\mathbb{Z},\mathbb{Z})$$
dont noyau et conoyau sont stablement constants dans  $\fct(\mathbf{S}(\mathbf{ab}),\mathbf{Ab})$.
\end{thm}

\begin{cor}
 Si ${\rm Tor}^{I_+}_i(\mathbb{Z},\mathbb{Z})=0$ pour $0<i<n$ et ${\rm Tor}^{I_+}_n(\mathbb{Z},\mathbb{Z})\neq 0$, alors le foncteur $H_n(\Gamma_I) : \mathbf{S}(\mathbf{ab})\to\mathbf{Ab}$ est faiblement polynomial de degré $2$. 
\end{cor}

 Dans le chapitre 6 de \cite{Dja-cong}, on discute une approche pour retrouver et généraliser ce résultat. En particulier, la conjecture 6.15 de \cite[page 64]{Dja-cong} dont la démonstration fait l'objet d'un travail en cours, énonce :
 
\begin{conj}[\cite{Dja-cong}]
 Pour tout anneau sans unité $I$ et tout entier $n\in\mathbb{N}$, le foncteur $H_n(\Gamma_I)$ est faiblement polynomial de degré au plus $2n$.
\end{conj}

 L'étude de l'homologie de $\Gamma_I$, au-delà du premier degré non excisif que traite le théorème~\ref{th-sus}, s'avère difficile. De fait, la conjecture précédente est certainement fausse, même avec une borne de degré moins bonne, si l'on remplace {\em faiblement} polynomial par {\em fortement} polynomial. L'obtention de résultats complètement généraux nécessite probablement de travailler dans la catégorie $\st(\mathbf{S}(\mathbf{ab}),\mathbf{Ab})$.
 
 D'un autre côté, avec des hypothèses supplémentaires sur $I$, inspirées de travaux sur la stabilité homologique pour les groupes linéaires, on peut obtenir le caractère {\em fortement} polynomial des foncteurs d'homologie des groupes de congruence associés à $I$. Les premiers résultats en ce sens, sans utilisation de catégories de foncteurs, apparaissent dans Putman \cite{Pu}. Church, Ellenberg, Farb et Nagpal \cite[Theorem D]{CEFN} ont amélioré le résultat de Putman en utilisant des méthodes fonctorielles reposant sur la seule considération de foncteurs sur $\Theta$. Cela correspond à précomposer nos foncteurs par le foncteur monoïdal fort canonique $\Theta\to\mathbf{S}(\mathbf{ab})$.  Ils ont obtenu le théorème suivant, exprimé avec le vocabulaire du présent article.
 
 \begin{thm}[Church-Ellenberg-Farb-Nagpal]
  Supposons que $I$ est un idéal propre d'un anneau d'entiers de corps de nombres. Alors, pour tout $n\in\mathbb{N}$, le foncteur $H_n(\Gamma_I) : \mathbf{S}(\mathbf{ab})\to\mathbf{Ab}$ est fortement polynomial.
 \end{thm}
 
 Ces auteurs obtiennent des bornes explicites sur le degré fort du foncteur, mais elles sont exponentielles en $n$ --- alors qu'on s'attend à une borne linéaire pour le degré faible (cf. la conjecture précédente).

\section{Classification des foncteurs polynomiaux sur les objets hermitiens}\label{sherm}

Nous nous intéressons dans cette section aux foncteurs polynomiaux sur une catégorie d'objets hermitiens sur une petite catégorie additive à dualité. Cette notion, définie au début de la section \ref{section-hermitiens}, généralise celle des espaces hermitiens (voir exemple \ref{espace-hermitien}).
Le résultat principal de cette section est le théorème suivant :
\begin{thm} \label{thm-pal}
Soient $\A$ une petite catégorie additive munie d'un foncteur de dualité $D : \A^{op}\to\A$ et  $\mathbf{H}(\A)$  la catégorie des {\em objets hermitiens} associés à la situation. On a une équivalence de catégories
$$\Pol_n(\mathbf{H}(\A),\B)/\Pol_{n-1}(\mathbf{H}(\A),\B)\xrightarrow{\simeq}\Pol_n(\A,\B)/\Pol_{n-1}(\A,\B)$$
pour tout $n\in\mathbb{N}$ et toute catégorie de Grothendieck $\B$.
\end{thm}

Pour ce faire, nous donnons dans la proposition~\ref{polfq-abstr} un critère abstrait garantissant que les catégories de foncteurs polynomiaux sur $\widetilde{\M}$ et sur $\A$ sont équivalentes, pour $\M$ un objet de $\mi$ et $\A$ un objet de $\mn$ reliés par une flèche de $\mi$ appropriée $\M\to\A$.
Nous appliquons ensuite ce critère aux espaces hermitiens pour obtenir dans la proposition~\ref{eqcat-fpol} une description des foncteurs polynomiaux sur $\widetilde{\mathbf{H}(\A)}$. En combinant cette proposition au résultat de la section~\ref{section2} on en déduit le théorème~\ref{thm-pal}.

\subsection{Un critère abstrait}

Le résultat principal de ce paragraphe (proposition \ref{polfq-abstr}) permet de comparer certaines catégories de foncteurs polynomiaux. Cette proposition technique, dont on se servira principalement dans le cas $\M=\mathbf{H}(\A)$ au paragraphe \ref{section-hermitiens}, peut s'appliquer à d'autres situations, comme l'illustrera la proposition~\ref{pcm}.

La démonstration de la proposition \ref{polfq-abstr} repose sur plusieurs lemmes formels, où interviendront des catégories sans unités baptisées {\em semi-catégories} dans \cite[§\,4]{Mi}. On rappellera ci-dessous les quelques notions relatives aux  semi-catégories qui nous seront utiles.

On commence par un énoncé d'équivalence de Morita simple et général.
\begin{lm}\label{LM0}
Soit $\A$ un objet de $\mn$, $\A'$ une sous-catégorie pleine et mono\"idale de $\A$ et $\iota: \A' \to \A$ le foncteur d'inclusion. Si tout objet de $\A$ est facteur direct d'un objet de $\A'$, alors, pour toute catégorie abélienne $\B$, les deux foncteurs $\iota^*$ de précomposition  par $\iota$ du diagramme commutatif suivant 
sont des équivalences de catégories :
$$\xymatrix{
\Pol_n(\A,\B) \ar[r]^{\iota^*} \ar[d]& \Pol_n(\A',\B) \ar[d]\\
\fct(\A,\B)\ar[r]_{\iota^*} & \fct(\A',\B)
}$$
\end{lm}

\begin{proof}
Rappelons qu'une catégorie $\C$ est dite complète au sens de Cauchy lorsque tous les idempotents de $\C$ se scindent \cite[Definition 6.5.8]{Borceux}. Comme $\B$ est abélienne, elle est complète au sens de Cauchy. D'après \cite[Proposition 6.5.9]{Borceux} il existe une petite catégorie compl\`ete au sens de Cauchy $\bar{\A}$ dont $\A$ est une sous-catégorie pleine telle que le foncteur $i^*: \fct(\bar{\A},\B)\to \fct(\A,\B)$, obtenu par précomposition avec le foncteur d'inclusion $i: \A \to \bar{\A}$, soit une équivalence de catégories. Comme $\bar{\A}$ est complète au sens de Cauchy, d'après \cite[Proposition 6.5.9]{Borceux} le foncteur $i \circ \iota: \A' \to \bar{\A}$ s'étend de manière unique (à isomorphisme près) en un foncteur $\overline{i \circ \iota}: \bar{\A'} \to \bar{\A}$ et la précomposition par ce foncteur fournit une équivalence de catégories $\overline{i \circ \iota}^*: \fct(\bar{\A},\B) \to  \fct(\bar{\A'},\B)$. De ces deux équivalences de catégories on déduit que $\iota^* : \fct(\A,\B)\to\fct(\A',\B)$ est une équivalence de catégories.

Du fait que $\A'$ est une sous-catégorie monoïdale de $\A$, le diagramme suivant commute
$$\xymatrix{\fct(\A,\B)\ar[r]\ar[d]_{\iota^*} & \fct(\A^{n+1},\B)\ar[d]\\
\fct(\A',\B)\ar[r] & \fct(\A'^{n+1},\B)
}$$
où les flèches horizontales sont les effets croisés $cr_{n+1}$ (cf. définition~\ref{ecr-gal}) et les flèches verticales les restrictions. Ces dernières sont des équivalences de catégories d'après ce qui précède --- en effet, tout objet de $\A^{n+1}$ est facteur direct d'un objet  de $\A'^{n+1}$. Il s'en suit, en utilisant la proposition~\ref{p-ecr} et le diagramme commutatif précédent, que, pour tout foncteur $F$ de $\fct(\A,\B)$, $F$ appartient à $\Pol_n(\A,\B)$ si et seulement si $\iota^*(F)$ appartient à $\Pol_n(\A',\B)$. Par conséquent, l'équivalence de catégories $\iota^* : \fct(\A,\B)\to\fct(\A',\B)$ induit une équivalence de catégories $\Pol_n(\A,\B)\to\Pol_n(\A',\B)$, d'où le lemme.
\end{proof}

Une \textit{semi-catégorie} $\C$ est définie de la même manière qu'une catégorie mis à part l'existence de morphismes identités. Ainsi, pour $C$ un objet de $\C$, l'ensemble $\C(C,C)$ peut être vide. Une \textit{sous-semi-catégorie} $\C'$ d'une catégorie est par définition une semi-catégorie vérifiant les mêmes conditions qu'une sous-catégorie, à l'exception de l'appartenance des morphismes identités à $\C'$. Une sous-semi-catégorie est dite \textit{idéale} si elle est stable par composition à droite et à gauche par des morphismes arbitraires.
Un \textit{semi-foncteur} entre semi-catégories vérifie les mêmes conditions qu'un foncteur, excepté la préservation des identités.

Dans l'énoncé ci-dessous, ainsi que dans la suite de ce paragraphe, la notion de {\em catégorie quotient} qui apparaît est celle qui consiste, sans changer les objets, à quotienter les ensembles de morphismes par une certaine relation d'équivalence. C'est exactement la notion de \cite[chap.~II, §\,8]{ML}. Elle diffère de celle de catégorie abélienne quotient, qui est un cas particulier de calcul des fractions, abondamment utilisée dans le reste de cet article.

\begin{lm}\label{LM1}
Soit $F: \D_1 \to \D_2$ un foncteur, $\C$ une sous-semi-catégorie idéale de $\D_1$ et $\iota: \C \to \D_1$ le semi-foncteur d'inclusion. Si $F \circ \iota$ est pleinement fidèle et essentiellement surjectif, alors :
\begin{enumerate}
\item
pour tout objet $c$ de  $\C$, $h_c:=(F \circ \iota)^{-1}(1_{F\circ \iota(c)})$ est un idempotent de $\C(c,c)$;
\item
le foncteur $F$ se factorise de manière unique sous la forme
 $$\xymatrix{
 \D_1 \ar[r]^F \ar[d]& \D_2\\
 \D_1/\langle h_c=1_c \rangle \ar@{-->}[ur]_{\tilde{F}}
 }$$
 et le foncteur $\tilde{F}$ est une équivalence de catégories;
 \item
 la semi-catégorie $\C$ est une catégorie dans laquelle, pour tout objet $c$, $1_{c}=h_c$.
 \end{enumerate}
\end{lm}

Ce résultat élémentaire, dont la démonstration est laissée au lecteur, est entièrement analogue à l'observation usuelle suivante en théorie des anneaux : si $\varphi : B\to A$ est un morphisme d'anneaux (unitaires) et $I$ un idéal de $B$ tel que la restriction à $I$ de $\varphi$ (qui est un morphisme d'anneaux sans unité) est bijective, alors :
\begin{itemize}
 \item l'image inverse de $1_A$ par cette bijection est un idempotent central $e$ de $B$ ;
 \item $\varphi$ induit un isomorphisme d'anneaux $B/(e-1)\xrightarrow{\simeq} A$ ;
 \item l'anneau sans unité $I$ est en fait unitaire, d'unité $e$. C'est l'idéal de $B$ engendré par l'idempotent central $e$.
\end{itemize}

\begin{nota} \label{notaI}
Soient $\Phi: \A \to \A'$ un morphisme de $\mn$, $x_1, \ldots, x_n$ des objets de $\A$ et $\Z[\A]$ la catégorie préadditive associée à $\A$. On note :
\begin{enumerate}
\item $\I^n(\A)$ la sous-semi-catégorie idéale préadditive de $\Z[\A]$ engendrée par les idempotents $e_n(x_1,\dots,x_n)$ considérés dans le corollaire \ref{cor-eps};
\item $\J(\A, \Phi)$ la sous-semi-catégorie idéale préadditive de $\Z[\A]$ ayant les mêmes objets que $\Z[\A]$ et telle que 
$$\J(\A, \Phi)(a_1, a_2)={\rm Ker}\big(\Z[\A(a_1,a_2)] \xrightarrow{\Z[\Phi]} \Z[\A'(\Phi(a_1),\Phi(a_2))]\big).$$
\end{enumerate}
\end{nota}
\begin{lm}\label{LM2}
Soient $n\in \mathbb{N}$ et $\Phi: \A \to \A'$ un morphisme de $\mn$. Avec les notations précédentes, si $\J(\A, \Phi) \subset \I^n(\A)$ et si $\Phi$ est plein et essentiellement surjectif, alors  la précomposition par $\Phi$ induit une équivalence de catégories
$$\Pol_n(\A',\B)\xrightarrow{\simeq}\Pol_n(\A,\B)$$
pour toute catégorie abélienne $\B$.
\end{lm}
\begin{proof}
On déduit de la proposition \ref{df2-ecr} que la catégorie $\Pol_n(\A',\B)$ est équivalente à celle des foncteurs additifs de $\Z[\A']/\I^n(\A')$ dans $\B$. Soient $a_1$ et $a_2$ deux objets de $\Z[\A]/\I^n(\A)$, on a le diagramme commutatif suivant 
$$\xymatrix{
 & \I^n(\A)(a_1, a_2) \ar@{->>}[r] \ar@{^{(}->}[d]& \I^n(\A')(\Phi(a_1), \Phi(a_2)) \ar@{^{(}->}[d]\\
 \J(\A,\Phi)(a_1, a_2) \ar@{^{(}->}[r] \ar@{^{(}->}[ur]& \Z[\A](a_1,a_2) \ar@{->>}[r] \ar@{->>}[d]&\Z[\A'](\Phi(a_1), \Phi(a_2)) \ar@{->>}[d]\\
 & \Z[\A]/ \I^n(\A)(a_1, a_2) \ar@{->>}[r] & \Z[\A']/ \I^n(\A')(\Phi(a_1), \Phi(a_2)) \\
}$$
où les flèches horizontales de droite, induites par le foncteur plein $\Phi$, sont surjectives. Une chasse au diagramme permet de déduire de l'hypothèse $\J(\A, \Phi) \subset \I^n(\A)$ que la flèche horizontale inférieure est également injective. Il s'ensuit que le foncteur
$$\Z[\A]/ \I^n(\A) \to \Z[\A']/ \I^n(\A')$$
induit par $\Phi$ est pleinement fidèle. Il est essentiellement surjectif puisque $\Phi$ l'est par hypothèse, ce qui achève la démonstration.
\end{proof}

\begin{pr}\label{polfq-abstr}
 Soient $\M$ un objet de $\mi$, $\A$ un objet de $\mn$, $\Phi : \M\to\A$ une flèche de $\mi$ et $\C$ une sous-catégorie monoïdale de $\A$ ayant les mêmes objets. On fait les hypothèses suivantes :
\begin{enumerate}
 \item\label{fc} si $f$ et $g$ sont deux morphismes composables de $\A$ dont la composée $g \circ f$ appartient à $\C$, alors $f$ est dans $\C$ ;\item\label{es} pour tout objet $a$ de $\A$, il existe un objet $x$ de $\A$ tel que $a\oplus x$ appartienne à l'image essentielle de $\Phi$ ;
\item\label{val} $\Phi$ est à valeurs dans $\C$ ;
\item\label{fig} pour tout morphisme $f : x\to y$ de $\M$, il existe un objet $z$ et un morphisme $g : y\to x\oplus z$ de $\M$ tels que $g.f$ soit égal au morphisme canonique $x\to x\oplus z$ ;
\item \label{ppf}
\begin{enumerate}
\item \label{ppf1} pour tous objets $a$, $b$ de $\M$ et tout morphisme $f : \Phi(a)\to\Phi(b)$ de $\A$, il existe un objet $t$ et un morphisme $\varphi : a\to b\oplus t$ de $\M$ tels que la composée
$$\Phi(a)\xrightarrow{\Phi(\varphi)}\Phi(b\oplus t)=\Phi(b)\oplus\Phi(t)\twoheadrightarrow\Phi(b)$$
(où la deuxième flèche est le morphisme canonique) soit égale à $f$ et que la composée
$$\Phi(a)\xrightarrow{\Phi(\varphi)}\Phi(b\oplus t)=\Phi(b)\oplus\Phi(t)\twoheadrightarrow\Phi(t)$$
appartienne à $\C$.
\item \label{ppf2}
De plus, si $\varphi' : a\to b\oplus t'$ est un autre morphisme possédant la même propriété, alors il existe un objet $u$ et des morphismes $\psi : t\to u$, $\psi' : t'\to u$ de $\M$ tels que le diagramme
$$\xymatrix{a\ar[r]\ar[r]^\varphi\ar[d]_{\varphi'} & b\oplus t\ar[d]^{b\oplus\psi} \\
 b\oplus t'\ar[r]^{b\oplus\psi'} & b\oplus u
}$$
commute.  
\end{enumerate}
\end{enumerate}

Alors la précomposition par le prolongement $\widetilde{\M}\to\A$ de $\Phi$ donné par la proposition~\ref{proltilde} induit une équivalence de catégories
$$\Pol_n(\A,\B)\xrightarrow{\simeq}\Pol_n(\widetilde{\M},\B)$$
pour tout $n\in\mathbb{N}$ et toute catégorie abélienne $\B$.
\end{pr}

\begin{rem}
Certaines hypothèses de l'énoncé précédent peuvent s'interpréter de la manière suivante :
\begin{itemize}
\item l'hypothèse \ref{es} signifie que $\Phi$ est essentiellement surjectif <<~à un décalage près~>> ;
\item l'hypothèse \ref{fig} veut dire que tout morphisme de $\M$ devient inversible à gauche dans $\widetilde{\M}$ ;
\item l'hypothèse \ref{ppf1} indique qu'une certaine restriction du foncteur $\widetilde{\M} \to \A$ induit par $\Phi$ est pleine ;
\item l'hypothèse \ref{ppf2} exprime que cette restriction est fidèle.
\end{itemize}
La présentation concrète de ces hypothèses dans la proposition \ref{polfq-abstr}, qui ne fait pas intervenir la catégorie $\widetilde{\M}$, a l'avantage de rendre plus aisée leur vérification dans les exemples.
\end{rem}

\begin{rem}\label{appl-herm}
Nous appliquerons dans la section \ref{section-hermitiens} la proposition \ref{polfq-abstr} à $\A$ une petite catégorie additive à dualité, $\M=\mathbf{H}(\A)$, $\Phi$ le foncteur d'oubli et $\C=\mathbf{M}(\A)$ la catégorie introduite dans l'exemple \ref{exini}.\ref{exm}.
\end{rem}

\begin{proof}[Démonstration de la proposition \ref{polfq-abstr}]
L'hypothèse~\ref{es} garantit que l'image essentielle $\A'$ de $\Phi$ vérifie les conditions du lemme~\ref{LM0}. On en déduit que $\Pol_n(\A,\B)\xrightarrow{\simeq}\Pol_n(\A',\B)$. Quitte à remplacer $\A$ par $\A'$ on peut donc supposer dans la suite que $\Phi$ est essentiellement surjectif.
 
 Soient $f:a \to b \oplus t$, $g: a \to b \oplus u$ et $h: t \to u$ des morphismes dans $\M$ tels que le diagramme suivant soit commutatif
 $$\xymatrix{
 a \ar[r]^-f \ar[rd]_g & b \oplus t \ar[d]^{1_b \oplus h}\\
 & b \oplus u.
 }$$
 En appliquant le foncteur $\Phi$ on obtient le diagramme commutatif suivant dans $\A$
 $$\xymatrix{
 \Phi(a) \ar[r]^-{\Phi(f)} \ar[rd]_{\Phi(g)} & \Phi(b \oplus t)=\Phi(b) \oplus \Phi(t) \ar[d]^-{\Phi(1_b \oplus h)} \ar@{->>}[r]^-q &\Phi(t) \ar[d]^{\Phi(h)}\\
 & \Phi(b \oplus u)=\Phi(b) \oplus \Phi(u) \ar@{->>}[r]_-p &\Phi(u).
 }$$
 Comme $\Phi(h)$ appartient à $\C$ par l'hypothèse \ref{val}, on déduit de l'hypothèse~\ref{fc} que $q \circ \Phi(f)$ appartient à $\C$ si et seulement si $p \circ \Phi(g)$ est dans $\C$. 
 
On définit ainsi une sous-semi-catégorie $\hat{\A}$ de $\widetilde{\M}$ ayant les mêmes objets : ses morphismes $a\to b$ sont les classes dans $\widetilde{\M}$ des morphismes $f : a\to b\oplus t$ de $\M$ tels que la composée
$$\Phi(a)\xrightarrow{\Phi(f)}\Phi(b\oplus t)=\Phi(b)\oplus\Phi(t)\twoheadrightarrow\Phi(t)$$
appartienne à $\C$. Ce qui précède montre que cette condition ne dépend pas du choix du représentant $f\in\M$ du morphisme de $\widetilde{\M}(a,b)$ initial.
 
Montrons que $\hat{\A}$ est une sous-semi-catégorie idéale de $\widetilde{\M}$. Soient $a \to b$ et $c \to d$ des morphismes de $\hat{\A}$ représentés, respectivement, par les morphismes $f : a \to b \oplus t$ et $h : c \to d \oplus v$ de $\M$ et $b \to c$ un morphisme de $\widetilde{\M}$ représenté par le morphisme $g : b \to c \oplus u$ de $\M$. Considérons le diagramme commutatif suivant de $\A$
$$\xymatrix{
\Phi(a) \ar[r]^-{\Phi((g\oplus t) \circ f)} \ar[dd]_-{\Phi(f)} \ar[ddr] & \Phi(c) \oplus \Phi(u\oplus t) \ar@{->>}[d]\\
& \Phi(u \oplus t)\ar@{->>}[d]\\
\Phi(b) \oplus \Phi(t) \ar@{->>}[r]& \Phi(t).
}$$
La flèche oblique de ce diagramme est dans $\C$ car $a \to b$ appartient à $\hat{\A}$. En appliquant l'hypothèse~\ref{fc}, on obtient que la composée $a \to b\to c$ est un morphisme de $\hat{\A}$. 

D'autre part, on a le diagramme commutatif suivant
$$\xymatrix{
\Phi(b) \ar[rr]^{\Phi((h \oplus u) \circ g)}  \ar[d]_-{\Phi(g)} \ar[drr]&& \Phi(d) \oplus \Phi(v) \oplus \Phi(u)  \ar@{->>}[d]\\
\Phi(c) \oplus \Phi(u) \ar@{->>}[r]_{\Phi(h) \oplus 1_{\Phi(u)}} &\Phi(d \oplus v) \oplus  \Phi(u) \ar@{->>}[r]& \Phi(v) \oplus \Phi(u).
}$$
La flèche oblique de ce diagramme est dans $\C$ car $c \to d$ appartient à $\hat{\A}$ et $\C$ est mono\"idale, et $\Phi(g)$ est dans $\C$ par l'hypothèse~\ref{val}. On en déduit que $b\to c \to d$ est  un morphisme de $\hat{\A}$.

L'hypothèse~\ref{ppf} montre que le semi-foncteur\,$$\hat{\A}\to\widetilde{\M}\to\A$$
composé de l'inclusion et du prolongement de $\Phi$ donné par la proposition~\ref{proltilde} est pleinement fidèle : l'hypothèse~\ref{ppf1} montre sa plénitude, et~\ref{ppf2} sa fidélité. Comme il est également essentiellement surjectif par l'hypothèse faite en début de démonstration, le lemme~\ref{LM1} montre que la catégorie $\A$ est équivalente au quotient de $\widetilde{\M}$ par les relations $h_x=1_x$ (pour tout $x\in {\rm Ob}\,\M$), où $h_x$ est l'idempotent de $\hat{\A}(x,x)$ image inverse de ${\rm Id}_{\Phi(x)}$.

Le lemme~\ref{LM2} implique donc qu'il suffit, pour conclure, d'établir que, pour tout $n\in\mathbb{N}^*$, dans la catégorie préadditive $\mathbb{Z}[\widetilde{\M}]$, les flèches $h_x-1_x$ appartiennent à l'idéal $\I^n(\widetilde{\M})$ introduit dans la notation \ref{notaI}.

Pour cela, on note que l'hypothèse~\ref{ppf1} entraîne (par récurrence sur $n$) l'existence, pour tout $x\in {\rm Ob}\,\widetilde{\M}$ d'un morphisme $f : x\to x_1\oplus\dots\oplus x_n$ de $\M$ tel que chaque composée
$$\Phi(x)\xrightarrow{\Phi(f)}\Phi(x_1\oplus\dots\oplus x_n)=\Phi(x_1)\oplus\dots\oplus\Phi(x_n)\twoheadrightarrow\Phi(x_i)$$
appartienne à $\C$, ce qui implique (par l'hypothèse~\ref{fc}) que, pour toute partie non vide $I$ de $\mathbf{n}$, la composée
$$\Phi(x)\xrightarrow{\Phi(f)}\Phi(x_1\oplus\dots\oplus x_n)\twoheadrightarrow\bigoplus_{i\in I}\Phi(x_i)$$
appartient également à $\C$. Donc $\epsilon_J(x_1,\dots,x_n).f$ appartient à $\hat{\A}$ (rappelons que les idempotents $\epsilon_J$ sont introduits dans la définition~\ref{epsilon}), soit
$$\epsilon_J(x_1,\dots,x_n).f.(h_x-1_x)=0,$$
pour toute partie {\em stricte} $J$ de $\mathbf{n}$. Par suite, comme $\epsilon_\mathbf{n}(x_1,\dots,x_n)=1$,
$$f.(h_x-1_x)=\sum_{J\subset\mathbf{n}}(-1)^{n-|J|}\epsilon_J(x_1,\dots,x_n).f.(h_x-1_x)$$
$$=e_n(x_1,\dots,x_n).f.(h_x-1_x)\in\I^n(\widetilde{\M}).$$

Par ailleurs, l'hypothèse~\ref{fig} montre que tout morphisme de $\M$ devient inversible à gauche dans $\widetilde{\M}$, comme les morphismes canoniques $t\to t\oplus u$. La relation précédente implique donc $h_x-1_x\in\I^n(\widetilde{\M})$, ce qui termine la démonstration.
\end{proof}

\begin{rem}
Il est très naturel de tenter d'appliquer la proposition~\ref{polfq-abstr} à la situation suivante, où $\A$ est une petite catégorie additive non triviale : $\M=\mathbf{S}(\A)$, $\C=\mathbf{M}(\A)$ et $\Phi : \mathbf{S}(\A)\to\A$ désigne le foncteur d'oubli. Néanmoins, la conclusion de la proposition tombe en défaut dans ce cas. De fait, on vérifie sans peine que toutes ses hypothèses sont vérifiées, {\em à l'exception de l'hypothèse~\ref{ppf2}}. En effet, soient $a$ un objet non nul de $\A$, $\varphi$ et $\varphi'$ les morphismes  de $\mathbf{S}(\A)(a,a\oplus a)$ donnés par les couples de morphismes de $\A$ $(\Delta : a\to a\oplus a,p_1 : a\oplus a\to a)$ et $(\Delta : a\to a\oplus a,p_2 : a\oplus a\to a)$ respectivement, où $\Delta$ désigne la diagonale et $p_1$ (resp. $p_2$) la projection sur le premier (resp. deuxième) facteur. Ces deux morphismes vérifient les conditions requises dans le point \ref{ppf} de la proposition~\ref{polfq-abstr} en prenant pour $f$ le morphisme $1_a\in\A(a,a)$, mais il n'est pas possible de trouver des morphismes $\psi$ et $\psi'$ satisfaisant à la condition~\ref{ppf2}. En effet, en appliquant le foncteur canonique $\mathbf{S}(\A)\to\A^{op}$ au diagramme commutatif fourni par la condition~\ref{ppf2} on obtient une contradiction. On peut toutefois utiliser la proposition~\ref{polfq-abstr} pour étudier les foncteurs polynomiaux sur $\M:=\mathbf{S}(\A)$, mais en utilisant le foncteur monoïdal canonique $\mathbf{S}(\A)\to\A^{op}\times\A$. On renvoie pour cela à la remarque~\ref{appl-herm} et à l'équivalence $\mathbf{S}(\A)\simeq\mathbf{H}(\A^{op}\times\A)$ donnée dans la remarque~\ref{SA-hermitienne} ci-après.
\end{rem}

La stratégie esquissée ci-dessus pour étudier les foncteurs polynomiaux sur $\mathbf{S}(\A)$ fonctionne en revanche directement pour les foncteurs polynomiaux sur $\mathbf{M}(\A)$, comme le montre la démonstration de la proposition suivante. Celle-ci constitue un échauffement pour les considérations hermitiennes de la section~\ref{section-hermitiens} qui sont analogues mais plus techniques.

\begin{pr}\label{pcm}
Soient $\A$ une petite catégorie additive, $\B$ une catégorie abélienne et $n$ un entier naturel. Le foncteur d'inclusion $\mathbf{M}(\A)\to\A$ induit une équivalence de catégories
$$\Pol_n(\A,\B)\xrightarrow{\simeq}\Pol_n(\widetilde{\mathbf{M}(\A)},\B).$$
\end{pr}

\begin{proof}
 On applique la proposition~\ref{polfq-abstr} au foncteur d'inclusion $\Phi : \mathbf{M}(\A)\to\A$ et à la sous-catégorie monoïdale $\C=\mathbf{M}(\A)$ de $\A$. Ses hypothèses~\ref{fc}, \ref{es} et \ref{val} sont manifestement vérifiées.
 
 Soit $f : A\to B$ un morphisme de $\mathbf{M}(\A)$, notons $f^! : B\to A$ un morphisme de $\A$ tel que $f^!.f=1_A$. Des notations analogues seront utilisées dans la suite de cette démonstration. Le morphisme $g : B\to A\oplus B$ de $\A$ de composantes $f^!$ et $1_B-f.f^!$ appartient à $\mathbf{M}(\A)$ car sa composée avec le morphisme de composantes $f$ et $1_B$ est l'identité de $B$. La composée $g\circ f$ est le morphisme canonique $A\to A\oplus B$. Cela montre l'hypothèse~\ref{fig}.
 
 Si $f : A\to B$ est un morphisme de $\A$, la considération du morphisme $A\to B\oplus A$ dont les composantes sont $f$ et $1_A$ établit l'hypothèse~\ref{ppf1}.
 
 Afin de vérifier l'hypothèse~\ref{ppf2}, considérons deux morphismes $A\xrightarrow{\left(\begin{array}{c}f\\
 g                                                                                    
                                                                                       \end{array}
\right)}B\oplus T$ et $A\xrightarrow{\left(\begin{array}{c}f\\
 h                                                                                    
                                                                                       \end{array}
\right)}B\oplus U$ de $\A$ dont les premières composantes $A\xrightarrow{f}B$ sont identiques et dont les secondes composantes $A\xrightarrow{g}T$ et $A\xrightarrow{h}U$ appartiennent à $\mathbf{M}(\A)$. Le morphisme $\alpha:=\left(\begin{array}{c}1_T\\
 h.g^!                                                                                                                                                                           \end{array}\right) : T\to T\oplus U$ appartient à $\mathbf{M}(\A)$, de même que $\beta:=\left(\begin{array}{c}g.h^!\\
    1_U                                                                                                                                                                        \end{array}\right) : U\to T\oplus U$, et l'on a $\beta.h=\alpha.g\in\A(A,T\oplus U)$. Par conséquent, on dispose dans $\A$ d'un diagramme
    commutatif
    $$\xymatrix{A\ar[rr]^{\left(\begin{array}{c}f\\
 g                                                                                                                                          \end{array}\right)}\ar[d]_{\left(\begin{array}{c}f\\
 h                                                                                                                                                                         \end{array}\right)} & & B\oplus T\ar[d]^{\left(\begin{array}{cc} 1_B & 0\\
0 &  \alpha                                                                                    
            \end{array}\right)}\\            
    B\oplus U\ar[rr]_-{\left(\begin{array}{cc} 1_B & 0\\
0 &  \beta                                                                                    
            \end{array}\right)} & & B\oplus (T\oplus U)
    }$$
dans $\A$ qui prouve que l'hypothèse~\ref{ppf2} de la proposition~\ref{polfq-abstr} est satisfaite, ce qui achève la démonstration.
\end{proof}

En combinant ce résultat au théorème~\ref{tilde-pol}, on obtient :

\begin{cor}\label{ccm}
 Soient $\A$ une petite catégorie additive, $\B$ une catégorie de Grothendieck et $n$ un entier naturel. Le foncteur d'inclusion $\mathbf{M}(\A)\to\A$ induit une équivalence de catégories
$$\Pol_n(\A,\B)/\Pol_{n-1}(\A,\B)\xrightarrow{\simeq}\Pol_n(\mathbf{M}(\A),\B)/\Pol_{n-1}(\mathbf{M}(\A),\B).$$
\end{cor}

 \subsection{Application aux catégories d'objets hermitiens} \label{section-hermitiens}
On commence par rappeler la définition d'une catégorie additive à dualité qui est le cadre classique pour traiter la théorie générale des formes hermitiennes ou symplectiques. Ces rappels sont issus de \cite[chap.~II, §\,2]{Kn} et de \cite[§\,4]{Dja-JKT}  dont on suit de près la terminologie.

Soit $\A$ une petite catégorie additive. Un {\it foncteur de dualité sur $\A$} est un foncteur $D: \A^{op} \to \A$  vérifiant les trois conditions suivantes:
\begin{enumerate}
\item $D$ est auto-adjoint: on dispose d'isomorphismes naturels
$$\sigma_{X,Y}: \A(X,DY) \xrightarrow{\simeq} \A(Y,DX)\ ;$$
\item $D$ est symétrique (i.e. $\sigma_{Y,X}=\sigma_{X,Y}^{-1}$);
\item l'unité ${\rm Id}_{\A} \to D^2$ (qui co\"incide avec la coünité par symétrie) est un isomorphisme (en particulier, $D$ est une équivalence de catégories).
\end{enumerate}
On notera $\bar{a}$ pour $\sigma_{X,Y}(a)$, où $a \in \A(X,DY)$.

\begin{ex} (Cas fondamental) \label{espace-hermitien}
Soit $A$ un anneau muni d'une involution (i.e. un anti-automorphisme de $A$ dont le carré est l'identité). On note $\mathbf{P}(A)$ (un squelette de) la catégorie des $A$-modules à gauche projectifs de type fini. Pour tout $A$-module à gauche $M$, l'involution de $A$ permet de munir le groupe abélien ${\rm Hom}_A(M,A)$ d'une action naturelle {\em à gauche} de $A$. Par conséquent, $D={\rm Hom}_A(-,A)$ est un foncteur de dualité sur $\mathbf{P}(A)$.
\end{ex}

Soient  $\varepsilon\in\{-1,1\}$ et $\mathbb{Z}_{\varepsilon}$ la représentation $\Z$ de $\Z/2$ avec action de l'élément non trivial par $\varepsilon$. On note $TD^2 :  \A^{op}\to\mathbf{Ab}$ le foncteur $\A^{op}\xrightarrow{({\rm Id},D)}\A^{op}\times\A\xrightarrow{{\rm Hom}_\A}\mathbf{Ab}$ qui est muni d'une action de $\Z /2$ induite par l'involution $\sigma_{X,X}$ de $\A(X,DX)$. On en définit un sous-foncteur $\Gamma D^2_\varepsilon$ et un quotient  $S D^2_\varepsilon$ par:
$$\Gamma D^2_\varepsilon=\mathrm{Hom}_{\Z/2}(\Z_\varepsilon,-)\circ TD^2 \qquad \mathrm{et} \qquad S D^2_\varepsilon= (\Z_\varepsilon \underset{\mathbb{Z}/2}{\otimes} -) \circ TD^2.$$
Le choix des notations fait référence aux deuxièmes puissances tensorielles, symétriques et extérieures et au fait que ces foncteurs dépendent du foncteur $D$.

On note $\bar{S}D^2_\varepsilon$ l'image du morphisme de norme $N: S D^2_\varepsilon \to \Gamma D^2_\varepsilon$ défini par $N(1 \otimes x)=x+ \varepsilon \bar{x}$. Notons qu'on a ainsi des transformations naturelles ${S}D^2_\varepsilon \to {T}D^2$ composée de la norme $S D^2_\varepsilon \to \Gamma D^2_\varepsilon$ et de l'inclusion $ \Gamma D^2_\varepsilon \hookrightarrow {T}D^2$ et $\bar{S}D^2_\varepsilon \to  {T}D^2$ composée des inclusions $\bar{S}D^2_\varepsilon \hookrightarrow  \Gamma D^2_\varepsilon $ et $ \Gamma D^2_\varepsilon \hookrightarrow {T}D^2$.

On suppose que $T : \A^{op}\to\mathbf{Ab}$ est l'un des foncteurs $S D^2_\varepsilon$ ou $\bar{S}D^2_\varepsilon$. On note $\A^T$ la catégorie d'éléments associée à $T$ ou plutôt à sa composée avec le foncteur d'oubli vers les ensembles. On rappelle que cette catégorie a pour objets les couples $(A,u)$ où $A$ est un objet de $\A$ et $u \in T(A)$ et un morphisme $(A,u) \to (B,v)$ est un morphisme $f: A \to B$ dans $\A$ tel que $T(f)(v)=u$. On note $\pi^T : \A^T\to\A$ le foncteur d'oubli.

\begin{defi}\label{df-herm}
La catégorie $\mathbf{H}^T(\A)$ des {\em objets hermitiens} associés à la situation est la sous-catégorie pleine  de $\A^T$ formée des objets $(A,u)$ pour lesquels l'image de $u\in T(A)$ dans $TD^2(A)=\A(A,DA)$ est inversible.
\end{defi}
Nous omettrons souvent l'exposant $T$ dans la notation et noterons $\mathbf{H}(\A)$ cette catégorie.
La catégorie $\A^T$ est la catégorie des objets hermitiens éventuellement dégénérés

\begin{ex} (Suite du cas fondamental) \label{ex-hermitien}
Soit $A$ un anneau muni d'une involution et $\mathbf{P}(A)$ la catégorie additive munie de la dualité introduite à l'exemple \ref{espace-hermitien}. Si $M$ est un objet de $\mathbf{P}(A)$, un élément de $S D^2_{1}(M)$
est une forme hermitienne sur $M$ (éventuellement dégénérée ; si $A$ est commutatif et l'involution triviale c'est donc une forme quadratique) et un élément de $\bar{S} D^2_{-1}(M)$ est une forme symplectique (éventuellement dégénérée).
\end{ex}

\begin{rem}\label{SA-hermitienne}
 Soit $\A$ une petite catégorie additive. La catégorie additive $\A^{op}\times\A$ possède une dualité naturelle, donnée par l'échange des deux facteurs du produit cartésien. La catégorie $\mathbf{S}(\A)$ (définie dans l'exemple~\ref{exini}.\ref{exs}) s'identifie à $\mathbf{H}(\A^{op}\times\A)$.
\end{rem}

Les catégories $\A^T$ et $\mathbf{H}(\A)$ sont munies de la structure monoïdale symétrique donnée par la somme orthogonale : la somme de $(A,x)$ et $(B,y)$ est $A\oplus B$ muni de l'image de $(x,y)$ par le morphisme canonique $T(A)\oplus T(B)\to T(A\oplus B)$. Ainsi $\A^T$ et $\mathbf{H}(\A)$ sont des objets de $\mi$ dont l'objet initial est $(0,0)$ et les foncteurs d'inclusion $\mathbf{H}(\A)\hookrightarrow\A^T$ et $\pi^T : \A^T\to\A$ sont des flèches de $\mi$. Par contre, $\A^T$ et $\mathbf{H}(\A)$  ne sont généralement pas des objets de $\mn$ car pour un objet $(V,x)$ de  $\A^T$ tel que $x \in T(V)$ est non nul, il n'y a pas de morphisme de $(V,x)$ dans $(0,0)$ dans la catégorie $\A^T$.

Pour appliquer la proposition~\ref{polfq-abstr}, on utilisera le résultat élémentaire suivant, qui est partiellement contenu dans \cite[lemme~5.2]{Dja-JKT}.

\begin{lm}\label{lm-herm}
 \begin{enumerate}
  \item Soit $(A,x)$ un objet de $\A^T$. Il existe un morphisme $\varphi : (A,x)\to (H,v)$ de $\A^T$ dont le but $(H,v)$ appartient à $\mathbf{H}(\A)$ et dont l'image par $\pi^T$ appartient à $\mathbf{M}(\A)$. De plus, si $f : (A,x)\to (E,w)$ est un autre morphisme possédant ces propriétés, on peut trouver un morphisme  $g : (H,v)\to (E,w)\oplus (H,v)$ de $\mathbf{H}(\A)$ tel que le diagramme
$$\xymatrix{(A,x)\ar[r]^\varphi\ar[d]_f & (H,v)\ar[d]^g \\
(E,w) \ar[r] & (E,w)\oplus (H,v)
}$$
commute, où la flèche horizontale inférieure est l'inclusion canonique.
\item  Soient $(A,x)$ et $(B,y)$ deux objets de $\mathbf{H}(\A)$ et $u\in\A(A,B)$. Il existe $(H,v)\in {\rm Ob}\,\mathbf{H}(\A)$ et $\varphi\in\mathbf{H}(\A)((A,x),(B,y)\oplus (H,v))$ dont la composante $A\to B$ est $u$ et dont la composante $A\to H$ est dans $\mathbf{M}(\A)$. De plus, si $\psi\in\mathbf{H}(\A)((A,x),(B,y)\oplus (E,w))$ est un autre morphisme possédant ces propriétés, on peut trouver un morphisme $g\in\mathbf{H}(\A)((H,v),(E,w)\oplus (H,v))$ tel que le diagramme
$$\xymatrix{(A,x)\ar[r]^-\varphi\ar[d]_\psi & (B,y)\oplus (H,v)\ar[d]^{B\oplus g} \\
(B,y)\oplus (E,w)\ar[r] & (B,y)\oplus (E,w)\oplus (H,v)
}$$
de $\mathbf{H}(\A)$ commute, où la flèche horizontale inférieure est l'inclusion canonique.
 \end{enumerate}
\end{lm}

\begin{proof}
 Pour la première assertion, $(H,v)$ est l'espace hyperbolique tel que  $H=A \oplus DA$, et $v$ est l'élément de 
$$T(A\oplus DA)\simeq T(A)\oplus T(DA)\oplus\A(DA,DA)$$
dont les composantes sont nulles dans $T(A)$ et $T(DA)$ et dont la dernière composante est $1_{DA}$. C'est un objet de $\mathbf{H}(\A)$ car l'élément de $\A(A\oplus DA,D(A\oplus DA))\simeq\A(A\oplus DA,DA\oplus A)$ associé a pour matrice $\left(\begin{array}{cc} 0 & 1\\                                                                                                            \epsilon & 0
\end{array}\right)
$. Si $r\in\A(A,DA)$ est un relevé de $x \in T(A)$, alors le morphisme $A\to A\oplus DA$ de $\A$ dont les composantes sont $1_A$ et $r$ appartient à $\mathbf{M}(\A)$, il définit un morphisme $\varphi : (A,x) \to (H,v)$ de $\A^T$.

Supposons maintenant que $f : (A,x) \to (E,w)$ est un autre morphisme de $\A^T$ ayant les propriétés requises. On définit $g\in\A(A\oplus DA,E \oplus A\oplus DA)$ par la matrice
$$\left(\begin{array}{cc} \pi^T(f)-\alpha r & \alpha\\
-\beta r & \beta \\
-r & 1_{DA}
\end{array}\right)
$$
où $\alpha\in\A(DA,E)$ est tel que $\bar{\alpha}s\pi^T(f)=1_A$, $s$ désignant l'élément de $\A(E,DE)$ associé à $w$ (un tel $\alpha$ existe car $s$ est un isomorphisme et $\pi^T(f)$ un monomorphisme scindé) et $\beta\in\A(DA,A)$ est le dual d'un relevé, noté $\gamma \in \A(DA,A)$, de $-T(\alpha)(w) \in T(DA)$ (i.e. $\gamma=\bar{\beta}+\varepsilon\beta$). On vérifie par un calcul facile que $g$ définit un morphisme $(H,v)\to (E,w)\oplus (H,v)$ de $\mathbf{H}(\A)$ ; de plus, il est clair que le diagramme
$$\xymatrix{A\ar[r]^-{\pi^T(\varphi)}\ar[d]_{\pi^T(f)} & A\oplus DA\ar[d]^g \\
E \ar[r] & E \oplus A\oplus DA
}$$
de $\A$ commute, ce qui termine l'établir la première partie du lemme.

Le deuxième point découle du premier, qu'on applique à $A$ muni de l'élément $x-T(u)(y)$ de $T(A)$.
\end{proof}

Combiné au résultat général de la proposition~\ref{polfq-abstr}, le lemme permet facilement d'obtenir  le résultat suivant. On rappelle que la notation $\widetilde{\mathbf{H}(\A)}$ qui y apparaît se réfère à la construction générale associant une catégorie $\widetilde{\M}$ de $\mn$ à une catégorie $\M$ de $\mi$ donnée au début de la section~\ref{section2}, page~\pageref{section2}. 

\begin{pr}\label{eqcat-fpol}
 Le foncteur $\pi^T : \mathbf{H}(\A)\to\A$ induit une équivalence de catégories
$$\Pol_n(\A,\B)\xrightarrow{\simeq}\Pol_n(\widetilde{\mathbf{H}(\A)},\B)$$
pour tout $n\in\mathbb{N}$ et toute catégorie abélienne $\B$.
\end{pr}

\begin{proof}
On établit que les hypothèses de la proposition~\ref{polfq-abstr} sont vérifiées pour le foncteur monoïdal $\Phi=\pi^T : \mathbf{H}(\A)\to\A$ et la sous-catégorie $\C=\mathbf{M}(\A)$ de $\A$. La propriété~\ref{fc} est évidente ; la propriété~\ref{es} résulte de ce que tout objet de $\A$ du type $V\oplus DV$ porte une forme hermitienne non dégénérée (par exemple hyperbolique --- cf. démonstration du lemme~\ref{lm-herm}). La propriété~\ref{val} découle immédiatement de ce que les espaces hermitiens considérés dans $\mathbf{H}(\A)$ sont non dégénérés ; la propriété~\ref{fig} en est également une conséquence classique et facile  \cite[proposition~4.6 ]{Dja-JKT}. Quant à l'hypothèse~\ref{ppf} de la proposition~\ref{polfq-abstr}, elle est satisfaite d'après le lemme~\ref{lm-herm}.
\end{proof}

\begin{rem}
La proposition précédente est une généralisation du théorème $4.5$ obtenu par le second auteur dans  \cite{V-pol} sur la classification des foncteurs polynomiaux sur les espaces quadratiques. En effet, pour $\A=\E^f$ (la catégorie des $\mathbb{F}_2$-espaces vectoriels de dimension finie), on a $ \mathbf{H}^{SD_1^2}(\E^f)=\E_q$ (la catégorie des espaces quadratiques (non dégénérés) de dimension finie sur le corps $\mathbb{F}_2$) et $\widetilde{\mathbf{H}^{SD_1^2}(\E^f)}=\mathcal{T}_q$ (voir l'exemple~\ref{rq-tq}). Ainsi, la proposition précédente fournit la reformulation suivante de \cite{V-pol} : $\Pol_n(\mathcal{T}_q,\B)\xrightarrow{\simeq}\Pol_n(\E^f,\B)$. Alors que la preuve de ce  théorème donnée dans \cite[Theorem $2.6$]{V-pol} s'appuyait sur une classification des foncteurs simples de la catégorie des foncteurs de source $\mathcal{T}_q$, on remplace ici cet argument par une généralisation de \cite[Theorem $2.6$]{V-pol} où on montre que le cran $0$ de la filtration des projectifs standard introduite dans cet article se scinde.

\end{rem}

Nous pouvons maintenant énoncer le résultat principal du présent travail.

\begin{thm}\label{thf}
 Le foncteur  $\pi^T : \mathbf{H}(\A)\to\A$ induit une équivalence de catégories
$$\Pol_n(\A,\B)/\Pol_{n-1}(\A,\B)\xrightarrow{\simeq}\Pol_n(\mathbf{H}(\A),\B)/\Pol_{n-1}(\mathbf{H}(\A),\B)$$
pour tout $n\in\mathbb{N}$ et toute catégorie de Grothendieck $\B$.
\end{thm}

\begin{proof}
 Combiner le théorème~\ref{tilde-pol} et la proposition~\ref{eqcat-fpol}.
\end{proof}

La catégorie $\Pol_n(\A,\B)/\Pol_{n-1}(\A,\B)$ est bien comprise : si $\A$ est la catégorie des modules à gauche projectifs de type fini sur un anneau $A$, par exemple, et $\B$ la catégorie des groupes abéliens, cette catégorie quotient est équi\-va\-lente à la catégorie des modules à droite sur l'algèbre de groupe tordue du groupe symétrique $\Sigma_n$ sur l'anneau $A^{\otimes n}$ (le produit tensoriel étant pris sur $\mathbb{Z}$ ; l'algèbre de groupe est tordue par l'action du groupe symétrique par permutation des facteurs du produit tensoriel) --- cf. \cite{Pira-rec}. Dans le cas général, $\Pol_n(\A,\B)/\Pol_{n-1}(\A,\B)$ est équivalente à la catégorie des $n$-multifoncteurs symétriques sur $\A$ à valeurs dans $\B$ qui sont additifs par rapport à chacune des $n$ variables, l'effet croisé $cr_n$ procurant l'équivalence.

\smallskip

En utilisant également la proposition~\ref{pr-filtr}, on obtient le résultat suivant.

\begin{cor}\label{cor-filtr}
Pour tout $n\in\mathbb{N}$, la catégorie $\Pol_n(\mathbf{H}(\A),\B)$ est la plus petite sous-catégorie pleine de $\st(\mathbf{H}(\A),\B)$ contenant l'image de $\Pol_n(\A,\B)$ par le foncteur $\pi_{\mathbf{H}(\A)}(\pi^T)^*$ de $\Pol_n(\A,\B)$ et stable par noyaux d'épimorphismes et extensions.
\end{cor}

Nous terminons cette section en montrant comment le théorème~\ref{thf} permet de généraliser un résultat sur l'homologie rationnelle des groupes linéaires qui découle des résultats de Scorichenko \cite{Sco} (voir aussi \cite[corollaire~5.11 et remarque~5.12]{Dja-JKT}). 
\begin{thm}\label{applsco}
 Soient $A$ un sous-anneau d'une extension algébrique du corps $\mathbb{Q}$ des nombres rationnels et $F : \mathbf{S}(\mathbf{P}(A))\to\mathbb{Q}\textrm{-}\mathbf{Vect}$ un foncteur polynomial. Alors le morphisme naturel
$$\underset{n\in\mathbb{N}}{\col}H_*(GL_n(A);F(A^n))\to H_*(GL_\infty(A);\mathbb{Q})\underset{\mathbb{Q}}{\otimes}\underset{n\in\mathbb{N}}{\col}\big(H_0(GL_n(A);F(A^n))\big)$$
de $\mathbb{Q}$-espaces vectoriels gradués est un isomorphisme.
\end{thm}

\begin{proof}
 Lorsque $F$ est la composée d'un foncteur polynomial $\mathbf{P}(A)^{op}\times\mathbf{P}(A)\to\mathbb{Q}\textrm{-}\mathbf{Vect}$ avec le foncteur  $\Delta_A:  \mathbf{S}(\mathbf{P}(A))\to\mathbf{P}(A)^{op}\times\mathbf{P}(A)$ d\'efini sur les objets par $\Delta_A(M)=(M,M)$, le résultat est une conséquence du résultat fondamental de Scorichenko \cite{Sco}, repris dans l'article publié \cite{Dja-JKT}.

On note aussi que le résultat est équivalent à la nullité de $H_i( \mathbf{S}(\mathbf{P}(A)) ;F)$ pour $i>0$ \cite[§\,1]{Dja-JKT}. Mais la classe des objets $F$ de $\st( \mathbf{S}(\mathbf{P}(A)),\mathbb{Q}\textrm{-}\mathbf{Vect})$ vérifiant cette propriété d'annulation (on utilise ici la proposition~\ref{pr-homst} pour raisonner sur l'image de $F$ dans cette catégorie quotient) est stable par extensions et par noyaux d'épimorphismes. La conclusion résulte donc du corollaire~\ref{cor-filtr}.
\end{proof}

\bibliographystyle{plain}
\bibliography{bibli-pol.bib}
\end{document}